%% file: approximate-factorization-Jack.tex
\documentclass[12pt,usenames,dvipsnames,svgnames,table]{amsart}

\input{preamble.tex}

\author{Maciej Dołęga}
\address{
Wydział Matematyki i Informatyki, 
Uniwersytet im.~Adama Mickiewicza, 
Collegium Mathematicum,
Umultowska 87, 
61-614 Poznań, 
Poland, \newline \indent Instytut Matematyczny,
Uniwersytet Wrocławski,  \mbox{pl.\ Grunwaldzki~2/4,} 50-384
Wrocław, Poland}
\email{maciej.dolega@amu.edu.pl}

\author{Piotr \'Sniady}
\address{
Institute of Mathematics, Polish Academy of Sciences,
\mbox{ul.~\'Sniadec\-kich 8,} \linebreak 00-956 Warszawa, Poland
} 
\email{psniady@impan.pl}

\title[Jack-deformed random Young diagrams]%
{Gaussian fluctuations \\ of Jack-deformed random Young diagrams}

\begin{document}

\begin{abstract}
We introduce a large class of random Young diagrams 
which can be regarded as a natural one-parameter deformation
of some classical Young diagram ensembles;
a deformation which is related to \emph{Jack polynomials} and \emph{Jack characters}. 
We show that each such a random Young diagram converges asymptotically to some limit shape
and that the fluctuations around the limit are asymptotically Gaussian. 
\end{abstract}



\maketitle

\section{Introduction}

\subsection{Random partitions\dots}
An \emph{integer partition}, called also a \emph{Young diagram}, 
is a weakly decreasing finite sequence $\lambda=(\lambda_1,\dots,\lambda_l)$ 
of positive integers $\lambda_1\geq \cdots \geq \lambda_l>0$.
We also say that $\lambda$ is a partition of $|\lambda|:=\lambda_1+\cdots+\lambda_l$.

\emph{Random partitions} 
occur in mathematics and physics in a wide variety of contexts,
in particular in the Gromov--Witten and Seiberg--Witten theories,
see the overview articles of Okounkov \cite{Okounkov2003} 
and Vershik \cite{VershikICM}. 

\subsection{\dots{}and random matrices}

Certain random partitions can be regarded as discrete counterparts of
some interesting ensembles of random matrices. 
We shall explore this link on a particular example of 
random matrices called 
\emph{$\beta$-ensembles} or \emph{$\beta$-log gases} \cite{Forrester},
i.e.~the probability distributions on $\R^n$ with the density of the form
\[ p(x_1,\dots,x_n) =  \frac{1}{Z} e^{V(x_1)+\cdots+V(x_n)}  \prod_{i<j} |x_i - x_j|^{\beta}, \]
where $V$ is some real-valued function and $Z$ is the normalization constant.
In the special cases $\beta\in\{1,2,4\}$ they describe 
the joint distribution of the eigenvalues of random matrices with natural symmetries;
the investigation of such ensembles for a generic value of $\beta$
is motivated, among others, by statistical mechanics. In this general
case the problem of computing their correlation functions heavily
relies on Jack polynomial theory \cite[Chapter 13]{Forrester}.

\subsection{Random Young diagrams related to log-gases}
Opposite to the special cases $\beta\in\{1,2,4\}$,
in the generic case of $\beta$-ensembles there seems to be no obvious unique 
way of defining their discrete counterparts and several alternative
approaches are available, see the work of Moll \cite{Moll2015}
as well as the work of
Borodin, Gorin and Guionnet \cite{BorodinGorinGuionnet2015}.
In the current paper we took another approach 
based on a deformation of the character theory of the symmetric groups.

The class of random Young diagrams considered in the current paper
as well as the classes from \cite{Moll2015,BorodinGorinGuionnet2015} are of quite distinct
flavors and it is not obvious why they should contain any elements in common,
except for the trivial example given by the Jack--Plancherel measure.
The problem of understanding the relations between these three classes
does not seem to be easy and is out of the scope of the current paper.

\subsection{Random Young diagrams and characters}

The names \emph{integer partitions} and \emph{Young diagrams} are equivalent,
but they are used in different contexts; for this reason we will use two 
symbols $\partitions{n}$ and $\Young{n}$ to denote the same object:
\emph{the set of integer partitions of $n$}, also known as
\emph{the set of Young diagrams with $n$ boxes}.
Any function on the set of partitions (or its some subset) will be referred to 
as \emph{character}.

Suppose that for a given integer $n\geq 0$ 
we are given some convenient family $(\chi_\lambda)$
of functions $\chi_\lambda\colon \partitions{n}\to\R$ which is indexed by 
$\lambda\in\Young{n}$.
We assume that $(\chi_\lambda)$ is a linear basis of the space of 
real functions on $\partitions{n}$ and that for each $\lambda\in\Young{n}$
\[  \chi_\lambda(1^n)= 1, \]
where $1^n=(1,1,\dots,1)$ is a partition of $n$ which consists of $n$ parts, each equal to $1$.
We will refer to the functions from the family $(\chi_\lambda)$ as \emph{irreducible characters}.

Our starting point is some character $\chi\colon\partitions{n}\to\R$ which fulfills
an analogous normalization 
\[  \chi(1^n)= 1. \]
We consider its expansion in the basis of irreducible characters
\begin{equation}
\label{eq:proba-linearcomb}
\chi = \sum_{\lambda\in\Young{n}} \mathbb{P}_\chi(\lambda) \ \chi_\lambda.   
\end{equation}
If the coefficients in this expansion are non-negative numbers, 
they define a probability measure $\mathbb{P}_\chi$ on the set $\Young{n}$
of Young diagrams with $n$ boxes; this probability measure is in the focus of the current paper.

\subsection{Irreducible characters of symmetric groups. Plancherel measure}

The most classical choice of the family $(\chi_\lambda)$ in \eqref{eq:proba-linearcomb}
stems from the representation theory of the symmetric groups.
For a Young diagram $\lambda\in\Young{n}$ with $n$ boxes let $\rho_\lambda\colon \Sym{n} \to M_k(\R)$ denotes 
the corresponding \emph{irreducible representation} \cite{SaganSymmetric} of the symmetric group $\Sym{n}$.
For a permutation $\pi\in \Sym{n}$ we define the value of the \emph{irreducible character $\chi_\lambda$
	of the symmetric group} 
as the fraction of the traces
\begin{equation}
\label{eq:character-irrep}
\chi_\lambda(\pi):= \frac{\Tr \rho_\lambda(\pi)}{\Tr \rho_\lambda(\id)}.   
\end{equation}
Since one can identify a permutation $\pi$ with its cycle decomposition,
it follows that the irreducible character $\chi_\lambda(\pi)$ 
is also well-defined if $\pi\in\partitions{n}$ is a partition of $n$.

For this classical choice of $(\chi_\lambda)$ several results are available.
Firstly, for a specific $\chi=\chi_{\reg}$ given by
\begin{equation}
\label{eq:regular}
\chi_{\reg}(\mu) = \begin{cases} 1 & \text{if $\mu=1^n$},\\ 0 & \text{otherwise,} \end{cases}    
\end{equation}
the corresponding probability measure $\mathbb{P}_{\chi_{\reg}}$ is the celebrated 
\emph{Plancherel measure} 
\cite{LoganShepp1977,VersikKerov1977,BaikDeiftJohansson1999,Kerov1993gaussian}
on the set of Young diagrams with $n$ boxes.
The probability measures $\mathbb{P}_\chi$ for more 
general choices of $\chi$ have been investigated, among others, in
\cite{Biane1998,Biane2001,Sniady2006c}. 

\subsection{Irreducible Jack characters}
In the current paper we will use another, more general, family $(\chi^{(\alpha)}_\lambda)$ 
of irreducible characters in \eqref{eq:proba-linearcomb}.
Our starting point is the family of \emph{Jack polynomials} $J^{(\alpha)}_\lambda$ 
\cite{Jack1970/1971}
which can be regarded as a deformation of the family of Schur polynomials;
a deformation that depends on the parameter~$\alpha>0$.
We use the normalization of Jack polynomials
from \cite[Section VI.10]{Macdonald1995}.

We expand Jack polynomial in the basis of power-sum symmetric functions:
\begin{equation} 
\label{eq:definition-theta-A}
J^{(\alpha)}_\lambda = \sum_\pi \theta^{(\alpha)}_\pi(\lambda)\ p_\pi. 
\end{equation}
The above sum runs over partitions $\pi$ such that $|\pi|=|\lambda|$. 
For a given integer $n\geq 1$, any Young diagram $\lambda\in\Young{n}$ and
any partition $\pi\in\partitions{n}$
we define the \emph{irreducible Jack character $\chi^{(\alpha)}_\lambda$} as  
\begin{equation}
\label{eq:character-Jack-unnormalized-zmiana}
\chi^{(\alpha)}_\lambda(\pi) :=  \alpha^{-\frac{\|\pi\|}{2}}\ \frac{z_\pi}{n!}\ \theta^{(\a)}_\pi(\lambda).
\end{equation}
Above,
\[ \|\pi\|:=|\pi|-\ell(\pi) \]
denotes the \emph{length} of the partition $\pi=(\pi_1,\dots,\pi_l)$, 
while $\ell(\pi)=l$ denotes its \emph{number of parts}.
Also,
the numerical factor $z_\lambda$ is defined by
\begin{equation}
\label{eq:zlambda}
z_\lambda:= \prod_{i\geq 1} m_i(\lambda)!\ i^{m_i(\lambda)},
\end{equation}
where
$m_i(\lambda):=\big|\{k: \lambda_k=i\}\big|$
is the \emph{multiplicity} of $i$ in the partition $\lambda$.

\smallskip

It is worth pointing out that in the special case $\alpha=1$ the corresponding
Jack character $\chi^{(1)}_\lambda(\pi)=\chi_\lambda(\pi)$
coincides with the irreducible character \eqref{eq:character-irrep}
of the symmetric group $\Sym{n}$, see \cite{Lassalle2009, DolegaFeray2014}.

\subsection{Probability measures $\mathbb{P}^{(\alpha)}_\chi$}
With $\chi_\lambda:=\chi_\lambda^{(\alpha)}$ given by the irreducible Jack characters, 
\eqref{eq:proba-linearcomb} takes the following more specific form
\begin{equation}
\label{eq:proba-linearcomb-A}
\chi = \sum_{\lambda\in\Young{n}} \mathbb{P}^{(\alpha)}_\chi(\lambda) \ \chi^{(\alpha)}_\lambda.
\end{equation}
If the coefficients $\mathbb{P}^{(\alpha)}_\chi(\lambda)\geq 0$ are non-negative,
we will say that $\chi$ is a \emph{reducible Jack character}.
The resulting probability measure $\mathbb{P}^{(\alpha)}_\chi$ on $\Young{n}$
is in the focus of the current paper.

\medskip

In the simplest example of $\chi:=\chi_{\reg}$ given by \eqref{eq:regular}
the corresponding probability measure $\mathbb{P}^{(\alpha)}_{\chi_{\reg}}$ turns out to be the 
celebrated \emph{Jack--Plancherel measure} \cite{Stanley1989,Fulman2004, Matsumoto2008, DolegaFeray2014,
	BorodinGorinGuionnet2015, Moll2015}
which is a one-parameter deformation of the Plancherel measure.

\subsection{Random Young diagrams related to Thoma's characters}
An additional motivation for considering this particular
class of random Young diagrams
stems from the research related to the problem of finding extremal characters 
of the infinite symmetric group $\Sym{\infty}$, solved by Thoma \cite{Thoma1964}. 
Vershik and Kerov \cite{VershikKerov1981a}
found an alternative, more conceptual proof of Thoma's result, 
which was based on the observation that 
\emph{characters of $\Sym{\infty}$
	are in a natural bijective correspondence with 
	certain sequences $(\lambda_1\nearrow \lambda_2\nearrow\cdots)$ 
	of growing random Young diagrams}.

The original Thoma's problem can be equivalently formulated as finding
all homomorphisms from the ring of symmetric functions to real numbers 
which are \emph{Schur-positive}, i.e.~which take non-negative values on all Schur polynomials. 
In this formulation the problem naturally asks for generalizations in which Schur polynomials
are replaced by another interesting family of symmetric functions.
Kerov, Okounkov and Olshanski \cite{KerovOkounkovOlshanski1998} considered the particular case of \emph{Jack polynomials}
and they proved that a direct analogue of Thoma's result holds true also in this case.
The main idea behind their proof was that the probabilistic viewpoint from the
above-mentioned work of Vershik and Kerov \cite{VershikKerov1981a} 
can be adapted to
the new setting of Jack polynomials. 
Thus, a side product of the work of Kerov, Okounkov and Olshanski is an interesting, natural class of 
random Young diagrams which fits 
into the framework which we consider in the current paper, see \cref{example:removal-of-boxes}
and the forthcoming paper \cite{DolegaSniady-examples-AFP} for more details.

\subsection{Cumulants}
\label{sec:concatenation}

For partitions $\pi_1,\dots,\pi_k$ we define their \emph{product} 
$\pi_1 \cdots \pi_k$ as their concatenation,
for example $(4,3) \cdot (5,3,1)=(5,4,3,3,1)$. In this way
the set of all partitions $\AllPartitions=\bigsqcup_{n\geq 0} \partitions{n}$ 
becomes a unital semigroup with the unit $1=\emptyset$ corresponding to the empty partition; 
we denote by $\R[\AllPartitions]$
the corresponding semigroup algebra, the elements of which are formal linear combinations
of partitions. Any character $\chi\colon\AllPartitions \to\R$ 
with $\chi(\emptyset)=1$ can be canonically extended 
to a linear map $\chi\colon\R[\AllPartitions]\to\R$ (such that 
$\chi(1)=\chi(\emptyset)=1$) which will be denoted by the same symbol.

For partitions $\pi_1,\dots,\pi_\ell$ we define their \emph{cumulant} with respect to the 
character $\chi\colon \AllPartitions\to\R$ as
a coefficient in the expansion of the logarithm of an analogue of a multidimensional
Laplace transform
\begin{equation}
\label{eq:log-laplace}
\kumupartitions^{\chi}_\ell(\pi_1,\dots,\pi_\ell):= 
\left. \frac{\partial^\ell}{\partial t_1 \cdots \partial t_\ell}
\log \chi\left( e^{t_1 \pi_1+\cdots+t_\ell \pi_\ell} \right) 
\right|_{t_1=\cdots=t_\ell=0},
\end{equation}
where the operations on the right-hand side should be understood 
in the sense of formal power series with coefficients 
either in $\R[\AllPartitions]$ or in $\R$.

In the special case when each partition 
$\pi_i=(l_i)$ consist of just one part we will use a simplified notation and we will write
\[ \kumupartitions^{\chi}_\ell(l_1,\dots,l_\ell) = \kumupartitions^{\chi}_\ell\big( (l_1),\dots,(l_\ell) \big).
\]
For example,
\begin{align*}
\kumupartitions^{\chi}_1(l_1) &= \chi(l_1), \\   
\kumupartitions^{\chi}_2(l_1,l_2) &= \chi(l_1,l_2)- \chi(l_1) \chi(l_2).
\end{align*}

\subsection{Asymptotics}

In the current paper we consider asymptotic problems which correspond to the limit 
when the number of boxes  $n\to\infty $ of the random Young diagrams tends to infinity.
This corresponds to considering a \emph{sequence} $(\chi_n)$ of reducible characters
$\chi_n\colon \partitions{n}\to\R$ and the resulting \emph{sequence}
$(\mathbb{P}^{(\alpha)}_{\chi_n})$ of probability measures on $\Young{n}$.

We also allow that the deformation parameter $\alpha(n)$ depends on $n$;
in order to make the notation light we will 
make this dependence implicit and write shortly $\alpha=\alpha(n)$.

\subsection{Hypothesis: asymptotics of $\alpha$}
\label{sec:hypothesis-double-scaling}
All results of the current paper will be based on the assumption 
that $\alpha=\alpha(n)$ is a sequence of positive numbers such that
\begin{equation} 
\label{eq:double-scaling-refined}
\frac{-\sqrt{\alpha}+\frac{1}{\sqrt{\alpha}}}{\sqrt{n}} = 
\sconstant + \frac{\sconstant'}{\sqrt{n}} + o\left( \frac{1}{\sqrt{n}} \right)
\end{equation}
holds true for $n\to\infty$ for some constants $\sconstant$, $\sconstant'$.

Note that the most important case when $\alpha$ is constant fits into this framework with $\sconstant=0$.
The generic case $\sconstant\neq 0$ will be referred to as 
\emph{double scaling limit}.

\subsection{Hypothesis: approximate factorization of characters}
\label{sec:partitions}

In the following we will use the following convention.
If $\chi_n\colon\partitions{n}\to\R$ is a function on the set of partitions of $n$, 
we will extend its domain to 
the set 
\[ \partitions{\leq n}:=\bigsqcup_{0\leq k\leq n} \partitions{k} \] 
of partitions of smaller numbers 
by setting
\begin{equation}
\label{eq:extended-character}
\chi_n(\pi):= \chi_n( \pi,1^{n-|\pi|}) \qquad \text{for } |\pi|\leq n,    
\end{equation}
i.e.~we extend the partition $\pi$ by adding an appropriate number of parts, each equal to $1$.

In this way the cumulant
$\kumupartitions^{\chi_n}_\ell (l_1,\dots,l_\ell)$ 
is well defined if $n\geq l_1+\cdots+l_\ell$
is large enough.

\begin{definition} 
	\label{def:approx-factorization-charactersA}
	Assume that for each integer $n\geq 1$ we are given a function $\chi_n\colon\partitions{n}\to\R$.
	We say that the sequence $(\chi_n)$ has \emph{approximate factorization property} 
	\cite{Sniady2006c} if
	for each integer $\ell\geq 1$ and all integers $l_1,\dots,l_\ell\geq 2$ the limit
	\begin{equation}
	\label{eq:aprox-fact-property}
	\lim_{n\to\infty} \kumupartitions^{\chi_n}_\ell (l_1,\dots,l_\ell)\ n^{\frac{l_1+\cdots+l_\ell+\ell-2}{2}} 
	\end{equation}
	exists and is finite.
	
	\smallskip
	
	We say that the sequence $(\chi_n)$ has \emph{enhanced approximate factorization property}
	if, additionally, in the case $\ell=1$  
	the rate of convergence in \eqref{eq:aprox-fact-property}
	takes the following explicit form:
	for each $l\geq 2$ there exist some constants $a_{l+1},b_{l+1}\in\R$ such that
	\begin{equation}
	\label{eq:refined-asymptotics-characters}
	\kumupartitions^{\chi_n}_1(l) \ n^{\frac{l-1}{2}} =  \chi_n(l) \ n^{\frac{l-1}{2}}  =
	a_{l+1}+ \frac{b_{l+1}+o(1)}{\sqrt{n}} \qquad \text{for }n\to\infty   
	\end{equation}
	and 
	\begin{equation}
	\label{eq:what-is-m}
	\sup_{l\geq 2} \frac{\sqrt[l]{|a_l|}}{l^m} < \infty,
	\qquad
	\text{where }
	m = \begin{cases}2 &\text{for } g \neq 0,\\
	1 & \text{for } g=0,\end{cases}  
	\end{equation}
	where $g$ is given by \eqref{eq:double-scaling-refined}.
\end{definition}

\begin{example}
	\label{example:plancherel}
	It is easy to check that for $\chi_{\reg}$ from 
	\eqref{eq:regular} all higher cumulants vanish:
	\[ \kumupartitions^{\chi_{\reg}}_\ell (l_1,\dots,l_\ell) = 0 \qquad \text{for $\ell\geq 2$}\]
	and the first cumulant takes a particularly simple form
	\[ \kumupartitions^{\chi_{\reg}}_1 (l) = \chi_{\reg} (l) =
	\begin{cases} 
	1 & \text{if } l=1, \\ 
	0 & \text{otherwise.}
	\end{cases}
	\]
	It follows that that the sequence $(\chi_n)$ for which $\chi_n:=\chi_{\reg}$ fulfills the enhanced approximate
	factorization property.
\end{example}

\subsection{Drawing Young diagrams}
\subsubsection{Anisotropic Young diagrams}
\label{sec:anisotropic-young}

\begin{figure}[tbp]
	
	\begin{center}
		\begin{tikzpicture}[scale=0.8]

		\begin{scope}
		\begin{scope}
		\fill[fill=blue!10] (0,0) -- (4,0) -- (4,1) -- (3,1) -- (3,2) -- (1,2) -- (1,3) -- (0,3) -- cycle;
		\clip (0,0) -- (4,0) -- (4,1) -- (3,1) -- (3,2) -- (1,2) -- (1,3) -- (0,3) -- cycle;
		\draw[dashed] (0,0) grid (5,5);
		\end{scope}
		
		\draw[->,thick] (0,0) -- (6,0) node[anchor=west]{{{$x$}}};
		\foreach \x in {1, 2, 3, 4, 5}
		{ \draw (\x, -2pt) node[anchor=north] {{\tiny{$\x$}}} -- (\x, 2pt); }
		
		\draw[->,thick] (0,0) -- (0,6) node[anchor=south] {{{$y$}}};
		\foreach \y in {1, 2, 3, 4, 5}
		{ \draw (-2pt,\y) node[anchor=east] {{\tiny{$\y$}}} -- (2pt,\y); }

		\draw[ultra thick,draw=blue] (5.5,0) -- (4,0) -- (4,1) -- (3,1) -- (3,2) -- (1,2) -- (1,3) -- (0,3) -- (0,5.5) ;
		
		\end{scope}
		
		\begin{scope}[xshift=250]
		
		\begin{scope}[xscale=0.5,yscale=1.5]
		
		\begin{scope}
		\fill[fill=blue!10] (0,0) -- (4,0) -- (4,1) -- (3,1) -- (3,2) -- (1,2) -- (1,3) -- (0,3) -- cycle;
		\clip (0,0) -- (4,0) -- (4,1) -- (3,1) -- (3,2) -- (1,2) -- (1,3) -- (0,3) -- cycle;
		\draw[dashed] (0,0) grid (5,5);
		\end{scope}
		
		\end{scope}

		\draw[->,thick] (0,0) -- (4,0) node[anchor=west]{{{$x$}}};
		\foreach \x in {1, 2, 3}
		{ \draw (\x, -2pt) node[anchor=north] {{\tiny{$\x$}}} -- (\x, 2pt); }
		
		\draw[->,thick] (0,0) -- (0,6) node[anchor=south] {{{$y$}}};
		\foreach \y in {1, 2, 3, 4, 5}
		{ \draw (-2pt,\y) node[anchor=east] {{\tiny{$\y$}}} -- (2pt,\y); }

		\begin{scope}[xscale=0.5,yscale=1.5]
		
		\draw[ultra thick,draw=blue] (6.5,0) -- (4,0) -- (4,1) -- (3,1) -- (3,2) -- (1,2) -- (1,3) -- (0,3) -- (0,3.5) ;
		
		\end{scope}
		\end{scope}

		\end{tikzpicture}
	\end{center}
	
	\caption{A Young diagram $\lambda=(4,3,1)$ shown in the French convention (left)
		and a generalized Young diagram $T_{\frac{1}{2},\frac{3}{2}} \lambda$ (right)
		obtained by an anisotropic scaling. The dashed lines indicate individual boxes.}
	\label{fig:stretch}
\end{figure}
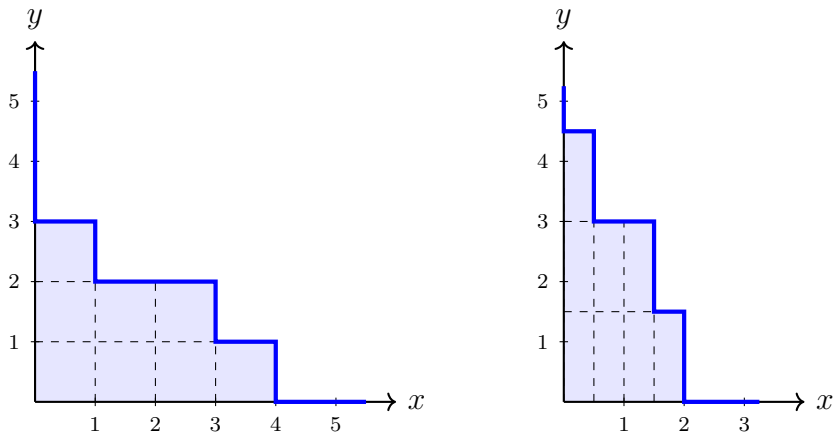

The usual way of drawing Young diagrams is to represent each individual box as a unit square,
see \cref{fig:stretch} (left).
However, when dealing with random Young diagrams related to Jack polynomials
it is more convenient to represent each box as a rectangle with width $w>0$ and height $h>0$ such that
\begin{equation}
\label{eq:aspect-ratio}
\frac{w}{h} = \alpha.   
\end{equation}
A Young diagram viewed like this becomes a polygon contained in the uppper-right quaterplane
which will be denoted by $T_{w,h}\lambda$, see \cref{fig:stretch} (right).
We will refer to such polygons as \emph{anisotropic Young diagrams};
they have been first considered by Kerov \cite{Kerov2000}.

\subsubsection{Russian convention. Profile of a Young diagram}
We draw (a\-ni\-so\-tro\-pic) Young diagrams on the plane 
with the usual Cartesian coordinates $(x,y)$.
However, it is also convenient to use the \emph{Russian} coordinate system 
$(u,v)$ given by
\[ u = x-y, \qquad v = x+y. \]
This new coordinate system gives rise to the \emph{Russian convention} 
for drawing (anisotropic) Young diagrams, see \cref{fig:french-fries}.

The boundary of a Young diagram $\lambda$ 
drawn in the Russian convention
(the solid zigzag line on the right-hand side of \cref{fig:french-fries})
is a graph of a function $\omega_\lambda$ 
which will be called the \emph{profile} of $\lambda$.
If the Young diagram is replaced by an anisotropic
Young diagram $T_{w,h}\lambda$
we define in an analogous way its profile $\omega_{T_{w,h}\lambda}$.

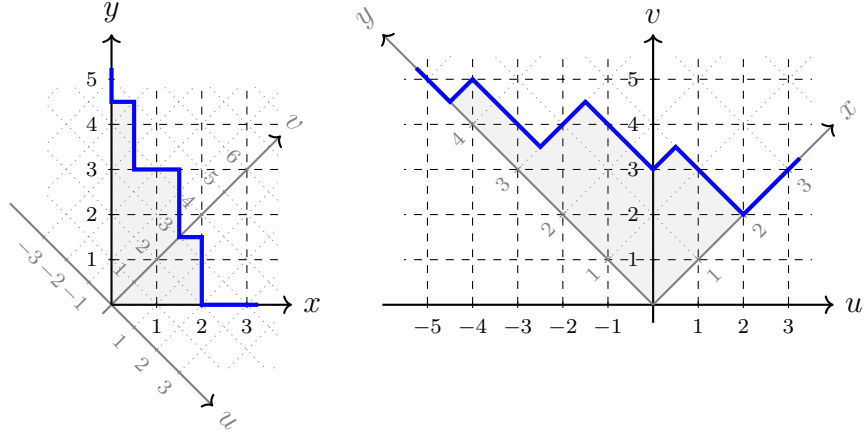
\begin{figure}[tbp]
	
	\begin{center}
		\begin{tikzpicture}[scale=0.6]
		
		\begin{scope}[scale=1/sqrt(2),rotate=-45,draw=gray]
		
		\begin{scope}
		\clip[rotate=45] (-2,-2) rectangle (5.2,6.8);
		\draw[thin, dotted, draw=gray] (-10,0) grid (10,10);
		\begin{scope}[rotate=45,draw=black,scale=sqrt(2)]
		\draw[thin, dashed] (0,0) grid (15,15);
		\end{scope}
		\end{scope}
		
		\draw[->,thick] (-4.5,0) -- (4.5,0) node[anchor=west,rotate=-45]{\textcolor{gray}{$u$}};
		\foreach \z in { -3, -2, -1, 1, 2, 3}
		{ \draw (\z, -2pt) node[anchor=north,rotate=-45] {\textcolor{gray}{\tiny{$\z$}}} -- (\z, 2pt); }
		
		\draw[->,thick] (0,-0.4) -- (0,7.5) node[anchor=south,rotate=-45]{\textcolor{gray}{$v$}};
		
		\foreach \t in {1, 2, 3, 4, 5, 6}
		{ \draw (-2pt,\t) node[anchor=east,rotate=-45] {\textcolor{gray}{\tiny{$\t$}}} -- (2pt,\t); }
		
		\begin{scope}[draw=black,rotate=45,scale=sqrt(2)]
		
		\draw[->,thick] (0,0) -- (4,0) node[anchor=west]{{{$x$}}};
		\foreach \x in {1, 2, 3}
		{ \draw (\x, -2pt) node[anchor=north] {{\tiny{$\x$}}} -- (\x, 2pt); }
		
		\draw[->,thick] (0,0) -- (0,6) node[anchor=south] {{{$y$}}};
		\foreach \y in {1, 2, 3, 4, 5}
		{ \draw (-2pt,\y) node[anchor=east] {{\tiny{$\y$}}} -- (2pt,\y); }
		
		\begin{scope}[xscale=0.5,yscale=1.5] 
		\draw[ultra thick,draw=blue] (6.5,0) -- (4,0) -- (4,1) -- (3,1) -- (3,2) -- (1,2) -- (1,3) -- (0,3) -- (0,3.5) ;
		\fill[fill=gray,opacity=0.1] (4,0) -- (4,1) -- (3,1) -- (3,2) -- (1,2) -- (1,3) -- (0,3) -- (0,0) -- cycle ;
		\end{scope}
		\end{scope}
		
		\end{scope}
		
		\begin{scope}[xshift=12cm, scale=1]
		\begin{scope}
		\clip (-5.5,0) rectangle (3.5,5.5);
		\draw[thin, dashed] (-6,0) grid (6,6);
		\begin{scope}[rotate=45,draw=gray,scale=sqrt(2)]
		\clip (0,0) rectangle (4.5,5.5);
		\draw[thin, dotted] (0,0) grid (6,6);
		\end{scope}
		\end{scope}
		
		\draw[->,thick] (-6,0) -- (4,0) node[anchor=west]{$u$};
		\foreach \z in {-5, -4, -3, -2, -1, 1, 2, 3}
		{ \draw (\z, -2pt) node[anchor=north] {\tiny{$\z$}} -- (\z, 2pt); }
		
		\draw[->,thick] (0,-0.4) -- (0,6) node[anchor=south]{$v$};
		\foreach \t in {1, 2, 3, 4, 5}
		{ \draw (-2pt,\t) node[anchor=east] {\tiny{$\t$}} -- (2pt,\t); }

		\begin{scope}[draw=gray,rotate=45,scale=sqrt(2)]
		
		\draw[->,thick] (0,0) -- (4,0) node[anchor=west,rotate=45] {\textcolor{gray}{{$x$}}};
		\foreach \x in {1, 2, 3}
		{ \draw (\x, -2pt) node[anchor=north,rotate=45] {\textcolor{gray}{\tiny{$\x$}}} -- (\x, 2pt); }
		
		\draw[->,thick] (0,0) -- (0,6) node[anchor=south,rotate=45] {\textcolor{gray}{{$y$}}};
		\foreach \y in {1, 2, 3, 4}
		{ \draw (-2pt,\y) node[anchor=east,rotate=45] {\textcolor{gray}{\tiny{$\y$}}} -- (2pt,\y); }
		
		\begin{scope}[xscale=0.5,yscale=1.5] 
		\draw[ultra thick,draw=blue] (6.5,0) -- (4,0) -- (4,1) -- (3,1) -- (3,2) -- (1,2) -- (1,3) -- (0,3) -- (0,3.5) ;
		\fill[fill=gray,opacity=0.1] (4,0) -- (4,1) -- (3,1) -- (3,2) -- (1,2) -- (1,3) -- (0,3) -- (0,0) -- cycle ;
		\end{scope}
		
		\end{scope}
		
		\end{scope}
		
		\end{tikzpicture}
	\end{center}
	
	\caption{
		The anisotropic Young diagram from \cref{fig:stretch}
		shown in the French and Russian conventions. 
		The solid line represents the \emph{profile} of the Young diagram. 
		The coordinate system $(u,v)$ corresponding to the Russian convention 
		and the coordinate system $(x,y)$ corresponding to the French convention are shown.}
	
	\label{fig:french-fries}
\end{figure}

\subsection{The first main result: Law of Large Numbers}

The following theorem is a generalization of the results of Biane \cite{Biane2001}
who considered the special case $\alpha=1$ and the corresponding representations of the symmetric groups.

\begin{theorem}[Law of large numbers]
	\label{theo:lln}
	Assume that $\alpha=\alpha(n)$ is such that \eqref{eq:double-scaling-refined} holds true.
	Assume that $\chi_n\colon\partitions{n}\to\R$ is a reducible Jack character; 
	we denote by $\lambda_n$ the corresponding random Young diagram with $n$ boxes
	distributed according to $\mathbb{P}^{(\alpha)}_{\chi_n}$.
	We assume also that the sequence $(\chi_n)$ of characters fulfills the enhanced approximate factorization property.
	
	Then there exists some deterministic function $\omega_{\Lambda_\infty}\colon\R\to\R$ 
	with the property that 
	\begin{equation}
	\label{eq:limit}
	\lim_{n\to\infty} \omega_{\Lambda_n} = \omega_{\Lambda_\infty},
	\end{equation}
	where 
	\begin{equation}
	\label{eq:Lambda}
	\Lambda_n:= T_{\sqrt{\frac{\alpha}{n}},\sqrt{\frac{1}{\alpha n}}} \lambda_n    
	\end{equation}
	%
	%
	and the convergence in \eqref{eq:limit} holds true with respect to the supremum norm, in probability.
	In other words, for each $\epsilon>0$
	\[ \lim_{n\to\infty} \mathbb{P}\big( \|\omega_{\Lambda_n}-\omega_{\Lambda_\infty}   \|_\infty>
	\epsilon \big) =0 .\]
\end{theorem}

\begin{remark}
The concrete formula for the profile $\omega_{\Lambda_\infty}$ may be obtained
by computing the corresponding $R$-transform and Cauchy transform, see
for example \cite[Theorem 3]{Biane2001}.
\end{remark}

The proof is postponed to \cref{sec:proof-of-LLN}.

\subsection{The second main result: Central Limit Theorem}
\label{sec:clt-simple}
We keep the notations from \cref{theo:lln}.
The difference 
\begin{equation}
\label{eq:fluctuations-around-infty}
\Delta_n:= \sqrt{n} \left( \omega_{\Lambda_n}- \omega_{\Lambda_\infty} \right).
\end{equation}
is a random function on the real line which quantifies 
the (suitably rescaled) discrepancy
between the shape of the random (anisotropic) Young diagram $\Lambda_n$
and the limit shape.
We will regard ${\Delta}_n$
as a Schwartz  distribution on the real line $\R$ or,
more precisely, as a \emph{random vector} from this space.

The following result is a generalization of Kerov's CLT 
\cite{Kerov1993gaussian,IvanovOlshanski2002} 
which concerned Plancherel measure in the special case $\alpha=1$
as well as a generalization of its extension by the first-named author and F\'eray \cite{DolegaFeray2014}
for the generic \emph{fixed} value of $\alpha>0$.
Indeed, \cref{example:plancherel} shows that the assumptions of the
following theorem are fulfilled for $\chi_n:=\chi_{\reg}$,
thus CLT holds for Jack--Plancherel measure \emph{in a wider generality, when
	$\alpha=\alpha(n)$ may vary with $n$}.

On the other hand the following result is also a generalization of the results of the second-named author
\cite{Sniady2006c} who considered a setup similar to the one below in the special case $\alpha=1$.

\begin{theorem}[Central Limit Theorem]
	\label{theo:clt}
	We keep the assumptions and the notations from \cref{theo:lln}.
	
	Then for $n\to\infty$ the random vector
	${\Delta}_n$ converges in distribution to some (non-centered) Gaussian random vector $\Delta_\infty$
	valued in the space $(\R[x])'$ of distributions, the dual space to polynomials.
\end{theorem}

The above statement about the convergence of the random vector ${\Delta}_n$
	should be understood in a rather specific sense, 
	formulated with help of some suitable test functions. Namely, for each finite collection of polynomials $f_1,\dots,f_k \in \C[x]$ we claim that the joint distribution of the random variables
\[ \langle \Delta_n,f_i \rangle := \int_\R \Delta_n(x)f_i(x) \dif x \qquad \text{for }i\in\{1,\dots,k\} \]
converges to the Gaussian distribution.

Informally speaking: asymptotically, for $n\to\infty$
\[ \omega_{\Lambda_n} \approx \omega_{\Lambda_\infty}+\frac{1}{\sqrt{n}} \Delta_\infty 
\]
where $\omega_{\Lambda_\infty}$ is a deterministic curve and $\Delta_\infty$
is a Gaussian process.

\begin{remark}
	\label{remark:theo:clt} 
	In order to prove \cref{theo:clt} 
	it is enough to show that
	the joint distribution of any finite family of random variables
	$(Y_k)_{k\geq 2}$ converges as $n\to\infty$ to a (non-centered) Gaussian distribution, where
	\begin{equation}
	\label{eq:my-random-variables-2}   
	Y_k:= \frac{k-1}{2} \int  u^{k-2} \ \Delta_n(u) \dif u  
	\end{equation}
	is (up to a simple scalar factor)
	the value of the Schwartz distribution $\Delta_n$ evaluated on a suitable
	polynomial test function.
\end{remark}

The proof is postponed to \cref{sec:proof-of-CLTCLT}.

\subsection{Example}
\label{example:removal-of-boxes}

Let $\alpha>0$ be a fixed positive integer.
For a given integer $i>0$ consider the rectangular Young diagram
\[  (i^{\alpha i}):=(\underbrace{i,\dots,i}_{\text{$\alpha i$ times}})\]
with $n':=\alpha i^2$ boxes. We will assume that $n'$ is an even number.
The special case $\alpha=1$ was considered already by Biane
\cite[Figures 1--3]{Biane1998}. 

Using a random iterative procedure introduced by Kerov
\cite{Kerov1996boundaryYounglattice}
which is an inverse of the Plancherel growth process 
and which will be presented in detail in
a forthcoming paper \cite{DolegaSniady-examples-AFP} 
we remove half of the boxes from the rectangular Young diagram 
$(i^{\alpha i})$; the resulting random Young diagram with $n:=\frac{1}{2} n'$ boxes
will be denoted by $\lambda_n$. We will use the same transformation 
\[T:=T_{\sqrt{\frac{\alpha}{n}},\sqrt{\frac{1}{\alpha n}}}\] in order to scale both 
the original rectangular Young diagram $(i^{\alpha i})$ as well as the resulting
random Young diagram $\lambda_n$.

We notice that the anisotropic Young diagram $T(i^{\alpha i})$ is a square,
see \cref{fig:biane-lln-square};
we shall denote it by $S$. 
As we shall see in \cite{DolegaSniady-examples-AFP}, 
the distribution of the random Young diagram $\lambda_n$ 
can be equivalently formulated via \eqref{eq:proba-linearcomb-A} 
in terms of the corresponding 
natural reducible Jack character and 
\cref{theo:lln} as well as \cref{theo:clt} are applicable. 
Thus the sequence of random anisotropic Young diagrams 
$\Lambda_n=T\lambda_n$ converges to some deterministic limit $\Lambda_\infty$.
Not very surprisingly, this limit $\Lambda_\infty$ turns out to be the bottom half 
of the square $S$, see \cref{fig:biane-lln-square}.

\begin{figure}
	\centering
	\begin{tikzpicture}[scale=3]
	
	\begin{scope}[draw=gray,rotate=45,scale=sqrt(2)]
	
	\draw[fill=green!10,draw=green] (0,0) rectangle (1,1);
	\draw[draw=none,pattern=north west lines, pattern color=blue,opacity=0.3] (0,0) -- (1,0) -- (0,1) -- cycle;
	\draw[blue,dashed,ultra thick] (1.2,0) -- (1,0) -- (0,1) -- (0,1.2);

	\draw[->,thick] (0,0) -- (1.5,0) node[anchor=west,rotate=45] {\textcolor{gray}{{$x$}}};
	\foreach \x in {1}
	{ \draw[thick] (\x, -2pt) node[anchor=north,rotate=45] {\textcolor{gray}{\tiny{$\sqrt{2}$}}} -- (\x, 2pt); }
	
	\draw[->,thick] (0,0) -- (0,1.5) node[anchor=south,rotate=45] {\textcolor{gray}{{$y$}}};
	\foreach \y in {1}
	{ \draw[thick] (-2pt,\y) node[anchor=east,rotate=45] {\textcolor{gray}{\tiny{$\sqrt{2}$}}} -- (2pt,\y); }
	\end{scope}

	\draw[->,thick] (-1.5,0) -- (1.5,0) node[anchor=west]{$u$};
	{ \draw[thick] (-1, -2pt) node[anchor=north] {\tiny{$-\sqrt{2}$}} -- (-1, 2pt); }
	{ \draw[thick] (1, -2pt) node[anchor=north] {\tiny{$\sqrt{2}$}} -- (1, 2pt); }
	
	\draw[->,thick] (0,-0.2) -- (0,2.6) node[anchor=south]{$v$};
	{ \draw[thick] (-2pt,1) node[anchor=south east] {\tiny{$\sqrt{2}$}} -- (2pt,1); }
	{ \draw[thick] (-2pt,2) node[anchor=south east] {\tiny{$2 \sqrt{2}$}} -- (2pt,2); }

	\end{tikzpicture}
	\caption{The green square $S$ depicts the anisotropic Young diagram $ T(i^{\alpha i}) $
		in the Russian coordinate system.
		The blue hatched triangle depicts the limit shape $\Lambda_\infty$;
		the blue dashed line depicts the corresponding profile $\omega_{\Lambda_\infty}$.}
	\label{fig:biane-lln-square}
\end{figure}
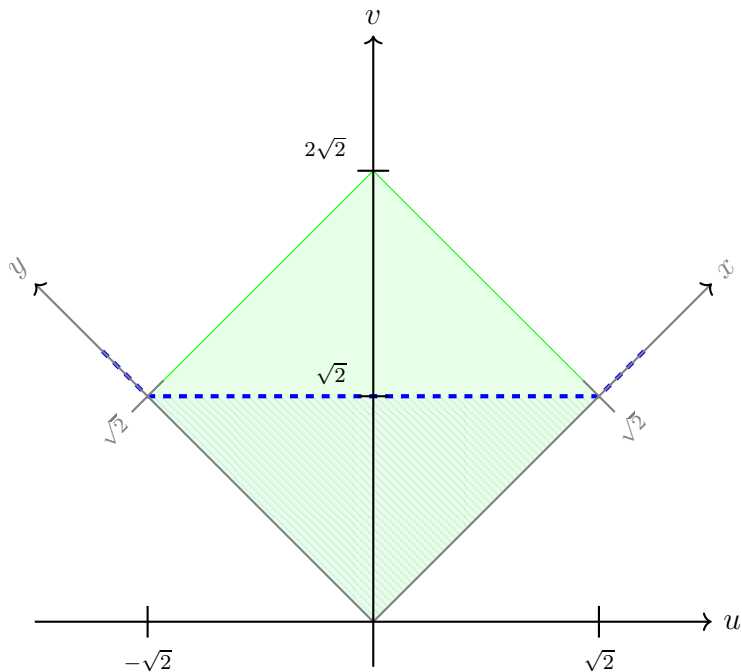

\begin{figure}
	\centering
	\subfile{figures/bok5alpha4-1000prob.tex}
	\caption{Diagonal gray lines form a 
		\emph{heatmap} showing probabilities that a given segment belongs to the profile
		of the random Young diagram $\Lambda_n$ (the intensity of color corresponds to
		the probability). 
		The rectangular grid of anisotropically stretched boxes is clearly visible.
		The blue dashed line depicts the limit profile $\omega_{\Lambda_\infty}$.
		In order to save space, only the neighborhood of $\Lambda_\infty$ is shown.
		The red solid line depicts the mean value $t\mapsto \E \omega_{\Lambda_n}(t)$.
		In this example $\alpha=4$, $i=5$, $n=50$.}
	\label{fig:bok5}
	
	\bigskip
	
	\centering
	\subfile{figures/bok10alpha4-1000prob.tex}
	\caption{The analogue of \cref{fig:bok5} for $\alpha=4$, $i=10$, $n=200$.
		The increased number of boxes is compensated by a decrease of the size of the individual boxes.
		With this choice of scaling, the fluctuations of random Young diagrams $\Lambda_n$ around $\Lambda_\infty$
		tend to zero as $n\to\infty$.}
	\label{fig:bok10}
\end{figure}

\cref{fig:bok5,fig:bok10} are an illustration of the Law of Large Numbers
(\cref{theo:lln}): as the number of boxes $n\to\infty$ tends to infinity,
suitably scaled random Young diagrams $\Lambda_n$ indeed seem to converge
to the deterministic limit $\Lambda_\infty$.

\begin{figure}
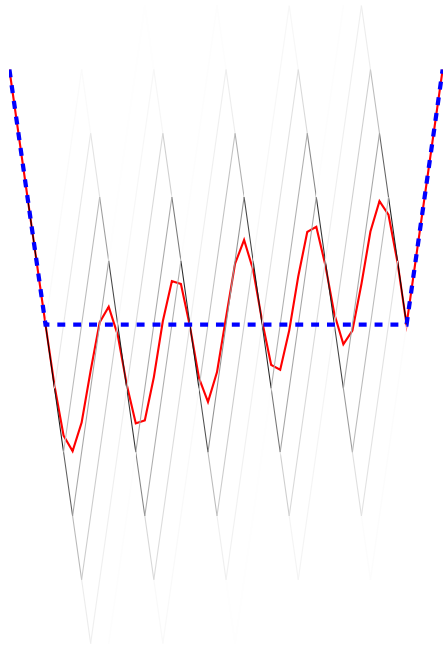

	\centering
	\subfile{figures/bok5alpha4-1000prob-ANISORESCALED.tex}
	\caption{The analogue of \cref{fig:bok5} for the profiles
		$\sqrt{n}\ \omega_{\Lambda_n}$ for which only the second Russian coordinate $v$
		was anisotropically stretched by factor $\sqrt{n}$.
		In this example $\alpha=4$, $i=5$, $n=50$.}
	\label{fig:anizo5}
	
\end{figure}

\begin{figure}
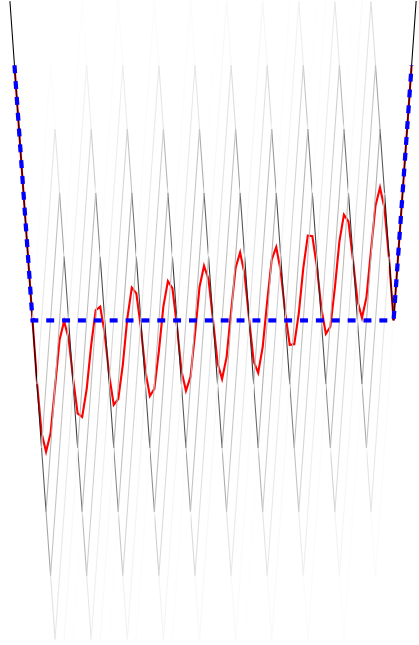

	\centering
	\subfile{figures/bok10alpha4-1000probANISOTROPIC.tex}
	\caption{
		The analogue of \cref{fig:anizo5}
		for $\alpha=4$, $i=10$, $n=200$.
		In this scaling the fluctuations of $\sqrt{n}\ \omega_{\Lambda_n}$ 
		(\emph{``the shaded area''} of the heatmap)
		around $\sqrt{n}\ \omega_{\Lambda_\infty}$ do not vanish as $n\to\infty$.
		Also the discrepancy between the mean value of these fluctuations 
		$\sqrt{n}\ \E \omega_{\Lambda_n}$ (the red solid line)
		and the `first-order approximation' $\sqrt{n}\ \omega_{\Lambda_\infty}$
		(the dashed blue line)
		does not vanish as $n\to\infty$. 
	}
\end{figure}

\begin{figure}
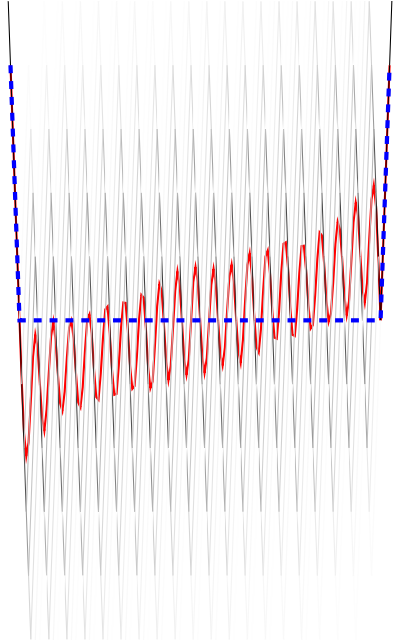

	\centering
	\subfile{figures/bok20alpha4-1000probANISO.tex}
	
	\caption{
		The analogue of \cref{fig:anizo5}
		for $\alpha=4$, $i=20$, $n=800$.
		As $n\to\infty$, the amplitude of the oscillations of the discrepancy 
		$\sqrt{n}\ \left( \E \omega_{\Lambda_n}- \omega_{\Lambda_\infty} \right)$
		remain roughly constant, but their frequency tends to infinity and thus
		the discrepancy converges in the sense of Schwartz distributions.}
	\label{fig:anizo20}
	
\end{figure}

\smallskip

In order to speak about CLT one should consider a more refined scaling:
the one in which one stretches the second Russian coordinate $v$ by a factor of $\sqrt{n}$.
This scaling has a bizarre feature: each individual box is drawn as a parallelogram in 
which the difference 
$v_{\operatorname{max}}-v_{\operatorname{min}}=\sqrt{\alpha}+\frac{1}{\sqrt{\alpha}}$
of the $v$-coordinates of the top and the bottom vertex does not depend on $n$
so one cannot claim that the \emph{size} of an individual box converges to zero;
nevertheless the \emph{area} of an individual box does converge to zero.
Figures~\ref{fig:anizo5}--\ref{fig:anizo20} illustrate this choice of the scaling.

\smallskip

\begin{figure}
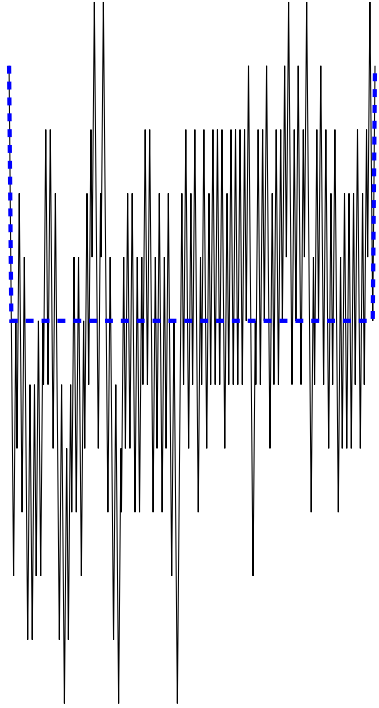

	\centering
	\subfile{figures/single-walk-bok80.tex}
	\caption{A sample profile $\sqrt{n}\ \omega_{\Lambda_n}$ for
		$\alpha=4$, $i=80$, $n=12800$.}
	\label{eq:sample-path}
\end{figure}

\cref{theo:clt} implies in particular that
the limit
\begin{equation}
\label{eq:weak-topology}
\E \Delta_\infty= \lim_{n\to\infty} \E \Delta_n    
\end{equation}
exists as a Schwartz distribution on the real line;
the convergence holds in the weak sense, i.e.~the limit
\[ 
\int u^k\ \E\Delta_\infty(u)  \dif u :=
\lim_{n\to\infty} \int u^k\ \E \Delta_n(u) \dif u 
\]
exists and is finite for an arbitrary integer $k\geq 0$.

The convergence in \eqref{eq:weak-topology} is illustrated in 
Figures~\ref{fig:anizo5}--\ref{fig:anizo20}:
the function $\E\Delta_n$ is the difference between the red solid curve
(i.e.~the plot of $\sqrt{n} \ \E \omega_{\Lambda_n}$)
and the blue dashed curve (i.e.~the plot of $\sqrt{n}\ \omega_{\Lambda_\infty}$). 
As one can see on these examples, the function
$\E\Delta_n$ has oscillations of period and amplitude related to the grid of the boxes
of the Young diagrams. As $n\to\infty$, the amplitude of these oscillations
does not converge to zero (so that the convergence in the supremum norm 
does not hold)
but their frequency tends to infinity (which 
is sufficient for convergence in the weak topology). 

\medskip

The central limit theorem in \cref{theo:clt} 
is somewhat reminiscent to CLT for random walks. A significant difference
lies in the nature of the limit object: in the case of the random walks it is
the Brownian motion which has continuous trajectories while in the case considered
in \cref{theo:clt} it is a random Schwartz distribution $\Delta_\infty$
for which computer simulations (such as the one shown in \cref{eq:sample-path})
suggest that it has  quite singular `trajectories', reminiscent to that of the white noise.
A systematic investigation of such trajectory-wise properties of $\Delta_\infty$
via short-distance asymptotics of the covariance of the corresponding Gaussian field
is out of the scope of the current paper.

\subsection{Content of the paper}
In \cref{sec:approx-fact-characters} we introduce the main algebraic
tool for our considerations, namely \cref{def:approx-factorization-charactersB}
which gives several convenient characterizations of
the approximate factorization property.
In \cref{sec:proof-of-key-tool-short} we prove this result.
\cref{sec:moments-determine-measure} is devoted to some technical results, mostly related
to probability measures which are uniquely determined by their moments.
In \cref{sec:proof-of-LLN} we give the proof of Law of Large Numbers (\cref{theo:lln}).
Finally, in \cref{sec:proof-of-CLTCLT} we give the proof of Central Limit Theorem
(\cref{theo:clt}).

\section{Approximate factorization of characters}
\label{sec:approx-fact-characters}

The purpose of this section is to give a number of 
conditions which are equivalent to the approximate factorization property 
(\cref{def:approx-factorization-charactersA}).
These conditions often turn out to be more convenient in 
applications, such as the ones from \cite{DolegaSniady-examples-AFP}.

\subsection{Conditional cumulants}
\label{sec:conditional-cumulants}

Let $\Alg$ and $\AlgB$ be commutative unital algebras 
and let $\E\colon\Alg\to\AlgB$ be a unital linear map.
We will say that $\E$ is a \emph{conditional expectation value}; 
in the literature one usually imposes some additional constraints on the structure of $\Alg$,
$\AlgB$ and $\E$, 
but for the purposes of the current paper such additional assumptions will not be necessary.

For any tuple $x_1,\dots,x_\ell \in\Alg$ we define their \emph{conditional cumulant} as
\begin{multline*} \condKumu{\Alg}{\AlgB}(x_1,\dots,x_\ell) =  
[t_1 \cdots t_\ell] \log \E e^{t_1 x_1+\dots+t_\ell x_\ell} = \\
\left. \frac{\partial^\ell}{\partial t_1 \cdots \partial t_\ell}
\log \E e^{t_1 x_1+\dots+t_\ell x_\ell} \right|_{t_1=\cdots=t_\ell=0}
\in \AlgB,
\end{multline*}
where the operations on the right-hand side should be understood 
in the sense of formal power series in variables
$t_1,\dots,t_\ell$.

\medskip

Note that the cumulants for partitions which we introduced in \cref{sec:concatenation}
fit into this general framework: for $\Alg:=\R[\AllPartitions]$ one should take
the semigroup algebra of partitions, for $\AlgB:=\R$ the real numbers   
and for $\E:=\chi\colon\R[\AllPartitions]\to\R$ the character.

\subsection{Normalized Jack characters}
The usual way of viewing the characters of the symmetric groups is to fix the 
irreducible representation $\lambda$
and to consider the character as a function of the conjugacy class $\pi$.
However, there is also another very successful viewpoint due to 
Kerov and Olshanski \cite{KerovOlshanski1994}, 
called \emph{dual approach}, which suggests
to do roughly the opposite. Lassalle \cite{Lassalle2008a,Lassalle2009} adapted this idea to
the framework of Jack characters.
In order for this dual approach to be successful 
one has to choose the most convenient normalization constants.
In the current paper we will use the normalization introduced by Dołęga and F\'eray 
\cite{DolegaFeray2014} 
which offers some advantages over the original normalization of Lassalle. 
Thus, with the right choice of the multiplicative constant, the unnormalized Jack character
$\chi^{(\alpha)}_\lambda$ from \eqref{eq:character-Jack-unnormalized-zmiana}
becomes the \emph{normalized Jack character $\Ch^{(\alpha)}_\pi(\lambda)$}, defined as follows.
\begin{definition}
	\label{def:jack-character-classical}
	Let $\alpha>0$ be given and let $\pi$ be a partition. 
	For any Young diagram $\lambda$ the value of 
	the \emph{normalized Jack character $\Ch_\pi^{(\alpha)}(\lambda)$} is given by:
	\begin{equation}
	\label{eq:what-is-jack-character-zmiana}
	\Ch_{\pi}^{(\alpha)}(\lambda):=
	\begin{cases}
	|\lambda|^{\underline{|\pi|}}\ \chi^{(\alpha)}_\lambda(\pi)
	&\text{if }|\lambda| \ge |\pi| ;\\
	0 & \text{if }|\lambda| < |\pi|,
	\end{cases}
	\end{equation}
	where 
	\[ n^{\underline{k}}:=n (n-1) \cdots (n-k+1)\]
	denotes the falling power and $\chi^{(\alpha)}_\lambda(\pi)$ is the Jack character
	\eqref{eq:character-Jack-unnormalized-zmiana}.
	The choice of an empty partition $\pi=\emptyset$ is acceptable; in this case
	$\Ch_\emptyset^{(\alpha)}(\lambda)=1$.
\end{definition}

Each Jack character depends on the deformation parameter $\alpha$; 
in order to keep the notation light we will make this dependence implicit
and we will simply write $\Ch_{\pi}(\lambda)$.

\subsection{The deformation parameters}
In order to avoid dealing with the square root of the variable $\alpha$
which is ubiquitous in the subject of Jack deformation,
we introduce an indeterminate $A := \sqrt{\a}$.
The algebra of Laurent polynomials in the indeterminate $A$ will be denoted by $\Laurent$.

\medskip

A special role will be played by the quantity
\begin{equation}
\label{eq:gamma}
\gamma := -A +\frac{1}{A} \in \Laurent   
\end{equation}
which already appeared in the numerator on the left-hand side of \eqref{eq:double-scaling-refined}.

\subsection{The linear space of $\alpha$-polynomial functions}
\label{sec:polynomial}
In a paper \cite{Sniady-AsymptoticsJack} the second-named author has defined a certain filtered linear space 
of \emph{$\alpha$-po\-ly\-no\-mial functions}.
This linear space consists of certain functions in the set $\AllYoung$
of Young diagrams with values in the ring $\Laurent$ of Laurent polynomials and, among many equivalent definitions, one can define it using normalized Jack characters.

\begin{definition}[{\cite[Proposition 5.2]{Sniady-AsymptoticsJack}}]
	\label{def:polynomial-functions}
	The linear space of $\alpha$-polynomial functions is the linear span (with rational coefficients) of the functions
	\begin{equation}
	\label{eq:linear-span}
	\gamma^k \Ch_\pi \colon \AllYoung \to \Laurent
	\end{equation}
	over the integers $k\geq 0$ and over partitions $\pi\in\AllPartitions$.
	
	The filtration on this vector space is specified as follows:
	for an integer $n\geq 0$ the subspace of vectors of degree at most $n$ 
	is the linear span of the elements \eqref{eq:linear-span} 
	over integers $k\geq 0$ and over partitions $\pi\in\AllPartitions$
	such that 
	\[ k+|\pi|+\ell(\pi)\leq n . \]
\end{definition}

\subsection{Algebras $\Poly$ and $\Poly_\disjoint$ of $\alpha$-polynomial functions}
\label{sec:polynomial-disjoint}

The vector space of $\alpha$-polynomial functions can be equipped with 
a product in two distinct natural ways (which will be reviewed in the following).
With each of these two products it becomes a commutative, unital filtered algebra.

Firstly, as a product we may take the \emph{pointwise product of functions on $\AllYoung$}.
The resulting algebra will be denoted by $\Poly$ (the fact that the $\Poly$ is closed under 
such product was proved by Dołęga and F\'eray
\cite[Theorem 1.4]{DolegaFeray2014}).

Secondly, as a product we may take the \emph{disjoint product $\disjoint$}, see
\cite[Section 2.3]{Sniady-AsymptoticsJack}, which
is defined on the linear base of Jack characters by
\emph{concatenation} (see \cref{sec:concatenation}) of the corresponding partitions
\[ \left( \gamma^p \Ch_\pi\right) \disjoint \left( \gamma^q \Ch_\sigma\right):= \gamma^{p+q} \Ch_{\pi \sigma}.\]
The resulting algebra will be denoted by
$\Poly_\disjoint$.

\subsection{Two probabilistic structures on $\alpha$-polynomial functions}
\label{sec:probabilistic-structures}

Assume that $\chi\colon\partitions{n}\to\R$ is a reducible Jack character 
and let $\mathbb{P}_\chi$ be the corresponding probability measure \eqref{eq:proba-linearcomb-A}
on the set $\Young{n}$ of Young diagrams with $n$ boxes.
With this setup, functions on $\Young{n}$ can be viewed as random variables;
we denote by $\E_\chi$ the corresponding expectation. 

\medskip

Let us fix a partition $\pi\in\partitions{n}$;
we denote by 
$\chi^{(\alpha)}_\# (\pi)$ the random variable 
$\Young{n}\ni\lambda\mapsto \chi^{(\alpha)}_\lambda (\pi)$ given by
irreducible Jack character \eqref{eq:character-Jack-unnormalized-zmiana}.
From the way the probability measure $\mathbb{P}_\chi$ was defined 
in \eqref{eq:proba-linearcomb-A} it follows immediately that
\begin{equation}
\label{eq:mean-value-simple}
\E_\chi\left[ \chi^{(\alpha)}_\# (\pi) \right] =  \chi(\pi).   
\end{equation}

\bigskip

Any $\alpha$-polynomial function $F$ can be restricted to the set $\Young{n}$ of
Young diagrams with $n$ boxes; thus it makes sense to speak about its
expected value $\E_\chi F$. In the case when $F=\Ch_\pi$ is a 
Jack character, this expected value can be explicitly calculated thanks to \eqref{eq:mean-value-simple}:

\begin{equation}
\label{eq:what-is-normalized-character}
\E_\chi \Ch_{\pi}=
\begin{cases}
n^{\underline{|\pi|}}\ \chi(\pi) & \text{if }n < |\pi|, \\
0                                        & \text{otherwise.}
\end{cases}
\end{equation}

\medskip

By considering the multiplicative structure on $\alpha$-polynomial functions
given by the pointwise product, we get in this way a conditional expectation
$\E_\chi\colon \Poly \to \R$;
the corresponding cumulants will be denoted by
$\kappa_{\ell}^{\chi}$.

On the other hand, by considering the disjoint product, we get a conditional expectation
$\E_\chi\colon \Poly_\disjoint \to \R$;
the corresponding cumulants will be denoted by
$\kappa_{\disjoint \ell}^{\chi}$.

\subsection{Equivalent characterizations of approximate factorization of characters}

The following result, \cref{def:approx-factorization-charactersB}, 
is the key tool for the purposes of the current paper.
Its main content is part \ref{item:afp2};
roughly speaking, it states that 
each of the four families of numbers 
\eqref{eq:AFPPartitions}--\eqref{eq:AFPDisjointCumulants}
can be transformed 
into the others.
Each of these four families describes some convenient aspect of the characters $\chi_n$
in the limit $n\to\infty$.
To be more specific:
\begin{itemize}
	\item The family \eqref{eq:AFPCumulants} 
	(and its subset, the family \eqref{eq:AFPCumulants-characters})
	has a direct probabilistic meaning.
	It contains information about the cumulants of some random variables
	which might be handy while proving probabilistic statements such as
	Central Limit Theorem or Law of Large Numbers.
	
	\item On the other hand, the cumulants appearing in the families 
	\eqref{eq:AFPPartitions} and \eqref{eq:AFPDisjointCumulants} are purely algebraic 
	and do \emph{not} have any direct probabilistic meaning.
	However, their merit lies in the fact that in many concrete applications
	(such as the ones from \cite{DolegaSniady-examples-AFP})
	it is much simpler to verify algebraic conditions \ref{item:AFP-A} and 
	\ref{item:disjoint-cumulants} than their probabilistic counterparts
	\ref{item:cumulants-classic-characters} and \ref{item:cumulants-classic-polynomial}.
	
\end{itemize}

\begin{theorem}[The key tool]
	\label{def:approx-factorization-charactersB}
	Assume that $\alpha=\alpha(n)$
	is such as in \cref{sec:hypothesis-double-scaling}.
	Assume also that for each integer $n\geq 1$ we are given a reducible Jack character 
	$\chi_n\colon\partitions{n}\to\R$.
	
	\begin{enumerate}[label=(\emph{\alph*})]
		\setlength\itemsep{2em}
		
		\item \label{item:afp2}
		\emph{Equivalent characterization of 
			approximate factorization property.}
		Then the following four conditions are equivalent: \\
		\begin{enumerate}[label=(\emph{\Alph*})]
			\setlength\itemsep{1em}
			\item \label{item:AFP-A}
			for each integer $\ell \geq 1$ and all integers $l_1,\dots,l_\ell\geq 2$ the limit
			\begin{equation}
			\label{eq:AFPPartitions}
			\lim_{n\to\infty} \kumupartitions^{\chi_n}_\ell (l_1, \dots, l_\ell) \ 
			n^{\frac{l_1+\cdots+l_\ell+\ell-2}{2}} 
			\end{equation}
			exists and is finite;


			\item 
			\label{item:cumulants-classic-characters}
			for each integer $\ell\geq 1$ and all $x_1,\dots,x_\ell\in\{\Ch_1,\Ch_2,\dots\}$ the limit
			\begin{equation} 
			\label{eq:AFPCumulants-characters}
			\lim_{n\to\infty}    \kappa^{\chi_n}_{\ell}(x_{1},\dots,x_{\ell})\ 
			n^{- \frac{\deg x_1+ \cdots + \deg x_\ell - 2(\ell-1)}{2}} 
			\end{equation}
			exists and is finite;

			\item 
			\label{item:cumulants-classic-polynomial}
			for each integer $\ell\geq 1$ and all $x_1,\dots,x_\ell\in\Poly$ the limit
			\begin{equation} 
			\label{eq:AFPCumulants}
			\lim_{n\to\infty}    \kappa^{\chi_n}_{\ell}(x_{1},\dots,x_{\ell})\ 
			n^{- \frac{\deg x_1+ \cdots + \deg x_\ell - 2(\ell-1)}{2}} 
			\end{equation}
			exists and is finite;

			\item 
			\label{item:disjoint-cumulants}
			for each integer $\ell\geq 1$ and all $x_1,\dots,x_\ell\in\Poly_\disjoint$ the limit
			\begin{equation} 
			\label{eq:AFPDisjointCumulants}
			\lim_{n\to\infty}    \kappa^{\chi_n}_{\disjoint \ell}(x_{1},\dots,x_{\ell})\ 
			n^{- \frac{\deg x_1+ \cdots + \deg x_\ell - 2(\ell-1)}{2}} 
			\end{equation}
			exists and is finite.

		\end{enumerate}

		

		
		
		\item 
		\label{item:equivalent-enhanced}
		Assume that the conditions from part  \ref{item:afp2}
		hold true.
		Furthermore, assume that for $\ell=1$ the rate of the convergence of
		\emph{any} of the four expressions under the limit symbol in
		\eqref{eq:AFPPartitions}--\eqref{eq:AFPDisjointCumulants} 
		is of the form 
		\begin{equation}
		\label{eq:nice-limit}
		\operatorname{const}_1+ \frac{\operatorname{const}_2+o(1)}{\sqrt{n}}    
		\end{equation}
		in the limit $n \to \infty$
		and all choices of $l_1$ (respectively, for all choices of $x_1$);
		the constants depend on the choice of $l_1$ (respectively, $x_1$).
		
		Then for $\ell=1$ the rate of convergence of \emph{each} of the four expressions
		\eqref{eq:AFPPartitions}--\eqref{eq:AFPDisjointCumulants} 
		is of the form \eqref{eq:nice-limit}.
	\end{enumerate}

\end{theorem}

When $\alpha=1$, part \ref{item:afp2} of the above result corresponds to \cite[Theorem and Definition 1]{Sniady2006c}. 
The proof is postponed to \cref{sec:proof-of-key-tool-short}.

\section{Proof of \cref{def:approx-factorization-charactersB}}
\label{sec:proof-of-key-tool-short}

In the current section we shall prove the key tool, \cref{def:approx-factorization-charactersB}.

Additionally, concerning part \ref{item:afp2} of \cref{def:approx-factorization-charactersB} 
we shall discuss the exact relationship between the limits 
of the quantities \eqref{eq:AFPPartitions}--\eqref{eq:AFPDisjointCumulants} 
in the case $\ell\in\{1,2\}$. This relationship provides the information about the limit shape
of random Young diagrams in \cref{theo:lln}
as well as about the covariance of the limit Gaussian process describing the fluctuations
in \cref{theo:clt}.

Concerning part \ref{item:equivalent-enhanced} of \cref{def:approx-factorization-charactersB} 
we shall discuss
the exact relationship between the constants which describe the fine asymptotics
\eqref{eq:nice-limit} of the quantities \eqref{eq:AFPPartitions}--\eqref{eq:AFPDisjointCumulants} 
in the case $\ell=1$.
This relationship provides the information about the mean value $\E\Delta_\infty$ 
of the limit Gaussian process from \cref{eq:weak-topology}.

\subsection{Approximate factorization property for $\alpha$-polynomial functions}
\begin{definition}
	Let $\Alg$ and $\AlgB$ be filtered, commutative, unital algebras 
	and let $\E:\Alg\rightarrow\AlgB$ be a 
	unital map.
	We say that $\E$ has \emph{approximate factorization property}
	if for all choices of $x_1,\dots,x_\ell\in\Alg$ we have that 
	\begin{equation}
	\label{eq:approximate-factorization-property}
	\deg_{\AlgB} \condKumu{\Alg}{\AlgB}(x_1,\dots,x_\ell) \leq \left( \deg_{\Alg} x_1 \right) + \cdots + \left( \deg_{\Alg} x_\ell \right)- 2(\ell-1).    
	\end{equation}
\end{definition}

We consider the filtered unital algebras $\Poly_\disjoint$ and
$\Poly$ from \cref{sec:polynomial-disjoint}, 
and as a conditional expectation between them we take the identity map:
\tikzexternaldisable
\begin{equation}
\label{eq:commutative-diagram-A}
\begin{tikzpicture}[node distance=2cm, auto, baseline=-0.8ex]
\node (A) {$\Poly_\disjoint$};
\node (B) [right of=A] {$\Poly.$};

\draw[->] (A) to node {$\id$} (B);
\end{tikzpicture}
\end{equation}
\tikzexternalenable
We denote by $\kumuDisjointPoint:=\kappa_{\Poly_\disjoint}^{\Poly}$ the conditional cumulants
related to the conditional expectation \eqref{eq:commutative-diagram-A}.

\medskip

We are now ready to state the main auxiliary result, 
proved very recently by the second-named author,
which will be necessary for the proof of the key tool, \cref{def:approx-factorization-charactersB}.

\begin{theorem}[{\cite[Theorem 2.3]{Sniady-AsymptoticsJack}}]
	\label{theo:factorization-of-characters}
	The identity map \eqref{eq:commutative-diagram-A}
	%
	has approximate factorization property.
	%
\end{theorem}

\subsection{Approximate factorization for $\alpha=1$}

We denote by $\Poly^{(1)}$ a version of the filtered algebra of $\alpha$-polynomial functions
$\Poly$ obtained by the specialization $\alpha:=1$, $\gamma:=0$; 
analogously we denote by $\Poly^{(1)}_\disjoint$ the algebra $\Poly^{(1)}$
equipped with the multiplication given by the disjoint product.

The following result has been proved earlier by the second-named author.
\begin{theorem}[{\cite[Theorem 15]{Sniady2006c}}]
	\label{theo:aprox-fact-property-oldversion}
	The identity map
	\tikzexternaldisable
	\begin{equation}
	\label{eq:commutative-diagram-Aprim}
	\begin{tikzpicture}[node distance=2.3cm, auto, baseline=-0.8ex]
	\node (A) {$\Poly^{(1)}_\disjoint$};
	\node (B) [right of=A] {$\Poly^{(1)}.$};
	
	\draw[<-,swap] (A) to node {$\id$} (B);
	\end{tikzpicture}
	\end{equation}
	\tikzexternalenable
	has approximate factorization property.
\end{theorem}

Theorems \ref{theo:factorization-of-characters} and \ref{theo:aprox-fact-property-oldversion}
are of the same flavor. There are two major differences between them:
firstly, the arrows in \eqref{eq:commutative-diagram-A} and \eqref{eq:commutative-diagram-Aprim}
point in the opposite directions;
secondly, the algebra $\Poly$ is more rich than its specialized version $\Poly^{(1)}$,
in particular the variable $\gamma\in\Poly$ is not treated like a scalar since 
$\deg \gamma=1 >0$.

\subsection{Proof of \cref{def:approx-factorization-charactersB}, part
	\ref{item:afp2}}
\label{sec:detailed}

This proof follows closely its counterpart from the work of the second-named author
\cite[Theorem and Definition 1]{Sniady2006c}
with the references to \cref{theo:aprox-fact-property-oldversion}
replaced by \cref{theo:factorization-of-characters} and, occasionally,
the roles of $\Poly$ and $\Poly_\disjoint$ reversed.
We present the details below.

\subsubsection{Proof of the equivalence \ref{item:AFP-A}$\iff$\ref{item:disjoint-cumulants}.}

The quantities \eqref{eq:AFPPartitions} and \eqref{eq:AFPDisjointCumulants}
coincide with their counterparts from the work of the second-named author
\cite[Eqs.~(12) and (13)]{Sniady2006c}.
The equivalence of the conditions \ref{item:AFP-A} and \ref{item:disjoint-cumulants}
was proved in \cite[Section 4.7]{Sniady2006c}.

\subsubsection{Proof of the equivalence \ref{item:cumulants-classic-characters}$\iff$\ref{item:cumulants-classic-polynomial}.}
The implication \ref{item:cumulants-classic-polynomial}$\implies$\ref{item:cumulants-classic-characters}
is immediate since $\Ch_1,\Ch_2,\ldots\in\Poly$.

\smallskip

The following result was proved by the second-named author 
\cite[Corollary 19]{Sniady2006c} (note that the original paper 
does not contain the assumption \eqref{eq:lost-assumption}
without which it is not true). The proof 
did not use any specific properties of the filtered algebra $\Poly$
and thus it remains valid also in our context when the original 
\emph{algebra of polynomial functions} is replaced by 
\emph{the algebra of $\alpha$-polynomial functions}.
\begin{lemma}
	\label{lem:Sniady2006-recycled-zombie}
	Assume that $X\subseteq \Poly$ is a set with the property that 
	each $z\in\Poly$ can be expressed as a polynomial
	in the elements of $X$:
	\begin{equation}
	\label{eq:takie-jednomiany-lubie}
	z= \sum_{1\leq i \leq n}  z_{i,1} \cdots z_{i,l_i}    
	\end{equation}
	for some $n\geq 0$ and $z_{i,j}\in X$ in such a way that
	such that for each value of $1\leq i\leq n$ 
	\begin{equation}
	\label{eq:lost-assumption}
	\deg z \geq  \deg z_{i,1}+ \cdots +\deg z_{i,l_i};   
	\end{equation}
	in other words the degree of each monomial should be bounded from above by the degree of $z$.
	
	Under the above assumption, 
	if condition \ref{item:cumulants-classic-characters} holds true for all $x_1,\dots,x_\ell\in X$
	then more general condition \ref{item:cumulants-classic-polynomial} holds true for arbitrary $x_1,\dots,x_\ell\in \Poly$.
\end{lemma}

We claim that the set  $X=\{\gamma,\Ch_1,\Ch_2,\Ch_3,\dots\}$ generates the filtered algebra $\Poly$
in the way specified in \cref{lem:Sniady2006-recycled-zombie}. 
Indeed, by the way the filtration on $\Poly$ was defined (cf.~\cref{def:polynomial-functions})
it is enough to check the assumption of \cref{lem:Sniady2006-recycled-zombie} for
$z=\gamma^k \Ch_{\pi}$ for integer $k\geq 0$ and a partition $\pi=(\pi_1,\dots,\pi_\ell)$;
we will do it by induction over $\deg z=k+|\pi|+\ell$.
We write 
\[ z=  \gamma^k \Ch_{\pi_1}\cdots \Ch_{\pi_\ell} + \left[  \gamma^k \Ch_{\pi} - \gamma^k \Ch_{\pi_1}\cdots \Ch_{\pi_\ell} \right]. \]
The first summand on the right-hand side is of the form which is fits the framework
given by the right-hand side of \eqref{eq:takie-jednomiany-lubie}.
By \cite[Corollary 2.5, Corollary 3.8]{DolegaFeray2014}, the second summand on the right-hand side is of smaller degree,
thus the inductive hypothesis can be applied. This concludes the proof.

We claim that \eqref{eq:AFPCumulants-characters} holds true for all $x_1,\dots,x_\ell\in X$.
Indeed, in the case when $x_1,\dots,x_\ell\in\{\Ch_1,\Ch_2,\dots\}$ this is just the assumption 
\ref{item:cumulants-classic-polynomial}. Consider now the remaining case when $x_j=\gamma$
for some index $j$. For $\ell=1$ the corresponding expression  from  \eqref{eq:AFPCumulants}
is equal to 
\[
\kappa^{\chi_n}_{1}(\gamma)\ 
n^{- \frac{1}{2}} =\frac{\gamma}{\sqrt{n}}
\]
which by \eqref{eq:double-scaling-refined} converges to a finite limit, as required.
For $\ell\geq 2$, the corresponding cumulant 
\[ \kappa^{\chi_n}_{\ell}(\dots,\gamma,\dots)=0\]
vanishes because it involves a deterministc random variable $\gamma$.
It follows immediately that the limit \eqref{eq:AFPCumulants} exists.

In this way we verified that the assumptions of \cref{lem:Sniady2006-recycled-zombie} are fulfilled.
Condition \ref{item:cumulants-classic-polynomial} follows immediately.

\subsubsection{Proof of the equivalence \ref{item:cumulants-classic-polynomial}$\iff$\ref{item:disjoint-cumulants}.}
The proof will follow closely the ideas from \cite[Section 4.7]{Sniady2006c}
with the roles of the cumulants $\kappa^{\chi_n}$ and $\kappa^{\chi_n}_{\disjoint}$ interchanged.
The original proof was based on the observation that the conditional cumulants 
(denoted in the original work \cite{Sniady2006c} by the symbol $k^{\operatorname{id}}=\kappa_{\Poly^{(1)}}^{\Poly_\disjoint^{(1)}}$)
related to the map \eqref{eq:commutative-diagram-Aprim} fulfill the degree bounds \eqref{eq:approximate-factorization-property}
given by the approximate factorization property.
By changing the meaning of the symbol $k^{\operatorname{id}}$ and setting 
$k^{\operatorname{id}}:=\kappa^{\Poly}_{\Poly_\disjoint}=\kumuDisjointPoint$ 
to be the conditional cumulants related to the map \eqref{eq:commutative-diagram-A}
and by applying \cref{theo:factorization-of-characters}
we still have in our more general context that the cumulants $k^{\operatorname{id}}$ fulfill the degree bounds \eqref{eq:approximate-factorization-property}.
The reasoning from \cite[Section 4.7]{Sniady2006c} is still valid in our context.

\subsection{Functionals $\Sfunct_k$}

For a Young diagram $\lambda$, a real number $\alpha>0$ and an integer $k\geq 2$ we define
\[
\Sfunct_k^{(\alpha)}(\lambda):= 
(k-1) \iint_{(x,y)\in\lambda}  \left( \sqrt{\alpha}\ x-\frac{1}{\sqrt{\alpha}}\ y \right)^{k-2}  \dif x \dif y,   
\]
where the integral on the right-hand side
is taken over a polygon on the plane defined by the Young diagram $\lambda$
(drawn in the French convention).
In order to keep the notation light we shall make the dependence on $\alpha$ implicit and
we shall simply write $\Sfunct_k(\lambda):=\Sfunct_k^{(\alpha)}(\lambda)$.

\smallskip

In \cite[Proposition 4.6]{Sniady-AsymptoticsJack} the second-named author proved that 
$\Sfunct_k\in\Poly$ is an $\alpha$-polynomial function 
of degree $k$.

\subsection{Free cumulants $\Rfunct_k$}
\label{sub:FreeCumu}
In many calculations related to 
the asymptotic representation theory it is convenient to parametrize
the shape of the Young diagram~$\lambda$ by \emph{free cumulants}
$\Rfunct_k(\lambda)$ (which depend on the parameter $\alpha$ in our settings). In the context of the representation theory of the symmetric groups
these quantities have been introduced by Biane \cite{Biane1998}.

For the purposes of the current paper it is enough to know that
for a fixed Young diagram $\lambda$ its sequence of functionals of shape and its 
sequence of free cumulants are related to each other by the following simple systems of equations
\cite[Eqs.~(14) and (15)]{DolegaFeraySniady2008}:
\begin{align}
\label{eq:s-r}
\Sfunct_l &= \sum_{i\geq 1} \frac{1}{i!}  (l-1)^{\underline{i-1}} 
\sum_{\substack{k_1,\dots,k_i\geq 2 \\ k_1+\cdots+k_i=l }} \Rfunct_{k_1} \cdots
\Rfunct_{k_i}, \qquad l\geq 2,\\
\label{eq:r-s}
\Rfunct_{l} &=
\sum_{i\geq 1} \frac{1}{i!} (-l+1)^{i-1} \sum_{\substack{k_1,\dots,k_i\geq 2 \\
		k_1+\cdots+k_i=l }} \Sfunct_{k_1} \cdots \Sfunct_{k_i},\qquad  l\geq 2,
\end{align}
where we use a shorthand notation $\Sfunct_k=\Sfunct_k(\lambda)$,
$\Rfunct_k=\Rfunct_k(\lambda)$.
In fact, the above-cited papers \cite{Biane1998,DolegaFeraySniady2008}
concerned only the special isotropic case $\alpha=1$,
however the passage to the anisotropic case $\alpha\neq 1$ does not create any difficulties,
see the work of Lassalle \cite{Lassalle2009} (who used a different normalization constants)
as well as of the first-named author and F\'eray \cite{DolegaFeray2014} 
(whose normalization we use).

\subsection{The case $\alpha=1$.}
\label{sec:alpha-equal-1}
The main advantage of free cumulants lies in the combination of the following two facts.
\begin{itemize}
	\item Each free cumulant $\Rfunct_k$ of a given Young diagram $\lambda$ can be efficiently calculated
	\cite{Biane1998} and its dependence on the shape of $\lambda$ takes a particularly simple
	form (more specifically, ``$\Rfunct_k$ is a \emph{homogeneous} function'').
	
	\item The family $(\gamma,\Rfunct_2,\Rfunct_3,\dots)$ forms a convenient algebraic basis of the algebra $\Poly$
	with $\deg \gamma=1$ and $\deg \Rfunct_k=k$.
	In the special case $\alpha=1$
	(which corresponds to $\gamma=0$) the expansion of $\Ch_l$ in this basis 
	takes the following, particularly
	simple form:
	\begin{equation} 
	\label{eq:Kerov-polynomial-gamma0}
	\Ch_l=
	\Rfunct_{l+1} +
	(\text{terms of degree at most $l-1$}).
	\end{equation}
\end{itemize}

One of the consequences of \eqref{eq:Kerov-polynomial-gamma0} is that 
in the special case $\alpha=1$ the relationship announced in
\cref{def:approx-factorization-charactersB}\ref{item:equivalent-enhanced}
between the 
refined asymptotics of the four quantities \eqref{eq:AFPPartitions}--\eqref{eq:AFPDisjointCumulants} 
for $\ell=1$ takes the following, particularly
simple form. Assume that there exists some sequence $(a_l)$ with the property that
\begin{align}
\label{eq:quick-convergence-1}
\chi_n(l)              &= a_{l+1}\ n^{-\frac{l-1}{2}} + O\left(  n^{-\frac{l+1}{2}} \right) 
\qquad\text{for all $l\geq 1$}; \\
\intertext{note that it is a stronger version of \eqref{eq:nice-limit} with 
	$\operatorname{const}_2\equiv 0$; then}
\label{eq:quick-convergence-2}
\E_{\chi_n}(\Ch_l)     &= a_{l+1}\ n^{ \frac{l+1}{2}} + O\left(  n^{ \frac{l-1}{2}} \right)
\qquad\text{for all $l\geq 1$}; \\
\label{eq:quick-convergence-3}
\E_{\chi_n}(\Rfunct_{l+1})  &= a'_{l+1}\ n^{ \frac{l+1}{2}} + O\left(  n^{ \frac{l-1}{2}} \right)
\qquad\text{for all $l\geq 1$}, \\
\label{eq:quick-convergence-4}
\E_{\chi_n}(\Sfunct_{l})  &= a''_{l}\ n^{ \frac{l}{2}} + O\left(  n^{ \frac{l-2}{2}} \right)
\qquad \text{for each $l\geq 2$,}
\intertext{where}
\label{eq:what-is-prime}
a'_{l+1} &= a_{l+1} 
\intertext{and}
\label{eq:my-little-constants}
a''_l &= \sum_{i\geq 1} \frac{1}{i!}  (l-1)^{\underline{i-1}} 
\sum_{\substack{k_1,\dots,k_i\geq 2 \\ k_1+\cdots+k_i=l }} a'_{k_1} \cdots a'_{k_i},    
\end{align}
see \eqref{eq:s-r} for the last equality and \cite{Sniady2006c} for more details.

\subsection{Details of \cref{def:approx-factorization-charactersB} part \ref{item:afp2}  
	in the generic case $\alpha\neq 1$.}
\label{sec:co-sie-dzieje-alpha-neq-1}
The original proof of \cite[Theorem and Definition 1]{Sniady2006c}
was based on the idea of expressing various elements $F$ of the algebra
of $\alpha$-polynomial functions $\Poly^{(1)}$
(such as the characters $\Ch_\pi$ or the conditional
cumulants $\kumuDisjointPoint$ of such characters)
as polynomials in the basis $\Rfunct_2,\Rfunct_3,\dots$
of free cumulants and studying the top-degree of such polynomials, as we did in \cref{sec:alpha-equal-1}.
In our more general context of $\alpha\neq 1$ (or, in other words, $\gamma\neq 0$)
the corresponding polynomial for $F$
might have some \emph{extra terms} which depend additionally on the variable~$\gamma$. 
These \emph{extra terms} might influence the asymptotic behavior of random Young diagrams.
We shall discuss this issue in more detail in the remaining of this section
as well as in \cref{sec:extra-terms-evil}.

\subsubsection{The first-order asymptotics when $\alpha$ is constant}

The first-named author and F\'eray \cite[Proposition 3.7]{DolegaFeray2014}
proved that the degree of 
\begin{equation}
\label{eq:F1} 
F :=  \Ch_l    
\end{equation}
with respect to the filtration from \cref{sec:alpha-equal-1}
remains equal to $l+1$ even when we pass from $\Poly^{(1)}$ to $\Poly$
or, in other words, that the degree of the \emph{extra terms} is bounded from above by $l+1$.
In the asymptotics when $\alpha$ is a constant and does not depend on $n$
(or, more generally, when the constant $g$ from \eqref{eq:double-scaling-refined} fulfills $g=0$),
it follows that $\gamma=O(1) \ll O\left( n^{\frac{\degg \gamma}{2}}\right) $.
Since each extra term is divisible by the monomial $\gamma$, it follows 
that the contribution of the extra terms
is negligible when compared to the unique original top-degree term $\Rfunct_{l+1}$. 
It follows that in this asymptotics the relationships between the 
quantities $(a_l)$, $(a'_l)$ and $(a''_l)$ which provide the first-order asymptotics of, respectively,
the characters, the mean value of free cumulant, and the mean value of the functionals of shape, remain the same
as in \cref{sec:alpha-equal-1} for the case $\alpha=1$.

\medskip

From \cref{theo:factorization-of-characters} it follows that
an analogous result holds true for the conditional cumulant
(covariance)
\begin{equation}
\label{eq:F2} 
F:= \kumuDisjointPoint(\Ch_{l_1},\Ch_{l_2})
\end{equation}
and the degree of $F$ remains equal to $l_1+l_2$ 
when we pass from $\Poly^{(1)}$ to $\Poly$.
A reasoning similar to the one above implies that
not only the proof of \cite[Theorem and Definition 1]{Sniady2006c}
remains valid in our context for $\ell=2$, but 
\emph{also the relationships between the 
	quantities \eqref{eq:AFPPartitions}--\eqref{eq:AFPDisjointCumulants}
	which provide the first-order asymptotics of cumulants
	remain the same as in \cite[Theorem 3]{{Sniady2006c}}.}

Anticipating the proof of \cref{theo:clt},
the above considerations imply the following explicit description
of the covariance of the limiting Gaussian process $\Delta_\infty$.

\begin{corollary}
	The covariance of the Gaussian process
	$\Delta_\infty$
	describing the limit fluctuations of random Young diagrams from \cref{theo:clt}
	coincides with its counterpart 
	for $\alpha=1$ from
	\cite[Theorem 3]{Sniady2006c}.
\end{corollary}

\subsubsection{The first-order asymptotics in the double scaling limit}
\label{sec:first-order-for-g-neq0}

In the asymptotics when $\alpha=\alpha(n)$ depends on $n$ 
in a way described in \cref{sec:hypothesis-double-scaling} with $g\neq 0$, 
the \emph{extra terms} in both examples \eqref{eq:F1}, \eqref{eq:F2}  
considered above are of the same 
order as the original terms. It follows that 
the relationship between 
the quantities $(a_l)$, $(a'_l)$ and $(a''_l)$ is altered
and depends on the constant $\sconstant$ from \eqref{eq:double-scaling-refined}, see \cref{subsubsect:DoubleScaling} below.
Also the covariance of the Gaussian process
describing the fluctuations of random Young diagrams is altered;  finding an explicit form for this covariance
is currently beyond our reach because no closed formula 
for the top-degree part of
the conditional cumulant $\kumuDisjointPoint(\Ch_{l_1},\Ch_{l_2})$
(an analogue of the results of \cite[Section 1]{Sniady-AsymptoticsJack} for $\Ch_n$)
is available. 

\subsection{Refined asymptotics of characters}
\label{subsec:refinedasymptotics}
In order to find more subtle relationships between the asymptotics of various
quantities appearing in \cref{def:approx-factorization-charactersB}
we need an analogue of Equation \eqref{eq:Kerov-polynomial-gamma0}
between the character $\Ch_l$ and the free cumulants in the generic case $\alpha\neq 1$.
We present below two such formulas: the one from \cref{sec:rough-estimate}
is conceptually simpler and will be sufficient for 
the scaling when $\alpha$ is fixed;
in the case of the double scaling limit we will need a more involved
formula from \cref{sec:top-degree-by-Sniady}.

\subsubsection{The rough estimate}
\label{sec:rough-estimate}

We start with the formula expressing the top-degree part of the
normalized Jack character $\Ch_l$ modulo terms divisible by $\gamma^2$
which follows follows from 
\cite[Section 11]{Lassalle2009} 
combined with the degree bounds of the first-named author and F\'eray
\cite[Proposition 3.7]{DolegaFeray2014}
as well as from \cite[Theorem A.3]{Sniady-AsymptoticsJack}:
\begin{multline}
\label{eq:kerov-polynomial-subtle}
\Ch_l=
\bigg[ \Rfunct_{l+1}+ \\
+\gamma \sum_{i\geq 1} \sum_{k_1+\cdots+k_i=l}
\frac{l}{i} 
(k_1-1) \cdots (k_i-1) \Rfunct_{k_1} \cdots \Rfunct_{k_i}+ \\
{+ (\text{terms divisible by $\gamma^2$ })\bigg]+ }
\\
(\text{terms of degree at most $l-1$}).
\end{multline}

\subsubsection{Closed formula for top-degree part of Jack characters}
\label{sec:top-degree-by-Sniady}
Let us fix an integer $l\geq 1$.
We will view the symmetric group $\Sym{l}$ as the set of permutations of the set
$[l]:=\{1,\dots,l\}$ and its subgroup 
\[\Sym{l-1}:=\big\{ \sigma\in\Sym{l} : \sigma(l)=l \big\}\] 
as the set of permutations of the same set $[l]$ which have $l$ as a fixpoint.
Consider the set 
\begin{multline*}
\mathcal{X}_l= \big\{ (\sigma_1,\sigma_2) \in \Sym{l}\times \Sym{l} : \\
\text{the group generated by $\sigma_1,\sigma_2$ acts transitively on the set $[l]$}
\big\}.   
\end{multline*}
The group $\Sym{l-1}$ acts on $\mathcal{X}_l$ by coordinate-wise conjugation:
\[ \pi\cdot (\sigma_1,\sigma_2) := \big(\pi \sigma_1 \pi^{-1}, \pi \sigma_2 \pi^{-1}\big). \]
The orbits of this action define an equivalence relation $\sim$ on $\mathcal{X}_l$;
the corresponding equivalence classes have a natural combinatorial interpretation as
\emph{non-labeled, rooted, bicolored, oriented maps with $l$ edges} 
which is out of scope of the current paper
(see \cite[Section 1.4.3]{Sniady-AsymptoticsJack} for details).

\medskip

For a permutation $\pi$ we denote by $C(\pi)$ the set of its cycles.

We say that a triple $(\sigma_1,\sigma_2,q)$ is an \emph{`expander'}
\cite[Appendix A.1]{Sniady-AsymptoticsJack}, see also \cite{DolegaFeraySniady2008}, 
if $\sigma_1,\sigma_2\in\Sym{l}$ are permutations
and $q\colon C(\sigma_2)\to\{2,3,\dots\}$ is a function on the set of cycles of $\sigma_2$
with the following two properties:
\[ \sum_{c\in C(\sigma_2)} q(c) = |C(\sigma_1)|+|C(\sigma_2)|\]
and for every set $A\subset C(\sigma_2)$ such that $A\neq \emptyset$ and $A\neq C(\sigma_2)$
we have that
\begin{multline*}
\#\big\{ c\in C(\sigma_1): \text{$c$ intersects at least one of the cycles in $A$} \big\} > \\
\sum_{d\in A} \big[ q(d)-1  \big].
\end{multline*}

The following is a refined version of the formula
\eqref{eq:kerov-polynomial-subtle}.
\begin{lemma}[{\cite[Theorem A.3 and Theorem 1.6]{Sniady-AsymptoticsJack}}]
	\label{lem:top-degree-jack-character}
	For each integer $l\geq 1$ the expansion of the character
	$\Ch_l$ as a polynomial in the variables $\gamma,\Rfunct_2,\Rfunct_3,\dots$ is given by
	\begin{multline}
	\label{eq:Sniady-proved-this}
	\Ch_l= \\
	\sum_{[(\sigma_1,\sigma_2)]\in \mathcal{X}_l/\sim} \gamma^{l+1-|C(\sigma_1)|-|C(\sigma_2)|} 
	\sum_{\substack{q\colon C(\sigma_2)\to\{2,3,\dots\} \\
			\text{$(\sigma_1,\sigma_2,q)$ is an expander}}}
	\prod_{c\in C(\sigma_2)} \Rfunct_{q(c)} + \\
	+\text{(terms of degree at most $l-1$)}.
	\end{multline}
	where the first sum runs over the representatives of the equivalence clases.
\end{lemma}

\subsection{Proof of part \ref{item:equivalent-enhanced} of \cref{def:approx-factorization-charactersB}}
\label{sec:extra-terms-evil}

\subsubsection{The scaling when $\alpha$ is constant}
\label{sec:concrete-alpha-constant}

Part \ref{item:equivalent-enhanced} of \cref{def:approx-factorization-charactersB} concerns the trivial 
case $\ell=1$ which one can easily prove from scratch, based on \eqref{eq:kerov-polynomial-subtle}.
We shall present a detailed proof only for a specific case which will be useful in applications
(more specifically, for the proof of \cref{theo:clt} from \cref{sec:proof-of-CLTCLT}) and we shall 
assume that the refined asymptotics of characters specified in
part \ref{item:equivalent-enhanced} of \cref{def:approx-factorization-charactersB}
holds true for the quantity \eqref{eq:AFPPartitions}. The other implications are
analogous.

For a specific choice of the constants
\begin{align}
\label{eq:quick-convergence-1-B}
\chi_n(l)              &= a_{l+1}\ n^{-\frac{l-1}{2}} + b_{l+1} \ n^{-\frac{l}{2}} + o\left(  n^{-\frac{l}{2}} \right)
\qquad \text{for $l\geq 1$}\\
\intertext{for some sequences $(a_l)$ and $(b_l)$,
	it is a simple exercise to use \eqref{eq:what-is-jack-character-zmiana} and \eqref{eq:kerov-polynomial-subtle}
	in order to show that}
\label{eq:quick-convergence-2-B}
\E_{\chi_n}(\Ch_l)     &= a_{l+1}\ n^{ \frac{l+1}{2}} + b_{l+1}\ n^{\frac{l}{2}}+ o\left(  n^{ \frac{l}{2}} \right)
\qquad \text{for $l\geq 1$,}\\
\label{eq:quick-convergence-3-B}
\E_{\chi_n}(\Rfunct_{l+1})  &= a'_{l+1}\ n^{ \frac{l+1}{2}} + b'_{l+1} \ n^{\frac{l}{2}} + o\left(  n^{ \frac{l}{2}} \right)
\qquad \text{for $l\geq 1$,}
\end{align}
where $(a'_l)$ is given again by \eqref{eq:what-is-prime} and $(b'_l)$ is the unique sequence which fulfills
\begin{equation}
\label{eq:co-sie-dzieje-z-b}
b_{l+1}= b'_{l+1}+\gamma \sum_{i\geq 1} \sum_{k_1+\cdots+k_i=l}
\frac{l}{i} 
(k_1-1) \cdots (k_i-1) a_{k_1} \cdots a_{k_i}.
\end{equation}
In particular, \eqref{eq:quick-convergence-2-B} shows that the
refined asymptotics of characters specified in
part \ref{item:equivalent-enhanced} of \cref{def:approx-factorization-charactersB}
holds true for the quantity \eqref{eq:AFPCumulants-characters}.

\medskip

Consider now the quantity under the limit symbol in \eqref{eq:AFPCumulants}
for $\ell=1$ and for the specific choice of $x_1=\Sfunct_l$.
Equation \eqref{eq:s-r} implies that
the refined asymptotics of characters specified in
part \ref{item:equivalent-enhanced} of \cref{def:approx-factorization-charactersB}
holds true for the quantity \eqref{eq:AFPCumulants}:
\begin{equation}
\label{eq:quick-convergence-4-B}
\E_{\chi_n}(\Sfunct_{l})  = a''_{l}\ n^{ \frac{l}{2}} + b''_{l} \ n^{\frac{l-1}{2}} + o\left(  n^{ \frac{l-1}{2}} \right)
\qquad \text{for each $l\geq 2$,}
\end{equation}
with the constants given by \eqref{eq:my-little-constants} and
\begin{equation}
\label{eq:what-is-b-bis}
b''_l = \sum_{i\geq 1} \frac{1}{(i-1)!}  (l-1)^{\underline{i-1}} 
\sum_{\substack{k_1,\dots,k_i\geq 2 \\ k_1+\cdots+k_i=l }} b'_{k_1} a'_{k_2} \cdots a'_{k_i}.    
\end{equation}

We conclude the proof by pointing out that for $\ell=1$ the expression under the limit symbol in \eqref{eq:AFPCumulants}
coincides with its counterpart from \eqref{eq:AFPDisjointCumulants}.
\qed

\begin{remark}
	One can see that generically for $\alpha\neq 1$ and $\gamma\neq 0$
	(even if the initial characters $\chi_n(l)$ have small subleading terms which corresponds to
	$b_l\equiv 0$) the subleading terms in
	\eqref{eq:quick-convergence-4-B}
	are much bigger than their counterparts for $\alpha=1$ from \eqref{eq:quick-convergence-4},
	namely they are of order $\frac{1}{\sqrt{n}}$ times the leading asymptotic term.
	As we shall see in \cref{sec:proof-of-CLTCLT}, 
	this leads to non-centeredness of the limiting Gaussian process $\Delta_\infty$. 
\end{remark}

\subsubsection{The double scaling limit}
\label{subsubsect:DoubleScaling}

In the double scaling limit $\sconstant\neq 0$ considered in \cref{sec:hypothesis-double-scaling}
the reasoning presented in \cref{sec:concrete-alpha-constant} above remains valid
if one replaces all the references to \eqref{eq:kerov-polynomial-subtle}
by \cref{lem:top-degree-jack-character}.
Note, however, that the relationship \eqref{eq:what-is-prime}
in this new context takes the form
\begin{equation*}
a_{l+1}= \sum_{[(\sigma_1,\sigma_2)]\in \mathcal{X}_l/\sim} g^{l+1-|C(\sigma_1)|-|C(\sigma_2)|} 
\sum_{\substack{q\colon C(\sigma_2)\to\{2,3,\dots\} \\
		\text{$(\sigma_1,\sigma_2,q)$ is an expander}}}
\prod_{c\in C(\sigma_2)} a'_{q(c)},
\end{equation*}
while \eqref{eq:co-sie-dzieje-z-b} takes the form
\begin{multline*} 
b_{l+1}=
\\
\shoveleft{\sum_{[(\sigma_1,\sigma_2)]\in \mathcal{X}_l/\sim} 
	\big({l+1-|C(\sigma_1)|-|C(\sigma_2)|}\big) 
	\ \sconstant' \ \sconstant^{l-|C(\sigma_1)|-|C(\sigma_2)|} \times} 
\\ 
\shoveright{\times
	\sum_{\substack{q\colon C(\sigma_2)\to\{2,3,\dots\} \\
			\text{$(\sigma_1,\sigma_2,q)$ is an expander}}}
	\prod_{c\in C(\sigma_2)} a'_{q(c)}+} \\
\shoveleft{
	+\sum_{[(\sigma_1,\sigma_2)]\in \mathcal{X}_l/\sim} \sconstant^{l+1-|C(\sigma_1)|-|C(\sigma_2)|} 
	\times 
} 
\\ 
\times
\sum_{\substack{q\colon C(\sigma_2)\to\{2,3,\dots\} \\
		\text{$(\sigma_1,\sigma_2,q)$ is an expander}}}
\sum_{c \in C(\sigma_2)} b'_{q(c)}\prod_{c'\in
	C(\sigma_2)\setminus\{c\}} a'_{q(c')}.
\end{multline*}

\section{Technical results}
\label{sec:moments-determine-measure}

This section is devoted to some technical results necessary for the 
proof of \cref{theo:lln}.

\subsection{Slowly growing sequence of moments determines the measure}

\begin{lemma}
	\label{lem:Carleman}
	Assume that $\mu$ is a probability measure which is supported on the interval $[x_0,\infty)$
	(respectively, on the interval $(-\infty,x_0]$) for some $x_0\in\R$ and such
	that
	\begin{equation}
	\label{eq:moments-slow}
	| m_{l} | \leq C^l\ l^{2l}    
	\end{equation}
	holds true for some constant $C$ and all integers $l\geq 1$, where
	\[ m_l = m_l(\mu) = \int_\R x^l \dif\mu \]
	is the $l$-th moment of $\mu$.
	
	Then the measure $\mu$ is uniquely determined by its moments.
	
	\medskip
	
	Similarly, if the measure $\mu$ is supported on the real line $\R$ and such that
	\begin{equation}
	\label{eq:moments-slow'}
	| m_{l} | \leq C^l\ l^{l}    
	\end{equation}
	holds true for some constant $C$ and all integers $l\geq 1$, then the measure $\mu$ is uniquely determined by its moments.
\end{lemma}
\begin{proof}
	In the case when $\mu$ is supported on the interval $[0,\infty)$ this is exactly
	Stieltjes moment problem, while in the case when $\mu$ is supported on
	the real line $\R$ this is exactly the Hamburger moment problem.
	It is easy to check that the assumptions \eqref{eq:moments-slow}, and \eqref{eq:moments-slow'}
	imply that the Carleman's conditions in both Stieltjes and Hamburger,
	respectively, problems are satisified and it follows that the measure $\mu$ is uniquely
	determined by its moments.
	
	Now, assume that $\mu$ is a probability measure which is supported on
	the interval $[x_0,\infty)$. We define a probability measure $\mu_{x_0}$ supported on
	the interval $[0,\infty)$, as a translation of $\mu$ that is,
	for any measurable set $A \subset \R$ we have
	\[ \mu_{x_0}(A) := \mu(A+x_0).\]
	Let us compute the moments of $\mu_{x_0}$:
	\[ m_l(\mu_{x_0}) = \int_\R x^l \dif\mu_{x_0} = \int_\R (x-x_0)^l \dif\mu =
	\sum_{k}\binom{l}{k}(-x_0)^{l-k}\int_\R x^k \dif\mu.\]
	This leads to the following inequalities:
	\begin{multline*} \left|m_l(\mu_{x_0}) \right| \leq \sum_{k}\binom{l}{k}|x_0|^{l-k}\ |m_k|
	\leq \\
	\sum_{k}\binom{l}{k}|x_0|^{l-k}\ C^k\ k^{2k} \leq
	(C+|x_0|)^l\ l^{2l}.
	\end{multline*}
	
	By Carleman's criterion it means that the measure $\mu_{x_0}$ is uniquely
	determined by its moments, which is equivalent by the construction
	that the measure $\mu$ is uniquely
	determined by its moments, too. 
	
	The case, when $\mu$ is supported on the interval $(-\infty,x_0]$ is
	analogous, and we leave it as a simple exercise.
\end{proof}

\subsection{Slow growth of $(\Rfunct_n)$ implies slow growth of $(\Sfunct_n)$}
\begin{lemma}
	\label{lem:small-free}
	Let $(\Sfunct_l)_{l\geq 2}$ be a sequence of real numbers
	and let $(\Rfunct_l)_{l\geq 2}$ given by \eqref{eq:r-s}
	be the corresponding sequence of free cumulants.
	Assume that the sequence of free cumulants fulfills the estimate
	\begin{equation}
	\label{eq:free-are-small}
	| \Rfunct_l | \leq C^l\ l^{ml}    
	\end{equation}
	for some constants $m,C\geq 0$ and all $l\geq 2$.
	
	Then the sequence $(\Sfunct_l)$ fulfills an analogous 
	estimate 
	\begin{equation}
	\label{eq:free-are-smallS}
	| \Sfunct_l | \leq C^l\ l^{ml};    
	\end{equation}
	possibly for another value of the constant $C$.
\end{lemma}

\begin{proof}
	The expansion \eqref{eq:s-r} for $\Sfunct_l$ in terms of the free cumulants gives immediately:
	\begin{multline} 
	\label{eq:ograniczenie}
	|\Sfunct_l| \leq \sum_{i\geq 1} \frac{1}{i!}  (l-1)_{i-1}\  
	C^l\sum_{\substack{k_1,\dots,k_i\geq 2 \\ k_1+\cdots+k_i=l }} k_1^{m k_1} \cdots
	k_i^{m k_i} \\
	\leq C^l l^{ml}\sum_{i\geq 1} \frac{1}{i!}  (l-1)_{i-1}
	\sum_{\substack{k_1,\dots,k_i\geq 2 \\ k_1+\cdots+k_i=l }}1.
	\end{multline}
	Since
	\[ \sum_{\substack{k_1,\dots,k_i\geq 2 \\ k_1+\cdots+k_i=l }}1 =
	\sum_{\substack{k_1,\dots,k_i\geq 1 \\ k_1+\cdots+k_i=l-i
	}}1=\binom{l-1-i}{i-1},\]
	we can bound the sum on the right-hand side of \eqref{eq:ograniczenie} as follows:
	\begin{multline*} 
	\sum_{i\geq 1} \frac{1}{i!}  (l-1)_{i-1}
	\sum_{\substack{k_1,\dots,k_i\geq 2 \\ k_1+\cdots+k_i=l }} 1 \leq
	\sum_{i\geq 1}\binom{l-1}{i-1}\binom{l-1-i}{i-1} \leq \\
	\sum_{i\geq
		0}\binom{l-1}{i}^2 
	\leq \left(\sum_{i\geq 0}\binom{l-1}{i}\right)^2\leq 2^{2l},
	\end{multline*}
	which plugged into \eqref{eq:ograniczenie} yields
	\[ |\Sfunct_l| \leq (4C)^l\ l^{ml},\]
	which finishes the proof.
\end{proof}

\subsection{Estimates on some classes of permutations}

Recall that for a permutation $\pi$ we denote by $C(\pi)$ the set of its cycles.
The length 
\[\|\pi\|:=l - |C(\pi)|\] 
of a permutation $\pi\in\Sym{l}$ is defined as the minimal number of factors
necessary to write $\pi$ as a product of transpositions.
\begin{lemma}
	\label{lem:not-many-permutations}
	For all integers $r\geq 0$ and $l\geq 1$
	\[ \#\Big\{ \pi\in\Sym{l} : \| \pi\| = r \Big\} \leq \frac{l^{2r}}{r!}. \]
\end{lemma}
\begin{proof}
	We claim that for each permutation $\pi\in\Sym{l}$ such that $\|\pi\|=r$ there exist
	at least $r$ transpositions $\tau$ with the property that the permutation 
	$\pi':= \pi \tau$ fulfills $\| \pi'\|=\|\pi\|-1$. Indeed, each such a transposition
	is of the form $\tau=(a,b)$ with $a\neq b$ being elements of the same cycle of
	$\pi$; it follows that the number of such transpositions is equal to
	\[ \sum_{c\in C(\pi)} \binom{|c|}{2} \geq \sum_{c\in C(\pi)} \big( |c|-1  \big) = \|\pi\|. \]
	
	By repeating inductively the same argument for the collection of permutations $(\pi')$ obtained above,
	it follows that the permutation $\pi$ can be written in at least $r!$ 
	different ways as a product of $r$ transpositions. Since there are $\binom{l}{2}<l^2$
	transpositions in $\Sym{l}$, this concludes the proof.
\end{proof}

We revisit \cref{sec:top-degree-by-Sniady}.
In the following we will need a convenient way of parametrizing 
the equivalence classes in $\mathcal{X}_l/\sim$;
for this purpose we note that in each equivalence class one can choose a representative (which is not necessarily unique)
$(\sigma_1,\sigma_2)$ with the property that the permutation $\sigma_2$ has a particularly simple
cycle structure, namely
\[ \sigma_2=(1,2,\dots,i_1)(i_1+1,i_1+2,\dots,i_2) \cdots (i_{\ell-1}+1,i_{\ell-1}+2,\dots,i_\ell)\]
for some increasing sequence $1\leq i_1<i_2<\cdots<i_\ell=l$.
Note that for a fixed $\ell=|C(\sigma_2)|$
\begin{multline} 
\label{eq:takietam}
\text{the number of permutations } \sigma_2 \text{ of the above form}\\
\text{is given by } \binom{l}{\ell-1} \leq l^{\ell-1}.
\end{multline}

\subsection{Growth of free cumulants}
\begin{proposition}
	\label{prop:CumulantsGrowSlow}
	We use the notations and assumptions of \cref{theo:lln}.
	For the random variable $\Rfunct_l=\Rfunct_l(\lambda_n)$ we define
	\begin{equation}
	\label{eq:free-cumulants-exist}
	r_l := \lim_{n\to\infty} {n^{- \frac{l}{2}}}\  \E_{\chi_n} \Rfunct_l.   
	\end{equation}
	
	Then there exists some constant $C$ such that
	\begin{equation}
	\label{eq:free-cumulants-grow-slow}
	|r_l| \leq C^{l}\ l^{ml}    
	\end{equation}
	holds true for each integer $l\geq 2$, where
	$m$ is given by \eqref{eq:what-is-m}.
\end{proposition}
\begin{proof}
	For a given value of $n$ we will investigate the collection of
	random variables 
	\begin{equation}
	\label{eq:my-collection}
	\left( {n^{- \frac{l}{2}}} \Rfunct_l \right)_{l\geq 2}.    
	\end{equation}
	Approximate factorization property, by 
	\cref{def:approx-factorization-charactersB}\ref{item:afp2}\ref{item:cumulants-classic-polynomial}, 
	implies (for $\ell=1$) that
	the limit \eqref{eq:free-cumulants-exist} exists and is finite for
	each integer $l \geq 2$. Furthermore, it implies that each cumulant
	$\kappa_{\ell}$ for $\ell\geq 2$ of the random variables \eqref{eq:my-collection}
	converges to zero as $n\to\infty$.
	
	Each moment (i.e.~the mean value of a product of some random variables)
	can be expressed as a polynomial in the cumulants of the individual random variables
	via \emph{``moment-cumulant formula''}. Thus
	by vanishing of the higher cumulants which correspond to $\ell\geq 2$,
	the expected value of a product of free cumulants approximately factorizes:
	\begin{equation}
	\label{eq:free-cumulants-behave-nice}
	\lim_{n\to\infty} \frac{1}{n^{\frac{k_1+\cdots+k_\ell}{2}}} 
	\ \E_{\chi_n}\left[ \Rfunct_{k_1} \cdots \Rfunct_{k_\ell} \right]  =
	r_{k_1} \cdots r_{k_\ell}.
	\end{equation}
	
	We consider first the case $g=0$, $m=1$.
	Let us divide both sides of \eqref{eq:kerov-polynomial-subtle} by $n^{\frac{l+1}{2}}$
	and take the mean value $\E_{\chi_n}$. 
	By taking the limit $n\to\infty$ and using \eqref{eq:free-cumulants-behave-nice}
	we obtain in this way
	\[ r_{l+1}=a_{l+1} \]
	and the claim follows immediately.
	
	\medskip
	
	\newcommand{\parar}{u}
	\newcommand{\parat}{v}
	
	From the following on we consider the generic case $g\neq 0$ and $m=2$.
	Analogously as above, 
	let us divide both sides of \eqref{eq:Sniady-proved-this} by $n^{\frac{l+1}{2}}$
	and take the mean value $\E_{\chi_n}$. 
	By taking the limit $n\to\infty$ and using \eqref{eq:free-cumulants-behave-nice}
	we obtain in this way
	\begin{equation}
	\label{eq:totoro}
	a_{l+1} = \sum_{[(\sigma_1,\sigma_2)]\in\mathcal{X}_l/\sim} 
	\sconstant^{l+1-|C(\sigma_1)|-|C(\sigma_2)|} 
	\sum_{\substack{q\colon C(\sigma_2)\to\{2,3,\dots\} \\
			\text{$(\sigma_1,\sigma_2,q)$ is an expander}}}
	\prod_{c\in C(\sigma_2)} r_{q(c)}.
	\end{equation}
	We shall cluster the summands according to the parameters $\parar$ and $\parat$ given by
	\[ \parar := \|\sigma_1\| \qquad \text{and} \qquad 
	\parat:=l+1-|C(\sigma_1)|-|C(\sigma_2)|. \]
	
	In the following we will show that 
	from the transitivity requirement in the definition of $\mathcal{X}_l$ it follows that 
	\[ \parat\geq 0. \]
	Indeed, let us construct a bipartite graph $G = (V_\circ \sqcup
	V_\bullet,E)$ with the vertices corresponding to the cycles of
	$\sigma_1$ and the cycles of $\sigma_2$: $V_\circ =
	C(\sigma_1), V_\bullet = C(\sigma_2)$; we connect two vertices 
	by an edge if the corresponding cycles are not disjoint, 
	that is $e = (c_1,c_2)\in E$ if $c_1 \cap c_2 \neq \emptyset$. Then, the
	transitivity of the action of the group generated by $\sigma_1,\sigma_2\in\Sym{l}$ 
	means precisely that
	the graph $G$ is connected. Since the number of edges of $G$ is bounded from above
	by $l$ it follows that the number of vertices of $G$, which is equal to
	$|C(\sigma_1)|+|C(\sigma_2)|$, cannot exceed $l+1$, which gives the required inequality.
	
	Furthermore, the contribution of the terms for which the equality $\parat=0$ holds true
	corresponds to the specialization $g=0$;
	by revisiting \eqref{eq:kerov-polynomial-subtle} it follows that
	this contribution is equal to $r_{l+1}$ which corresponds to
	the unique equivalence class 
	\[ \big\{ (\operatorname{id},\sigma_2) : 
	\sigma_2\in\Sym{l} \text{ is such that $|C(\sigma_2)|=1$} \big\}
	\]
	for which $\parar=0$ and $\parat=0$. It is easy to check that it is the unique
	summand for which $\parar=0$ (since the latter condition is equivalent to
	$\sigma_1=\operatorname{id}$). 
	By singling out this particular summand,
	\eqref{eq:totoro} can be transformed to
	\begin{equation}
	\label{eq:induction-step}
	r_{l+1} = a_{l+1} -
	\sum_{\substack{\parar\geq 1 \\ \parat\geq 1}}
	\sum_{\substack{[(\sigma_1,\sigma_2)]\in\mathcal{X}_l/\sim \\ \|\sigma_1\|=\parar \\ 
			l+1-|C(\sigma_1)|-|C(\sigma_2)|=\parat}} 
	\sconstant^{\parat} 
	\sum_{\substack{q\colon C(\sigma_2)\to\{2,3,\dots\} \\
			\text{$(\sigma_1,\sigma_2,q)$ is an expander}}}
	\prod_{c\in C(\sigma_2)} r_{q(c)}.
	\end{equation}
	
	With the notations used in \cref{eq:induction-step}
	\begin{equation}
	\label{eq:sum-of-qs}
	\parat+ \sum_{c\in C(\sigma_2)} q(c) = \parat+ |C(\sigma_1)|+|C(\sigma_2)|=l+1    
	\end{equation}
	which implies that for each $C>0$ we have that
	\begin{multline}
	\label{eq:induction-step-2}
	\left| \frac{ r_{l+1}}{C^{l+1}\ (l+1)^{2(l+1)}} \right|  \leq 
	\left| \frac{a_{l+1}}{C^{l+1}\ (l+1)^{2(l+1)}} \right| + \\
	\sum_{\substack{\parar\geq 1 \\ \parat\geq 1}} 
	\sum_{\substack{[(\sigma_1,\sigma_2)]\in\mathcal{X}_l/\sim \\ \|\sigma_1\|=\parar \\ l+1-|C(\sigma_1)|-|C(\sigma_2)|=\parat}} 
	\left(\frac{|\sconstant|}{C\ (l+1)^2}\right)^{\parat} \times \\ \times
	\sum_{\substack{q\colon C(\sigma_2) \to\{2,3,\dots\} \\
			\text{$(\sigma_1,\sigma_2,q)$ is an expander}}}
	\prod_{c\in C(\sigma_2)} \left| \frac{r_{q(c)}}{C^{q(c)}\ (l+1)^{2 q(c)}} \right|.
	\end{multline}
	
	By \cref{lem:not-many-permutations} and \eqref{eq:takietam}, 
	for each pair of integers $\parar,\parat\geq 1$ the number of equivalence classes $[(\sigma_1,\sigma_2)]$ which
	could possibly contribute to the above sum is bounded from above by
	\[ \frac{l^{2\parar}}{\parar!}  (l+1)^{|C(\sigma_2)|-1} = \frac{l^{2\parar}}{\parar!}  (l+1)^{\parar-\parat} \leq \frac{(l+1)^{3\parar-\parat}}{\parar!}.\]
	
	For each representative $(\sigma_1,\sigma_2)$ of an equivalence class
	the number of functions $q\colon C(\sigma_2)\to\{2,3,\dots\}$ which fulfill
	\eqref{eq:sum-of-qs} is bounded from above by 
	\[ (l+1)^{|C(\sigma_2)|-1}= (l+1)^{\parar-\parat}.\] 
	
	\medskip

	Our strategy is to prove \eqref{eq:free-cumulants-grow-slow} for $m=2$ by induction over $l\geq 2$.
	It is trivial to check that $r_2 = 1$ always holds true and thus the induction base $l=2$
	is valid if $C\geq 1$.
	We shall assume that that \eqref{eq:free-cumulants-grow-slow} holds true for $2\leq k\leq l$.
	The value of the constant $C$ will be specified at the end of the proof in such a way that
	each induction step can be justified.
	The induction hypothesis implies that
	\begin{equation*}
	\prod_{c\in C(\sigma_2)} \left| r_{q(c)} \right| \leq C^{l+1-\parat}  \prod_{c\in C(\sigma_2)} q(c)^{2q(c)}.
	\end{equation*}
	
	We claim that the following inequality holds true:
	\begin{equation}
	\label{eq:convex}
	\prod_{c\in C(\sigma_2)} q(c)^{2q(c)} \leq 2^{4(\parar-\parat)}\  (l+1)^{2\big( l+1-\parat-2(\parar-\parat)\big)}.  
	\end{equation}
	Indeed, the logarithm of the left-hand side
	is a convex function
	\[ \R_+^{|C(\sigma_2)|} \ni \big( q(c) : c\in C(\sigma_2) \big) \mapsto 2 \sum_{c\in C(\sigma_2)} {q(c)}\log q(c);\]
	its supremum 
	over the simplex given by inequalities $q(c)\geq 2$ and the equality \eqref{eq:sum-of-qs}
	is attained in one of the simplex vertices which corresponds to
	\[ \big( q(c) \big)_{c\in C(\sigma_2)} = \big(l+1-\parat-2(\parar-\parat),\underbrace{2,\dots,2}_{\text{$|C(\sigma_2)|-1=\parar-\parat$ times}} \big);   \]
	this concludes the proof of \eqref{eq:convex}.
	
	In this way we proved that
	\[ \prod_{c\in C(\sigma_2)} \left| \frac{r_{q(c)}}{C^{q(c)} (l+1)^{2 q(c)}} \right| \leq 
	\left( \frac{2}{l+1} \right)^{4(\parar-\parat)}.
	\]
	
	It follows that the right-hand side of \eqref{eq:induction-step-2} is bounded from above by
	\begin{multline*} 
	\left| \frac{a_{l+1}}{C^{l+1}\ (l+1)^{2(l+1)}} \right| + \\
	\shoveright{\sum_{\substack{\parar\geq 1 \\ \parat\geq 1}} 
		\frac{(l+1)^{3\parar-\parat}}{\parar!} (l+1)^{\parar-\parat} \left(\frac{|\sconstant|}{C\ (l+1)^2}\right)^{\parat}
		\left( \frac{2}{l+1} \right)^{4(\parar-\parat)} \leq} \\
	\shoveleft{\left| \frac{a_{l+1}}{C^{l+1}\ (l+1)^{2(l+1)}} \right| + 
		\sum_{\parar\geq 1} 
		\frac{2^{4\parar}}{\parar!} \sum_{\parat\geq 1} \left(\frac{|\sconstant|}{2^4 C}\right)^{\parat} = }\\
	\left| \frac{a_{l+1}}{C^{l+1}\ (l+1)^{2(l+1)}} \right| + 
	(e^{16}-1) \frac{ \frac{|\sconstant|}{2^4 C} }{1 - \frac{|\sconstant|}{2^4 C}}
	\end{multline*}
	for $\frac{|\sconstant|}{2^4 C}<1$.
	The right-hand side tends to zero uniformly over $l$ as $C\to\infty$;
	there exists therefore some $C$ such that the right-hand side is smaller than~$1$.
	Such a choice of $C$ assures that each inductive step is justified.
	This concludes the proof.
\end{proof}

\section{Law of large numbers. Proof of \cref{theo:lln}}
\label{sec:proof-of-LLN}

For Reader's convenience the proof of \cref{theo:lln} was split
into several subsections which consitute the current section.

\subsection{Measure associated with a Young diagram}
\label{subsec:MeasureAssociatedWithYoung}

Suppose an anisotropic Young diagram $\Lambda\subseteq \R\times \R$ 
(viewed in the French coordinate system) is given.
We will assume that the area of $\Lambda$ is equal to $1$. Let $(x,y)\in\Lambda$
be a random point in $\Lambda$, sampled with the uniform probability.
We denote by $P_{\Lambda}$ the probability distribution of its Russian coordinate
\[ u = x-y.\]
It is a probability measure on $\R$ with the probability density
\begin{equation}
\label{eq:density}
f_{\Lambda}(u)=\frac{\omega_\Lambda(u)-|u|}{2}.   
\end{equation}
This density is a Lipschitz function with the Lipschitz constant equal to $1$.
Such probability measures $P_{\Lambda}$ will be our main tool for investigation
of asymptotics of Young diagrams $\Lambda$. The Reader should be advised that
this is \emph{not} Kerov's transition measure
which is also a probability measure on the real line associated with
a Young diagram \cite{Kerov1993transition} for similar purposes. 

Any anisotropic Young diagram $\Lambda$ with unit area which contains some point $(x_0,y_0)$, contains also
the whole rectangle $\{ (x,y) : 0\leq x \leq x_0, 0\leq y\leq y_0 \}$; by comparison of the areas 
it follows that $x_0 y_0\leq 1$.
The latter inequality written in the Russian coordinate system gives the following restriction
on the possible values of the corresponding profile $\omega$:
\[ |u| \leq \omega(u) \leq \sqrt{u^2+4} \] 
and for the corresponding density 
\begin{equation}
\label{eq:bounds-on-density}
0\leq f_\Lambda(u) \leq \frac{\sqrt{u^2+4}-|u|}{2} = \frac{2}{\sqrt{u^2+4}+|u|}.    
\end{equation}

A simple change of variables in the integrals shows that
for a Young diagram $\lambda_n$ with $n$ boxes and the corresponding anisotropic Young diagram
$\Lambda_n$ given by \eqref{eq:Lambda} the moments of the measure $P_{\Lambda_n}$
are given by
\[ \int u^{k} \dif P_{\Lambda_n}(u) = \frac{1}{k+1} \frac{1}{n^{\frac{k+2}{2}}}   \Sfunct_{k+2}(\lambda_n)
\qquad \text{for $k\geq 0$.}
\]

\subsection{Random variables $S^{[n]}_k$ and their convergence in probability.}
We start the proof of \cref{theo:lln}.
For $k\geq 2$ consider the random variable 
\[ S_k^{[n]}:=\frac{1}{n^{\frac{k}{2}}} \Sfunct_k(\lambda_n).\]
By aproximate factorization property,
the condition \ref{item:AFP-A} from \cref{def:approx-factorization-charactersB}
is fulfilled; it follows that the condition \ref{item:cumulants-classic-polynomial}
is fulfilled as well. In the special case $\ell=1$ and $x_1=\Sfunct_k$ it follows that the limit
\begin{equation}
\label{eq:limit-in-LLN}
s_k:=\lim_{n\to\infty} \E_{\chi_n} S^{[n]}_k  = 
\lim_{n\to\infty} \frac{1}{n^{\frac{k}{2}}} \E_{\chi_n}  \Sfunct_k   
\end{equation}
exists; in the special case $\ell=2$ and $x_1=x_2=\Sfunct_k$ it follows that the variance
\[ \operatorname{Var} S^{[n]}_k = 
\kappa_2^{\chi_n}\left( \frac{1}{n^{\frac{k}{2}}} \Sfunct_k,\frac{1}{n^{\frac{k}{2}}} \Sfunct_k  \right)=
O\left( \frac{1}{n} \right) \]
converges to zero. 
Chebyshev's inequality implies that for each $k\geq 2$ the sequence of random variables $(S_k^{[n]})_n$ converges 
(as $n\to\infty$) to $s_k$ in probability. 

\subsection{The limiting probability measure $P_{\Lambda_\infty}$}
\label{sec:limit-measure-Plambdainfty}
By \eqref{eq:bounds-on-density}, for each $n$ the mean value  
\begin{equation}
\label{eq:mean-value}
u \mapsto \E f_{\Lambda_n}(u)   
\end{equation}
exists, is finite, and fulfills analogous bounds to \eqref{eq:bounds-on-density}.
It is the density of the probability measure $\E P_{\Lambda_n}$;
the moments of this measure are given by
\[ \int u^k \dif \E P_{\Lambda_n}(u)= \frac{1}{k+1} \E_{\chi_n} S^{[n]}_{k+2}  .\]
The topology on the set of probability measures (with all moments finite) 
given by convergence of moments can be metrized; we denote by $d$ the corresponding distance.
Equation \eqref{eq:limit-in-LLN} implies that the sequence of measures
$\left( \E P_{\Lambda_n} \right)$ is a Cauchy sequence in the metric space given by $d$; 
this sequence 
converges therefore \emph{in moments} to some probability
measure which will be denoted by $P_{\Lambda_\infty}$,
in particular
\[ \int u^k \dif P_{\Lambda_\infty}(u)=  \frac{1}{k+1} s_{k+2}.\]

In general, it might happen that the measure $P_{\Lambda_\infty}$ is not unique;
it turns out, however, 
that in the setup which we consider 
\emph{the measure $P_{\Lambda_\infty}$ is uniquely
	determined by its moments}; we shall prove it in the following.

\subsection{The measure $P_{\Lambda_\infty}$ is determined by its moments}
By a minor modification of the proof of \eqref{eq:free-cumulants-behave-nice} we have
\begin{equation}
\label{eq:free-cumulants-behave-nice-B}
\lim_{n\to\infty} \frac{1}{n^{\frac{k_1+\cdots+k_\ell}{2}}} 
\ \E_{\chi_n}\left[ \Sfunct_{k_1} \cdots \Sfunct_{k_\ell} \right]  =
s_{k_1} \cdots s_{k_\ell}.
\end{equation}
It follows that the relation between the families of real numbers
$(s_k)_{k\geq 2}$ and $(r_k)_{k \geq 2}$ (given by \eqref{eq:free-cumulants-exist})
is given by an analogue of \eqref{eq:r-s}.

\cref{prop:CumulantsGrowSlow} states that there exists some constant
$C$ such that
\[ |r_k| \leq C^{k} k^{mk}  \]
for all positive integers $k \geq 2$, where  $m$ is given by \eqref{eq:what-is-m}.
Thus \cref{lem:small-free}
gives us the following estimates for the moments of
$P_{\Lambda_\infty}$:
\[ \left| \int u^k \dif P_{\Lambda_\infty}(u) \right| = 
\left| \frac{1}{k+1} s_{k+2} \right| \leq
C'^{k+2} (k+2)^{m(k+2)}\leq C''^kk^{mk}\]
for some constants $C',C''$.

\medskip

We consider first the case $\sconstant=0$.
\cref{lem:Carleman} implies immediately that the measure $P_{\Lambda_\infty}$ is uniquely determined by its moments
which concludes the proof.

\medskip

In the case $\sconstant>0$ the height of each box constituting the anisotropic Young diagram
$\Lambda_n$ is equal to $\sconstant+o(1)>c$ for some constant $c>0$, uniformly over $n$, 
cf.~\cref{sec:hypothesis-double-scaling}. By comparison of the areas it follows that the length $l$ of the 
bottom rectangle constituting $\Lambda_n$ fulfills
$ l c \leq 1 $; in particular it follows that the support of the measure 
$P_{\Lambda_n}$ 
is contained in the interval $\left( -\infty, \frac{1}{c} \right]$.
It follows that an analogous inclusion holds true for the support of the mean value
$\E P_{\Lambda_n}$; by passing to the limit the same is true for $P_{\Lambda_\infty}$. 
It follows that \cref{lem:Carleman} can be applied which concludes the proof.

\medskip

The case $\sconstant<0$ is fully analogous.

\subsection{Weak convergence of probability measures implies uniform convergence of densities}
For $\epsilon>0$ and $u_0\in\R$ let $\phi_\epsilon\colon \R\to\R_+$ be a function 
on the real line such that $\phi_\epsilon$ is supported on an $\epsilon$-neighborhood of $u_0$
and $\int \phi_\epsilon(u) \dif u = 1$. 
Since $\E f_{\Lambda_n}$ is Lipschitz with constant $1$, it follows that
\begin{multline} 
\label{eq:Lipschitz-is-good}
\left| \E f_{\Lambda_n}(u_0) - \int\ \phi_\epsilon(u)\ \dif \E P_{\Lambda_n}(u) \right| =\\
\left| \int \phi_\epsilon(u)\ \left( \E f_{\Lambda_n} (u_0) - \E f_{\Lambda_n}(u)\right)  \dif u \right| \leq 
\\ 
\int \phi_\epsilon(u)\ \left|\E f_{\Lambda_n} (u_0) - \E f_{\Lambda_n}(u)\right|\  \dif u\leq \epsilon.
\end{multline}
By weak convergence of probability measures, the integral on the left-hand side converges, 
as $n\to\infty$, to $\int\ \phi_\epsilon(u)\ \dif P_{\Lambda_\infty}(u)$. By passing to the limit, the above inequality implies
therefore that
\[
\limsup_{n\to\infty}
\left| \E f_{\Lambda_n}(u_0) - \int\ \phi_\epsilon(u)\ \dif P_{\Lambda_\infty}(u) \right| 
\leq \epsilon.
\]
Since this inequality holds true for arbitrary $\epsilon>0$, it follows that the sequence 
$\E f_{\Lambda_n}(u_0)$ is a Cauchy sequence, hence it converges to a finite limit
which will be denoted by $f_{\Lambda_\infty}(u_0)$.
In other words, we have proved that the functions $\E f_{\Lambda_n}$ converge pointwise
to the function $f_{\Lambda_\infty}$; since all functions $\E f_{\Lambda_n}$ are Lipschitz with the same constant,
\emph{the convergence is uniform on a compact set $K=[-R,R]$ for arbitrary value of $R$}.

On the other hand, inequalities \eqref{eq:bounds-on-density} show that
the distance between $\E f_{\Lambda_n}$ and $f_{\Lambda_\infty}$
with respect to the supremum norm on the set $K^c=\R\setminus K$
is bounded by
\begin{equation}
\label{eq:sup-bound}
\left\| \E f_{\Lambda_n} - f_{\Lambda_\infty} \right\|_{L^{\infty}[K^c]} \leq
\frac{2}{\sqrt{R^2+4}+|R|}. 
\end{equation}
Since the right-hand side converges to zero as $R\to\infty$, it follows that
\emph{the sequence of functions $\E f_{\Lambda_n}$ converges to $f_{\Lambda_\infty}$ 
	uniformly on the whole real line $\R$}.

To conclude, we proved the following theorem which might be of independent interest.

\begin{theorem}
Let $(\Lambda_n)$ be a sequence of random anisotropic Young diagrams, each with the unit area.
Let $(P_{\Lambda_n})$ be the corresponding sequence of random probability measures on $\R$ with
densities $(f_{\Lambda_n})$ as
in \cref{subsec:MeasureAssociatedWithYoung}. Assume that
the sequence of probability measures $(\E P_{\Lambda_n})$ converges to some limit in the weak topology. 

Then there exists a
function $f_{\Lambda_\infty}$ such that $\E f_{\Lambda_n} \to
f_{\Lambda_\infty}$ uniformly on $\R$.
\end{theorem}

\subsection{Convergence of densities, in probability.}
We have proved that the sequence of random variables $(S_k^{[n]})_n$ converges 
(as $n\to\infty$) to $s_k$ in probability; 
in other words for each $\epsilon>0$ and each integer $k\geq 0$
\begin{equation} 
\label{eq:moments-convergence}
\lim_{n\to\infty} \mathbb{P}\left( \left| \int u^{k} \dif P_{\Lambda_n}(u)
-  
\int u^{k} \dif P_{\Lambda_\infty}(u)
\right| >
\epsilon \right) =0;
\end{equation} 
in other words the sequence of random probability measures $P_{\Lambda_n}$ converges to the measure
$P_{\Lambda_\infty}$ \emph{in moments, in probability}.

The weak topology of probability measures
can be metrized, for example by Lévy--Prokhorov distance $\pi$.
Since the measure $P_{\Lambda_\infty}$ is uniquely determined by its moments, 
convergence to $P_{\Lambda_\infty}$ \emph{in moments} implies convergence to the same limit
\emph{in the weak topology of probability measures}.
With the help of the distances $d$ (cf.~\cref{sec:limit-measure-Plambdainfty})
and $\pi$, the latter statement can be rephrased as follows:
for each $\varepsilon>0$
there exists $\delta>0$ such that for any probability measure $\mu$ 
\[  d(\mu, P_{\Lambda_\infty}) < \delta \implies 
\pi(\mu, P_{\Lambda_\infty}) < \varepsilon.
\]
It follows
that the sequence of random probability measures $P_{\Lambda_n}$ converges to the measure
$P_{\Lambda_\infty}$ \emph{in the weak topology of probability measures, in probability}, 
i.e.~for each $\varepsilon>0$
\begin{equation}
\label{eq:Levy}
\lim_{n\to\infty} \mathbb{P}\left( \pi(P_{\Lambda_n}, P_{\Lambda_\infty})  >
\varepsilon \right) =0.  
\end{equation}

\medskip

Let $\epsilon>0$; by adapting the proof of \eqref{eq:Lipschitz-is-good}, we get that
\begin{equation}
\label{eq:weak-conv-implications1}
\left| f_{\Lambda_n}(u_0) - \int\ \phi_\epsilon(u)\ \dif  P_{\Lambda_n}(u) \right| 
\leq \epsilon.  
\end{equation}
Lévy--Prokhorov distance $\pi$ metrizes the weak convergence of probability measures;
it follows that we can choose sufficiently small $\varepsilon>0$ 
with the property that for any probability measure
$\mu$ on $\R$
\begin{equation}
\label{eq:weak-conv-implications2}
\pi(\mu, P_{\Lambda_\infty})\leq \varepsilon \implies 
\left| \int \phi_\epsilon(u) \dif \mu(u)-  \int \phi_\epsilon(u) \dif P_{\Lambda_\infty}(u) \right| <\epsilon.
\end{equation}
Equation \eqref{eq:Levy} combined with \eqref{eq:weak-conv-implications2}
as well as \eqref{eq:weak-conv-implications1}
imply therefore that
\[ \lim_{n\to\infty} 
\mathbb{P} \left( \left| f_{\Lambda_n}(u_0) - \int \phi_\epsilon(u) \dif P_{\Lambda_\infty}(u) \right| > 2\epsilon
\right) = 0.
\]
As $\epsilon\searrow 0$, the integral $\int \phi_\epsilon(u) \dif P_{\Lambda_\infty}(u) $ 
converges to the density $f_{\Lambda_\infty}(u_0)$ by an analogue of \eqref{eq:Lipschitz-is-good}.
In this way we proved that \emph{for each $u_0$ 
	the sequence $f_{\Lambda_n}(u_0)$ converges to
	$f_{\Lambda_\infty}(u_0)$ in probability}.

By the same type of argument as in \eqref{eq:sup-bound}
it follows that 
\emph{the sequence of functions $f_{\Lambda_n}$ converges uniformly to $f_{\Lambda_\infty}$ in probability}, 
i.e.~for each $\epsilon>0$
\[\lim_{n\to\infty} \mathbb{P} \left( \left\| f_{\Lambda_n} - f_{\Lambda_\infty} \right\|_{\infty} > \epsilon \right) = 0 \]
with respect to the supremum norm.

\subsection{Back to Young diagrams.}
The function $\omega_{\Lambda_\infty}$ which was promised in the formulation of \cref{theo:lln}
is simply given by the relationship \eqref{eq:density} for the specific choice of 
$f_{\Lambda}:=f_{\Lambda_\infty}$, namely
\[ \omega_{\Lambda_\infty}(u):= 2 f_{\Lambda_\infty}(u) + |u|.\]

We just finished the proof of
the fact that $\omega_{\Lambda_n}$ converges to
$\omega_{\Lambda_\infty}$ in the supremum norm as $n
\to \infty$, in probability, thus the proof of \cref{theo:lln} is completed.

\section{Central Limit Theorem. Proof of \cref{theo:clt}} 
\label{sec:proof-of-CLTCLT}

\begin{proof}[Proof of \cref{theo:clt}]
	
	In the light of \cref{remark:theo:clt} we shall investigate the cumulants of the form
	\[ \kappa_\ell \left(Y_{i_1}, \dots, Y_{i_\ell}  \right)\]
	for the random variables $(Y_k)$ given by \eqref{eq:my-random-variables-2}.
	
	\bigskip

	We start with the case $\ell\geq 2$. 
	By \eqref{eq:fluctuations-around-infty}, each $Y_k$ is equal (up to a deterministic shift) to
	the random variable
	\[ X_k:= 
	\sqrt{n}\ \frac{k-1}{2} \int  u^{k-2} \ \omega_{\Lambda_n}(u) \dif u =
	n^{-\frac{k-1}{2}} \Sfunct_k(\lambda_n). 
	\]
	For $\ell\geq 2$ the cumulant $\kappa_\ell$ is translation-invariant;
	it follows therefore that
	\begin{equation}
	\label{eq:cumulant-concrete}
	\kappa_\ell \left( Y_{i_1},\dots, Y_{i_\ell} \right)=
	\kappa_\ell \left( X_{i_1},\dots, X_{i_\ell} \right)=
	\kappa_\ell \left( {{n}^{-\frac{i_1-1}{2}}}\ \Sfunct_{i_1}, \dots, 
	{{n}^{-\frac{i_\ell-1}{2}}}\ \Sfunct_{i_\ell}    \right).
	\end{equation}
	
	By \cref{def:approx-factorization-charactersB}\ref{item:afp2} 
	the approximate factorization property of $(\chi_n)$ is equivalent to condition 
	\ref{item:cumulants-classic-polynomial} which we apply in the special case when
	$(x_1,\dots,x_{\ell}):=(\Sfunct_{i_1},\dots,\Sfunct_{i_\ell})$.
	The latter implies  that the right-hand side of 
	\eqref{eq:cumulant-concrete}
	is of order
	$O\left( {n}^{\frac{2-\ell}{2}} \right)$.
	This implies that for each $\ell\geq 3$
	\[ \lim_{n\to\infty} \kappa_\ell \left( Y_{i_1},\dots, Y_{i_\ell} \right) = 0.\]
	
	\smallskip
	
	Consider now the case $\ell=2$. If we adapt the above reasoning, 
	we get that
	the limit
	\[ \lim_{n\to\infty} \kappa_2\left( Y_{i_1}, Y_{i_2} \right)\]
	exists and is finite.
	
	\bigskip
	
	We consider now the case $\ell=1$. 
	By \cref{def:approx-factorization-charactersB}\ref{item:equivalent-enhanced},
	enhanced approximate factorization property implies that for each $k\geq 2$
	there exist constants
	$a_k'', b_k''$ such
	that
	\[ {n^{- \frac{k}{2}}}\ 
	\E \Sfunct_k(\lambda_n) = a_k'' +
	\frac{b_k''+o(1)}{\sqrt{n}}\]
	as $n\to\infty$.
	It follows that the cumulant
	\[
	\kappa_1(Y_k)=
	\E Y_k
	=
	\sqrt{n} \left({n^{- \frac{k}{2}}}\ \E
	\Sfunct_k(\lambda_n) - \lim_{m\to\infty} {m^{- \frac{k}{2}}}\ \E \Sfunct_k(\lambda_m) 
	\right) 
	\]
	converges as $n\to\infty$ to $b_k''$  (given explicitly by \eqref{eq:what-is-b-bis}).

	\bigskip

	Let us summarize the above discussion.
	We have proved that the limit
	\[ \lim_{n\to\infty} \kappa_\ell \left( Y_{i_1},\dots, Y_{i_\ell} \right) \]
	exists and is finite for any choice of $\ell \geq 1$ and $i_1,\dots,i_\ell \geq 2$.
	In other words: the joint distribution of the random variables $(Y_i)$
	converges in moments as $n\to\infty$ to the joint distribution of an abstract family of random
	variables $(Z_i)$ with the property that all cumulants vanish: 
	$\kappa_\ell \left( Z_{i_1},\dots, Z_{i_\ell} \right)=0$, except for $\ell\leq 2$.
	The latter is the defining property of the Gaussian distribution.
	Since the Gaussian distribution is uniquely determined by its moments,
	it follows that $(Y_i)$ converges to $(Z_i)$ not only in moments
	but also in the weak topology of probability measures, as required.
\end{proof}

\section*{Acknowledgments}

Research supported by \emph{Narodowe Centrum Nauki}, grant number \linebreak 2014/15/B/ST1/00064.

\bibliographystyle{alpha}
\bibliography{biblio2015}

\end{document}

%% file: preamble.tex
\usepackage{subfiles}
\usepackage{mathrsfs}
\usepackage{datetime}

\usepackage{hyperref}

\usepackage{times}
\usepackage{amssymb}
\usepackage{genyoungtabtikz}

\usepackage{enumitem}

\usepackage{amsthm}
\usepackage{thmtools,thm-restate}
\usepackage[capitalise,noabbrev]{cleveref}

\declaretheorem[numberwithin=section]{lemma}
\declaretheorem[sibling=lemma]{theorem}

\declaretheorem[sibling=lemma]{corollary}
\declaretheorem[sibling=lemma]{proposition}

\declaretheorem[sibling=lemma,style=remark]{example}
\declaretheorem[sibling=lemma,style=remark]{remark}

\declaretheorem[sibling=lemma,style=remark]{definition}

\declaretheorem[sibling=lemma,style=remark,name=Inconvenient Definition,Refname={Inconvenient Definition}]%
{inconvenientdefinition}
\declaretheorem[sibling=lemma,style=remark,name=Toy Problem,Refname={Problem}]%
{ntproblem}
\declaretheorem[sibling=lemma,style=remark,name=Toy Problem,Refname={Problem}]%
{tproblem}
\declaretheorem[sibling=lemma,style=theorem,name=Definition-Theorem,Refname={Definition-Theorem},refname={Definition-Theorem}]%
{defthm}

\Crefname{inconvenientdefinition}{inconvenientdefinition}{{inconvenientdefinition}}
\Crefname{ntproblem}{ntproblem}{{ntproblem}}
\Crefname{defthm}{defthm}{{defthm}}

\usepackage{color}
\usepackage{gensymb}
\usepackage{xcolor}

\usepackage{tikz}
\usepackage{tikz-cd}
\usepackage[backgroundcolor=black!10]{todonotes}
\usetikzlibrary{matrix,arrows,calc}
\usetikzlibrary{decorations.markings}
\usetikzlibrary{patterns}
\usetikzlibrary{snakes}

\newcommand{\tikzexternaldisable}{}
\newcommand{\tikzexternalenable}{}

\usepackage{subfig}
\usepackage{xspace}
\usepackage[T1]{fontenc}
\usepackage[utf8]{inputenc}
\usepackage{bm}

\usepackage{subfiles}

\usepackage{pgfplots}

\usepackage{commath}

\def\a{\alpha}

\newcommand{\kumupartitions}{k}

\newcommand{\sconstant}{g}

\DeclareMathOperator{\Tr}{Tr}
\DeclareMathOperator{\degg}{deg}
\DeclareMathOperator{\id}{id}

\newcommand{\Sym}[1]{\mathfrak{S}(#1)}

\newcommand{\Sfunct}{\mathcal{S}}

\newcommand{\Rfunct}{\mathcal{R}}

\newcommand{\disjoint}{\bullet}

\newcommand{\E}{\mathbb{E}}

\newcommand{\Q}{\mathbb{Q}}
\newcommand{\C}{\mathbb{C}}

\newcommand{\R}{\mathbb{R}}
\newcommand{\Alg}{\mathcal{A}}
\newcommand{\AlgB}{\mathcal{B}}

\newcommand{\condKumu}[2]{\kappa_{#1}^{#2}}
\newcommand{\kumuDisjointPoint}{\kappa_{\disjoint}}

\newcommand{\AllYoung}{\mathbb{Y}}
\newcommand{\Young}[1]{\mathbb{Y}_{#1}}
\newcommand{\AllPartitions}{\mathcal{P}}
\newcommand{\partitions}[1]{\AllPartitions_{#1}}

\DeclareMathOperator{\reg}{reg}

\newcommand{\Poly}{\mathscr{P}}

\DeclareMathOperator{\Ch}{Ch}

\newcommand{\Laurent}{\Q\left[A,A^{-1}\right]}

\numberwithin{equation}{section}

%% file: figures/bok5alpha4-1000prob.tex
 
 \begin{tikzpicture}[scale=0.240000]
\draw[red,thick](-24.000000,24.000000) -- (-20.000000,20.000000) -- (-19.000000,19.038000) -- (-18.000000,18.260000) -- (-17.000000,18.016000) -- (-16.000000,18.464000) -- (-15.000000,19.278000) -- (-14.000000,20.036000) -- (-13.000000,20.280000) -- (-12.000000,19.838000) -- (-11.000000,19.058000) -- (-10.000000,18.452000) -- (-9.000000,18.498000) -- (-8.000000,19.174000) -- (-7.000000,20.056000) -- (-6.000000,20.682000) -- (-5.000000,20.638000) -- (-4.000000,19.974000) -- (-3.000000,19.164000) -- (-2.000000,18.788000) -- (-1.000000,19.256000) -- (0.000000,20.112000) -- (1.000000,20.956000) -- (2.000000,21.328000) -- (3.000000,20.860000) -- (4.000000,20.054000) -- (5.000000,19.368000) -- (6.000000,19.292000) -- (7.000000,19.908000) -- (8.000000,20.756000) -- (9.000000,21.454000) -- (10.000000,21.534000) -- (11.000000,20.934000) -- (12.000000,20.166000) -- (13.000000,19.692000) -- (14.000000,19.906000) -- (15.000000,20.636000) -- (16.000000,21.462000) -- (17.000000,21.936000) -- (18.000000,21.722000) -- (19.000000,20.974000) -- (20.000000,20.000000) -- (24.000000,24.000000);
\draw[black!84.300000](5.000000,19.000000)--(4.000000,20.000000);
\draw[black!15.100000](9.000000,23.000000)--(8.000000,24.000000);
\draw[black!15.900000](-8.000000,16.000000)--(-9.000000,17.000000);
\draw[black!88.900000](-18.000000,18.000000)--(-19.000000,19.000000);
\draw[black!90.500000](-3.000000,19.000000)--(-4.000000,20.000000);
\draw[black!0.300000](-8.000000,24.000000)--(-9.000000,25.000000);
\draw[black!89.000000](-11.000000,19.000000)--(-12.000000,20.000000);
\draw[black!0.200000](2.000000,14.000000)--(1.000000,15.000000);
\draw[black!1.000000](-14.000000,14.000000)--(-15.000000,15.000000);
\draw[black!0.100000](7.000000,25.000000)--(6.000000,26.000000);
\draw[black!98.100000](-19.000000,19.000000)--(-20.000000,20.000000);
\draw[black!60.700000](18.000000,22.000000)--(17.000000,23.000000);
\draw[black!31.200000](2.000000,22.000000)--(1.000000,23.000000);
\draw[black!11.100000](-14.000000,22.000000)--(-15.000000,23.000000);
\draw[black!83.200000](-4.000000,20.000000)--(-5.000000,21.000000);
\draw[black!0.900000](15.000000,25.000000)--(14.000000,26.000000);
\draw[black!0.100000](-1.000000,25.000000)--(-2.000000,26.000000);
\draw[black!88.400000](12.000000,20.000000)--(11.000000,21.000000);
\draw[black!7.400000](-15.000000,15.000000)--(-16.000000,16.000000);
\draw[black!80.300000](-10.000000,18.000000)--(-11.000000,19.000000);
\draw[black!53.800000](6.000000,18.000000)--(5.000000,19.000000);
\draw[black!87.400000](19.000000,21.000000)--(18.000000,22.000000);
\draw[black!2.300000](-7.000000,15.000000)--(-8.000000,16.000000);
\draw[black!0.600000](1.000000,15.000000)--(0.000000,16.000000);
\draw[black!1.300000](16.000000,16.000000)--(15.000000,17.000000);
\draw[black!26.300000](17.000000,23.000000)--(16.000000,24.000000);
\draw[black!0.900000](0.000000,24.000000)--(-1.000000,25.000000);
\draw[black!46.000000](10.000000,22.000000)--(9.000000,23.000000);
\draw[black!80.000000](11.000000,21.000000)--(10.000000,22.000000);
\draw[black!52.200000](-5.000000,21.000000)--(-6.000000,22.000000);
\draw[black!6.300000](0.000000,16.000000)--(-1.000000,17.000000);
\draw[black!100.000000](-21.000000,21.000000)--(-22.000000,22.000000);
\draw[black!7.400000](16.000000,24.000000)--(15.000000,25.000000);
\draw[black!73.400000](3.000000,21.000000)--(2.000000,22.000000);
\draw[black!37.800000](-13.000000,21.000000)--(-14.000000,22.000000);
\draw[black!12.600000](15.000000,17.000000)--(14.000000,18.000000);
\draw[black!27.600000](-16.000000,16.000000)--(-17.000000,17.000000);
\draw[black!98.700000](20.000000,20.000000)--(19.000000,21.000000);
\draw[black!18.700000](-6.000000,22.000000)--(-7.000000,23.000000);
\draw[black!19.100000](7.000000,17.000000)--(6.000000,18.000000);
\draw[black!47.700000](-9.000000,17.000000)--(-10.000000,18.000000);
\draw[black!90.300000](4.000000,20.000000)--(3.000000,21.000000);
\draw[black!1.900000](-15.000000,23.000000)--(-16.000000,24.000000);
\draw[black!3.600000](-7.000000,23.000000)--(-8.000000,24.000000);
\draw[black!26.500000](-1.000000,17.000000)--(-2.000000,18.000000);
\draw[black!62.200000](-17.000000,17.000000)--(-18.000000,18.000000);
\draw[black!7.200000](1.000000,23.000000)--(0.000000,24.000000);
\draw[black!72.100000](-12.000000,20.000000)--(-13.000000,21.000000);
\draw[black!4.200000](8.000000,16.000000)--(7.000000,17.000000);
\draw[black!3.400000](8.000000,24.000000)--(7.000000,25.000000);
\draw[black!73.700000](13.000000,19.000000)--(12.000000,20.000000);
\draw[black!39.300000](14.000000,18.000000)--(13.000000,19.000000);
\draw[black!68.800000](-2.000000,18.000000)--(-3.000000,19.000000);
\draw[black!1.900000](-20.000000,20.000000)--(-16.000000,24.000000);
\draw[black!1.000000](-14.000000,14.000000)--(-10.000000,18.000000);
\draw[black!0.300000](-13.000000,21.000000)--(-9.000000,25.000000);
\draw[black!15.100000](-11.000000,19.000000)--(-7.000000,23.000000);
\draw[black!20.200000](-16.000000,16.000000)--(-12.000000,20.000000);
\draw[black!6.300000](-4.000000,20.000000)--(0.000000,24.000000);
\draw[black!30.900000](5.000000,19.000000)--(9.000000,23.000000);
\draw[black!2.300000](-7.000000,15.000000)--(-3.000000,19.000000);
\draw[black!1.300000](16.000000,16.000000)--(20.000000,20.000000);
\draw[black!0.100000](-6.000000,22.000000)--(-2.000000,26.000000);
\draw[black!6.400000](-15.000000,15.000000)--(-11.000000,19.000000);
\draw[black!3.300000](-12.000000,20.000000)--(-8.000000,24.000000);
\draw[black!18.900000](12.000000,20.000000)--(16.000000,24.000000);
\draw[black!33.600000](-10.000000,18.000000)--(-6.000000,22.000000);
\draw[black!34.600000](-17.000000,17.000000)--(-13.000000,21.000000);
\draw[black!0.800000](-5.000000,21.000000)--(-1.000000,25.000000);
\draw[black!34.900000](6.000000,18.000000)--(10.000000,22.000000);
\draw[black!20.200000](-1.000000,17.000000)--(3.000000,21.000000);
\draw[black!11.300000](15.000000,17.000000)--(19.000000,21.000000);
\draw[black!5.700000](0.000000,16.000000)--(4.000000,20.000000);
\draw[black!0.200000](2.000000,14.000000)--(6.000000,18.000000);
\draw[black!6.500000](11.000000,21.000000)--(15.000000,25.000000);
\draw[black!31.800000](-9.000000,17.000000)--(-5.000000,21.000000);
\draw[black!26.700000](-18.000000,18.000000)--(-14.000000,22.000000);
\draw[black!3.300000](3.000000,21.000000)--(7.000000,25.000000);
\draw[black!26.700000](14.000000,18.000000)--(18.000000,22.000000);
\draw[black!42.300000](-2.000000,18.000000)--(2.000000,22.000000);
\draw[black!14.900000](7.000000,17.000000)--(11.000000,21.000000);
\draw[black!0.400000](1.000000,15.000000)--(5.000000,19.000000);
\draw[black!0.900000](10.000000,22.000000)--(14.000000,26.000000);
\draw[black!13.600000](-8.000000,16.000000)--(-4.000000,20.000000);
\draw[black!9.200000](-19.000000,19.000000)--(-15.000000,23.000000);
\draw[black!24.000000](-3.000000,19.000000)--(1.000000,23.000000);
\draw[black!11.700000](4.000000,20.000000)--(8.000000,24.000000);
\draw[black!34.400000](13.000000,19.000000)--(17.000000,23.000000);
\draw[black!4.200000](8.000000,16.000000)--(12.000000,20.000000);
\draw[black!0.100000](2.000000,22.000000)--(6.000000,26.000000);
\draw[ultra thick,blue,dashed](-24.000000,24.000000) --(-20.000000,20.000000) -- (20.000000,20.000000) -- (24.000000,24.000000); \end{tikzpicture} 
 

%% file: figures/bok10alpha4-1000prob.tex
 
 \begin{tikzpicture}[scale=0.120000]
\draw[red,thick](-44.000000,44.000000) -- (-40.000000,40.000000) -- (-39.000000,39.026000) -- (-38.000000,38.240000) -- (-37.000000,37.940000) -- (-36.000000,38.230000) -- (-35.000000,38.966000) -- (-34.000000,39.696000) -- (-33.000000,39.988000) -- (-32.000000,39.716000) -- (-31.000000,39.082000) -- (-30.000000,38.540000) -- (-29.000000,38.486000) -- (-28.000000,38.922000) -- (-27.000000,39.614000) -- (-26.000000,40.160000) -- (-25.000000,40.216000) -- (-24.000000,39.792000) -- (-23.000000,39.134000) -- (-22.000000,38.676000) -- (-21.000000,38.768000) -- (-20.000000,39.300000) -- (-19.000000,40.028000) -- (-18.000000,40.514000) -- (-17.000000,40.426000) -- (-16.000000,39.894000) -- (-15.000000,39.210000) -- (-14.000000,38.814000) -- (-13.000000,38.928000) -- (-12.000000,39.520000) -- (-11.000000,40.220000) -- (-10.000000,40.616000) -- (-9.000000,40.504000) -- (-8.000000,39.914000) -- (-7.000000,39.210000) -- (-6.000000,38.896000) -- (-5.000000,39.114000) -- (-4.000000,39.782000) -- (-3.000000,40.522000) -- (-2.000000,40.856000) -- (-1.000000,40.644000) -- (0.000000,40.004000) -- (1.000000,39.332000) -- (2.000000,39.082000) -- (3.000000,39.422000) -- (4.000000,40.140000) -- (5.000000,40.824000) -- (6.000000,41.066000) -- (7.000000,40.722000) -- (8.000000,40.016000) -- (9.000000,39.360000) -- (10.000000,39.178000) -- (11.000000,39.582000) -- (12.000000,40.296000) -- (13.000000,40.954000) -- (14.000000,41.150000) -- (15.000000,40.770000) -- (16.000000,40.114000) -- (17.000000,39.616000) -- (18.000000,39.626000) -- (19.000000,40.112000) -- (20.000000,40.822000) -- (21.000000,41.330000) -- (22.000000,41.322000) -- (23.000000,40.846000) -- (24.000000,40.192000) -- (25.000000,39.790000) -- (26.000000,39.902000) -- (27.000000,40.516000) -- (28.000000,41.238000) -- (29.000000,41.658000) -- (30.000000,41.552000) -- (31.000000,40.952000) -- (32.000000,40.308000) -- (33.000000,40.040000) -- (34.000000,40.364000) -- (35.000000,41.098000) -- (36.000000,41.800000) -- (37.000000,42.084000) -- (38.000000,41.750000) -- (39.000000,40.966000) -- (40.000000,40.000000) -- (44.000000,44.000000);
\draw[black!85.300000](8.000000,40.000000)--(7.000000,41.000000);
\draw[black!55.200000](30.000000,42.000000)--(29.000000,43.000000);
\draw[black!0.100000](11.000000,45.000000)--(10.000000,46.000000);
\draw[black!10.700000](-34.000000,42.000000)--(-35.000000,43.000000);
\draw[black!2.900000](-3.000000,35.000000)--(-4.000000,36.000000);
\draw[black!0.200000](-2.000000,34.000000)--(-3.000000,35.000000);
\draw[black!79.500000](-8.000000,40.000000)--(-9.000000,41.000000);
\draw[black!98.300000](40.000000,40.000000)--(39.000000,41.000000);
\draw[black!0.300000](3.000000,45.000000)--(2.000000,46.000000);
\draw[black!5.500000](-19.000000,35.000000)--(-20.000000,36.000000);
\draw[black!71.200000](-24.000000,40.000000)--(-25.000000,41.000000);
\draw[black!17.500000](27.000000,37.000000)--(26.000000,38.000000);
\draw[black!50.200000](22.000000,42.000000)--(21.000000,43.000000);
\draw[black!0.100000](-21.000000,45.000000)--(-22.000000,46.000000);
\draw[black!0.200000](22.000000,34.000000)--(21.000000,35.000000);
\draw[black!49.400000](18.000000,38.000000)--(17.000000,39.000000);
\draw[black!100.000000](-40.000000,40.000000)--(-41.000000,41.000000);
\draw[black!10.700000](12.000000,36.000000)--(11.000000,37.000000);
\draw[black!69.000000](15.000000,41.000000)--(14.000000,42.000000);
\draw[black!15.000000](-4.000000,36.000000)--(-5.000000,37.000000);
\draw[black!100.000000](-46.000000,46.000000)--(-47.000000,47.000000);
\draw[black!1.200000](-10.000000,34.000000)--(-11.000000,35.000000);
\draw[black!13.200000](36.000000,44.000000)--(35.000000,45.000000);
\draw[black!65.000000](-37.000000,37.000000)--(-38.000000,38.000000);
\draw[black!100.000000](-42.000000,42.000000)--(-43.000000,43.000000);
\draw[black!1.800000](27.000000,45.000000)--(26.000000,46.000000);
\draw[black!0.400000](-25.000000,33.000000)--(-26.000000,34.000000);
\draw[black!82.900000](-23.000000,39.000000)--(-24.000000,40.000000);
\draw[black!0.900000](-28.000000,44.000000)--(-29.000000,45.000000);
\draw[black!70.100000](25.000000,39.000000)--(24.000000,40.000000);
\draw[black!0.100000](18.000000,46.000000)--(17.000000,47.000000);
\draw[black!6.500000](20.000000,44.000000)--(19.000000,45.000000);
\draw[black!98.700000](-39.000000,39.000000)--(-40.000000,40.000000);
\draw[black!100.000000](-44.000000,44.000000)--(-45.000000,45.000000);
\draw[black!0.200000](-5.000000,45.000000)--(-6.000000,46.000000);
\draw[black!1.200000](-18.000000,34.000000)--(-19.000000,35.000000);
\draw[black!9.300000](28.000000,44.000000)--(27.000000,45.000000);
\draw[black!4.700000](-11.000000,35.000000)--(-12.000000,36.000000);
\draw[black!3.600000](12.000000,44.000000)--(11.000000,45.000000);
\draw[black!9.000000](-27.000000,35.000000)--(-28.000000,36.000000);
\draw[black!67.200000](7.000000,41.000000)--(6.000000,42.000000);
\draw[black!0.600000](6.000000,34.000000)--(5.000000,35.000000);
\draw[black!62.500000](2.000000,38.000000)--(1.000000,39.000000);
\draw[black!1.600000](-4.000000,44.000000)--(-5.000000,45.000000);
\draw[black!18.900000](-12.000000,36.000000)--(-13.000000,37.000000);
\draw[black!60.600000](-1.000000,41.000000)--(-2.000000,42.000000);
\draw[black!89.200000](39.000000,41.000000)--(38.000000,42.000000);
\draw[black!1.500000](-20.000000,44.000000)--(-21.000000,45.000000);
\draw[black!0.400000](26.000000,46.000000)--(25.000000,47.000000);
\draw[black!24.400000](19.000000,37.000000)--(18.000000,38.000000);
\draw[black!46.800000](-25.000000,41.000000)--(-26.000000,42.000000);
\draw[black!2.600000](-26.000000,34.000000)--(-27.000000,35.000000);
\draw[black!82.800000](16.000000,40.000000)--(15.000000,41.000000);
\draw[black!59.100000](10.000000,38.000000)--(9.000000,39.000000);
\draw[black!29.700000](11.000000,37.000000)--(10.000000,38.000000);
\draw[black!1.700000](36.000000,36.000000)--(35.000000,37.000000);
\draw[black!37.300000](6.000000,42.000000)--(5.000000,43.000000);
\draw[black!14.900000](13.000000,43.000000)--(12.000000,44.000000);
\draw[black!82.000000](0.000000,40.000000)--(-1.000000,41.000000);
\draw[black!44.300000](-13.000000,37.000000)--(-14.000000,38.000000);
\draw[black!10.100000](-3.000000,43.000000)--(-4.000000,44.000000);
\draw[black!2.500000](35.000000,45.000000)--(34.000000,46.000000);
\draw[black!72.900000](-22.000000,38.000000)--(-23.000000,39.000000);
\draw[black!11.900000](-35.000000,35.000000)--(-36.000000,36.000000);
\draw[black!85.200000](-7.000000,39.000000)--(-8.000000,40.000000);
\draw[black!0.600000](34.000000,46.000000)--(33.000000,47.000000);
\draw[black!4.600000](28.000000,36.000000)--(27.000000,37.000000);
\draw[black!0.100000](30.000000,34.000000)--(29.000000,35.000000);
\draw[black!73.800000](23.000000,41.000000)--(22.000000,42.000000);
\draw[black!2.800000](-34.000000,34.000000)--(-35.000000,35.000000);
\draw[black!55.400000](-9.000000,41.000000)--(-10.000000,42.000000);
\draw[black!35.800000](37.000000,43.000000)--(36.000000,44.000000);
\draw[black!69.800000](-14.000000,38.000000)--(-15.000000,39.000000);
\draw[black!21.900000](-20.000000,36.000000)--(-21.000000,37.000000);
\draw[black!54.300000](-17.000000,41.000000)--(-18.000000,42.000000);
\draw[black!35.500000](-36.000000,36.000000)--(-37.000000,37.000000);
\draw[black!32.700000](3.000000,37.000000)--(2.000000,38.000000);
\draw[black!100.000000](-41.000000,41.000000)--(-42.000000,42.000000);
\draw[black!13.500000](5.000000,43.000000)--(4.000000,44.000000);
\draw[black!65.700000](-6.000000,38.000000)--(-7.000000,39.000000);
\draw[black!38.900000](-5.000000,37.000000)--(-6.000000,38.000000);
\draw[black!63.400000](33.000000,39.000000)--(32.000000,40.000000);
\draw[black!10.300000](-11.000000,43.000000)--(-12.000000,44.000000);
\draw[black!29.000000](-10.000000,42.000000)--(-11.000000,43.000000);
\draw[black!8.000000](20.000000,36.000000)--(19.000000,37.000000);
\draw[black!52.700000](-29.000000,37.000000)--(-30.000000,38.000000);
\draw[black!74.900000](17.000000,39.000000)--(16.000000,40.000000);
\draw[black!6.400000](-27.000000,43.000000)--(-28.000000,44.000000);
\draw[black!11.100000](4.000000,36.000000)--(3.000000,37.000000);
\draw[black!44.000000](26.000000,38.000000)--(25.000000,39.000000);
\draw[black!22.900000](21.000000,43.000000)--(20.000000,44.000000);
\draw[black!40.200000](14.000000,42.000000)--(13.000000,43.000000);
\draw[black!89.300000](-38.000000,38.000000)--(-39.000000,39.000000);
\draw[black!83.600000](1.000000,39.000000)--(0.000000,40.000000);
\draw[black!100.000000](-43.000000,43.000000)--(-44.000000,44.000000);
\draw[black!82.800000](9.000000,39.000000)--(8.000000,40.000000);
\draw[black!3.000000](4.000000,44.000000)--(3.000000,45.000000);
\draw[black!0.100000](-17.000000,33.000000)--(-18.000000,34.000000);
\draw[black!84.200000](-15.000000,39.000000)--(-16.000000,40.000000);
\draw[black!27.300000](-28.000000,36.000000)--(-29.000000,37.000000);
\draw[black!80.000000](31.000000,41.000000)--(30.000000,42.000000);
\draw[black!1.500000](-12.000000,44.000000)--(-13.000000,45.000000);
\draw[black!24.500000](-18.000000,42.000000)--(-19.000000,43.000000);
\draw[black!81.700000](-31.000000,39.000000)--(-32.000000,40.000000);
\draw[black!1.700000](21.000000,35.000000)--(20.000000,36.000000);
\draw[black!1.300000](19.000000,45.000000)--(18.000000,46.000000);
\draw[black!2.300000](5.000000,35.000000)--(4.000000,36.000000);
\draw[black!100.000000](-45.000000,45.000000)--(-46.000000,46.000000);
\draw[black!82.200000](32.000000,40.000000)--(31.000000,41.000000);
\draw[black!33.200000](34.000000,38.000000)--(33.000000,39.000000);
\draw[black!77.100000](-30.000000,38.000000)--(-31.000000,39.000000);
\draw[black!35.000000](-33.000000,41.000000)--(-34.000000,42.000000);
\draw[black!100.000000](-47.000000,47.000000)--(-48.000000,48.000000);
\draw[black!76.600000](-16.000000,40.000000)--(-17.000000,41.000000);
\draw[black!66.700000](38.000000,42.000000)--(37.000000,43.000000);
\draw[black!45.300000](-21.000000,37.000000)--(-22.000000,38.000000);
\draw[black!20.100000](-26.000000,42.000000)--(-27.000000,43.000000);
\draw[black!8.100000](-19.000000,43.000000)--(-20.000000,44.000000);
\draw[black!63.600000](-32.000000,40.000000)--(-33.000000,41.000000);
\draw[black!28.200000](29.000000,43.000000)--(28.000000,44.000000);
\draw[black!1.300000](-35.000000,43.000000)--(-36.000000,44.000000);
\draw[black!33.100000](-2.000000,42.000000)--(-3.000000,43.000000);
\draw[black!0.200000](-9.000000,33.000000)--(-10.000000,34.000000);
\draw[black!0.800000](29.000000,35.000000)--(28.000000,36.000000);
\draw[black!82.700000](24.000000,40.000000)--(23.000000,41.000000);
\draw[black!0.400000](-33.000000,33.000000)--(-34.000000,34.000000);
\draw[black!2.200000](13.000000,35.000000)--(12.000000,36.000000);
\draw[black!10.800000](35.000000,37.000000)--(34.000000,38.000000);
\draw[black!26.800000](-30.000000,38.000000)--(-26.000000,42.000000);
\draw[black!16.400000](16.000000,40.000000)--(20.000000,44.000000);
\draw[black!12.900000](27.000000,37.000000)--(31.000000,41.000000);
\draw[black!14.200000](-12.000000,36.000000)--(-8.000000,40.000000);
\draw[black!1.700000](36.000000,36.000000)--(40.000000,40.000000);
\draw[black!2.700000](-3.000000,35.000000)--(1.000000,39.000000);
\draw[black!0.200000](22.000000,34.000000)--(26.000000,38.000000);
\draw[black!30.000000](2.000000,38.000000)--(6.000000,42.000000);
\draw[black!4.300000](-19.000000,35.000000)--(-15.000000,39.000000);
\draw[black!0.100000](30.000000,34.000000)--(34.000000,38.000000);
\draw[black!3.500000](-11.000000,35.000000)--(-7.000000,39.000000);
\draw[black!30.900000](33.000000,39.000000)--(37.000000,43.000000);
\draw[black!100.000000](44.000000,44.000000)--(48.000000,48.000000);
\draw[black!6.600000](-24.000000,40.000000)--(-20.000000,44.000000);
\draw[black!2.700000](-1.000000,41.000000)--(3.000000,45.000000);
\draw[black!6.400000](-27.000000,35.000000)--(-23.000000,39.000000);
\draw[black!16.400000](-20.000000,36.000000)--(-16.000000,40.000000);
\draw[black!9.400000](-39.000000,39.000000)--(-35.000000,43.000000);
\draw[black!0.600000](6.000000,34.000000)--(10.000000,38.000000);
\draw[black!0.200000](-10.000000,42.000000)--(-6.000000,46.000000);
\draw[black!19.000000](11.000000,37.000000)--(15.000000,41.000000);
\draw[black!24.100000](-5.000000,37.000000)--(-1.000000,41.000000);
\draw[black!6.300000](20.000000,36.000000)--(24.000000,40.000000);
\draw[black!11.300000](8.000000,40.000000)--(12.000000,44.000000);
\draw[black!7.500000](23.000000,41.000000)--(27.000000,45.000000);
\draw[black!8.800000](-16.000000,40.000000)--(-12.000000,44.000000);
\draw[black!22.600000](32.000000,40.000000)--(36.000000,44.000000);
\draw[black!23.600000](-36.000000,36.000000)--(-32.000000,40.000000);
\draw[black!25.800000](-29.000000,37.000000)--(-25.000000,41.000000);
\draw[black!30.000000](10.000000,38.000000)--(14.000000,42.000000);
\draw[black!0.100000](13.000000,43.000000)--(17.000000,47.000000);
\draw[black!26.700000](26.000000,38.000000)--(30.000000,42.000000);
\draw[black!18.700000](-15.000000,39.000000)--(-11.000000,43.000000);
\draw[black!0.600000](29.000000,43.000000)--(33.000000,47.000000);
\draw[black!0.400000](-33.000000,33.000000)--(-29.000000,37.000000);
\draw[black!2.200000](13.000000,35.000000)--(17.000000,39.000000);
\draw[black!21.600000](3.000000,37.000000)--(7.000000,41.000000);
\draw[black!0.700000](29.000000,35.000000)--(33.000000,39.000000);
\draw[black!1.000000](-10.000000,34.000000)--(-6.000000,38.000000);
\draw[black!22.500000](34.000000,38.000000)--(38.000000,42.000000);
\draw[black!1.400000](-25.000000,41.000000)--(-21.000000,45.000000);
\draw[black!10.500000](0.000000,40.000000)--(4.000000,44.000000);
\draw[black!2.200000](-26.000000,34.000000)--(-22.000000,38.000000);
\draw[black!23.800000](-21.000000,37.000000)--(-17.000000,41.000000);
\draw[black!0.200000](-9.000000,33.000000)--(-5.000000,37.000000);
\draw[black!0.900000](-33.000000,41.000000)--(-29.000000,45.000000);
\draw[black!27.800000](-6.000000,38.000000)--(-2.000000,42.000000);
\draw[black!0.100000](-26.000000,42.000000)--(-22.000000,46.000000);
\draw[black!27.700000](17.000000,39.000000)--(21.000000,43.000000);
\draw[black!3.500000](7.000000,41.000000)--(11.000000,45.000000);
\draw[black!29.800000](-22.000000,38.000000)--(-18.000000,42.000000);
\draw[black!1.400000](22.000000,42.000000)--(26.000000,46.000000);
\draw[black!29.500000](-37.000000,37.000000)--(-33.000000,41.000000);
\draw[black!8.500000](-8.000000,40.000000)--(-4.000000,44.000000);
\draw[black!100.000000](40.000000,40.000000)--(44.000000,44.000000);
\draw[black!18.300000](-28.000000,36.000000)--(-24.000000,40.000000);
\draw[black!1.300000](-40.000000,40.000000)--(-36.000000,44.000000);
\draw[black!25.400000](9.000000,39.000000)--(13.000000,43.000000);
\draw[black!1.200000](14.000000,42.000000)--(18.000000,46.000000);
\draw[black!0.100000](6.000000,42.000000)--(10.000000,46.000000);
\draw[black!27.600000](25.000000,39.000000)--(29.000000,43.000000);
\draw[black!26.600000](-14.000000,38.000000)--(-10.000000,42.000000);
\draw[black!1.900000](30.000000,42.000000)--(34.000000,46.000000);
\draw[black!2.400000](-34.000000,34.000000)--(-30.000000,38.000000);
\draw[black!13.700000](-31.000000,39.000000)--(-27.000000,43.000000);
\draw[black!8.800000](4.000000,36.000000)--(8.000000,40.000000);
\draw[black!5.200000](15.000000,41.000000)--(19.000000,45.000000);
\draw[black!0.100000](-17.000000,33.000000)--(-13.000000,37.000000);
\draw[black!3.800000](28.000000,36.000000)--(32.000000,40.000000);
\draw[black!25.500000](-13.000000,37.000000)--(-9.000000,41.000000);
\draw[black!9.100000](35.000000,37.000000)--(39.000000,41.000000);
\draw[black!0.200000](-2.000000,34.000000)--(2.000000,38.000000);
\draw[black!1.500000](21.000000,35.000000)--(25.000000,39.000000);
\draw[black!23.800000](1.000000,39.000000)--(5.000000,43.000000);
\draw[black!0.400000](-25.000000,33.000000)--(-21.000000,37.000000);
\draw[black!1.100000](-18.000000,34.000000)--(-14.000000,38.000000);
\draw[black!5.500000](-32.000000,40.000000)--(-28.000000,44.000000);
\draw[black!23.000000](-7.000000,39.000000)--(-3.000000,43.000000);
\draw[black!25.000000](18.000000,38.000000)--(22.000000,42.000000);
\draw[black!0.300000](-2.000000,42.000000)--(2.000000,46.000000);
\draw[black!16.400000](-23.000000,39.000000)--(-19.000000,43.000000);
\draw[black!0.400000](21.000000,43.000000)--(25.000000,47.000000);
\draw[black!24.300000](-38.000000,38.000000)--(-34.000000,42.000000);
\draw[black!1.700000](5.000000,35.000000)--(9.000000,39.000000);
\draw[black!1.400000](-9.000000,41.000000)--(-5.000000,45.000000);
\draw[black!8.500000](12.000000,36.000000)--(16.000000,40.000000);
\draw[black!12.100000](-4.000000,36.000000)--(0.000000,40.000000);
\draw[black!16.400000](19.000000,37.000000)--(23.000000,41.000000);
\draw[black!18.900000](24.000000,40.000000)--(28.000000,44.000000);
\draw[black!1.500000](-17.000000,41.000000)--(-13.000000,45.000000);
\draw[black!10.700000](31.000000,41.000000)--(35.000000,45.000000);
\draw[black!9.100000](-35.000000,35.000000)--(-31.000000,39.000000);
\draw[ultra thick,blue,dashed](-44.000000,44.000000) --(-40.000000,40.000000) -- (40.000000,40.000000) -- (44.000000,44.000000); \end{tikzpicture} 
 

%% file: figures/bok5alpha4-1000prob-ANISORESCALED.tex
 
 \begin{tikzpicture}
\begin{scope}[xscale=0.120000,yscale=0.12*sqrt(50)]
\clip(-20*1.2,20*0.75) rectangle (20*1.2,20*1.25);
\draw[red,thick](-24.000000,24.000000) -- (-20.000000,20.000000) -- (-19.000000,19.038000) -- (-18.000000,18.260000) -- (-17.000000,18.016000) -- (-16.000000,18.464000) -- (-15.000000,19.278000) -- (-14.000000,20.036000) -- (-13.000000,20.280000) -- (-12.000000,19.838000) -- (-11.000000,19.058000) -- (-10.000000,18.452000) -- (-9.000000,18.498000) -- (-8.000000,19.174000) -- (-7.000000,20.056000) -- (-6.000000,20.682000) -- (-5.000000,20.638000) -- (-4.000000,19.974000) -- (-3.000000,19.164000) -- (-2.000000,18.788000) -- (-1.000000,19.256000) -- (0.000000,20.112000) -- (1.000000,20.956000) -- (2.000000,21.328000) -- (3.000000,20.860000) -- (4.000000,20.054000) -- (5.000000,19.368000) -- (6.000000,19.292000) -- (7.000000,19.908000) -- (8.000000,20.756000) -- (9.000000,21.454000) -- (10.000000,21.534000) -- (11.000000,20.934000) -- (12.000000,20.166000) -- (13.000000,19.692000) -- (14.000000,19.906000) -- (15.000000,20.636000) -- (16.000000,21.462000) -- (17.000000,21.936000) -- (18.000000,21.722000) -- (19.000000,20.974000) -- (20.000000,20.000000) -- (24.000000,24.000000);
\draw[black!84.300000](5.000000,19.000000)--(4.000000,20.000000);
\draw[black!15.100000](9.000000,23.000000)--(8.000000,24.000000);
\draw[black!15.900000](-8.000000,16.000000)--(-9.000000,17.000000);
\draw[black!88.900000](-18.000000,18.000000)--(-19.000000,19.000000);
\draw[black!90.500000](-3.000000,19.000000)--(-4.000000,20.000000);
\draw[black!0.300000](-8.000000,24.000000)--(-9.000000,25.000000);
\draw[black!89.000000](-11.000000,19.000000)--(-12.000000,20.000000);
\draw[black!0.200000](2.000000,14.000000)--(1.000000,15.000000);
\draw[black!1.000000](-14.000000,14.000000)--(-15.000000,15.000000);
\draw[black!0.100000](7.000000,25.000000)--(6.000000,26.000000);
\draw[black!98.100000](-19.000000,19.000000)--(-20.000000,20.000000);
\draw[black!60.700000](18.000000,22.000000)--(17.000000,23.000000);
\draw[black!31.200000](2.000000,22.000000)--(1.000000,23.000000);
\draw[black!11.100000](-14.000000,22.000000)--(-15.000000,23.000000);
\draw[black!83.200000](-4.000000,20.000000)--(-5.000000,21.000000);
\draw[black!0.900000](15.000000,25.000000)--(14.000000,26.000000);
\draw[black!0.100000](-1.000000,25.000000)--(-2.000000,26.000000);
\draw[black!88.400000](12.000000,20.000000)--(11.000000,21.000000);
\draw[black!7.400000](-15.000000,15.000000)--(-16.000000,16.000000);
\draw[black!80.300000](-10.000000,18.000000)--(-11.000000,19.000000);
\draw[black!53.800000](6.000000,18.000000)--(5.000000,19.000000);
\draw[black!87.400000](19.000000,21.000000)--(18.000000,22.000000);
\draw[black!2.300000](-7.000000,15.000000)--(-8.000000,16.000000);
\draw[black!0.600000](1.000000,15.000000)--(0.000000,16.000000);
\draw[black!1.300000](16.000000,16.000000)--(15.000000,17.000000);
\draw[black!26.300000](17.000000,23.000000)--(16.000000,24.000000);
\draw[black!0.900000](0.000000,24.000000)--(-1.000000,25.000000);
\draw[black!46.000000](10.000000,22.000000)--(9.000000,23.000000);
\draw[black!80.000000](11.000000,21.000000)--(10.000000,22.000000);
\draw[black!52.200000](-5.000000,21.000000)--(-6.000000,22.000000);
\draw[black!6.300000](0.000000,16.000000)--(-1.000000,17.000000);
\draw[black!100.000000](-21.000000,21.000000)--(-22.000000,22.000000);
\draw[black!7.400000](16.000000,24.000000)--(15.000000,25.000000);
\draw[black!73.400000](3.000000,21.000000)--(2.000000,22.000000);
\draw[black!37.800000](-13.000000,21.000000)--(-14.000000,22.000000);
\draw[black!12.600000](15.000000,17.000000)--(14.000000,18.000000);
\draw[black!27.600000](-16.000000,16.000000)--(-17.000000,17.000000);
\draw[black!98.700000](20.000000,20.000000)--(19.000000,21.000000);
\draw[black!18.700000](-6.000000,22.000000)--(-7.000000,23.000000);
\draw[black!19.100000](7.000000,17.000000)--(6.000000,18.000000);
\draw[black!47.700000](-9.000000,17.000000)--(-10.000000,18.000000);
\draw[black!90.300000](4.000000,20.000000)--(3.000000,21.000000);
\draw[black!1.900000](-15.000000,23.000000)--(-16.000000,24.000000);
\draw[black!3.600000](-7.000000,23.000000)--(-8.000000,24.000000);
\draw[black!26.500000](-1.000000,17.000000)--(-2.000000,18.000000);
\draw[black!62.200000](-17.000000,17.000000)--(-18.000000,18.000000);
\draw[black!7.200000](1.000000,23.000000)--(0.000000,24.000000);
\draw[black!72.100000](-12.000000,20.000000)--(-13.000000,21.000000);
\draw[black!4.200000](8.000000,16.000000)--(7.000000,17.000000);
\draw[black!3.400000](8.000000,24.000000)--(7.000000,25.000000);
\draw[black!73.700000](13.000000,19.000000)--(12.000000,20.000000);
\draw[black!39.300000](14.000000,18.000000)--(13.000000,19.000000);
\draw[black!68.800000](-2.000000,18.000000)--(-3.000000,19.000000);
\draw[black!1.900000](-20.000000,20.000000)--(-16.000000,24.000000);
\draw[black!1.000000](-14.000000,14.000000)--(-10.000000,18.000000);
\draw[black!0.300000](-13.000000,21.000000)--(-9.000000,25.000000);
\draw[black!15.100000](-11.000000,19.000000)--(-7.000000,23.000000);
\draw[black!20.200000](-16.000000,16.000000)--(-12.000000,20.000000);
\draw[black!6.300000](-4.000000,20.000000)--(0.000000,24.000000);
\draw[black!30.900000](5.000000,19.000000)--(9.000000,23.000000);
\draw[black!2.300000](-7.000000,15.000000)--(-3.000000,19.000000);
\draw[black!1.300000](16.000000,16.000000)--(20.000000,20.000000);
\draw[black!0.100000](-6.000000,22.000000)--(-2.000000,26.000000);
\draw[black!6.400000](-15.000000,15.000000)--(-11.000000,19.000000);
\draw[black!3.300000](-12.000000,20.000000)--(-8.000000,24.000000);
\draw[black!18.900000](12.000000,20.000000)--(16.000000,24.000000);
\draw[black!33.600000](-10.000000,18.000000)--(-6.000000,22.000000);
\draw[black!34.600000](-17.000000,17.000000)--(-13.000000,21.000000);
\draw[black!0.800000](-5.000000,21.000000)--(-1.000000,25.000000);
\draw[black!34.900000](6.000000,18.000000)--(10.000000,22.000000);
\draw[black!20.200000](-1.000000,17.000000)--(3.000000,21.000000);
\draw[black!11.300000](15.000000,17.000000)--(19.000000,21.000000);
\draw[black!5.700000](0.000000,16.000000)--(4.000000,20.000000);
\draw[black!0.200000](2.000000,14.000000)--(6.000000,18.000000);
\draw[black!6.500000](11.000000,21.000000)--(15.000000,25.000000);
\draw[black!31.800000](-9.000000,17.000000)--(-5.000000,21.000000);
\draw[black!26.700000](-18.000000,18.000000)--(-14.000000,22.000000);
\draw[black!3.300000](3.000000,21.000000)--(7.000000,25.000000);
\draw[black!26.700000](14.000000,18.000000)--(18.000000,22.000000);
\draw[black!42.300000](-2.000000,18.000000)--(2.000000,22.000000);
\draw[black!14.900000](7.000000,17.000000)--(11.000000,21.000000);
\draw[black!0.400000](1.000000,15.000000)--(5.000000,19.000000);
\draw[black!0.900000](10.000000,22.000000)--(14.000000,26.000000);
\draw[black!13.600000](-8.000000,16.000000)--(-4.000000,20.000000);
\draw[black!9.200000](-19.000000,19.000000)--(-15.000000,23.000000);
\draw[black!24.000000](-3.000000,19.000000)--(1.000000,23.000000);
\draw[black!11.700000](4.000000,20.000000)--(8.000000,24.000000);
\draw[black!34.400000](13.000000,19.000000)--(17.000000,23.000000);
\draw[black!4.200000](8.000000,16.000000)--(12.000000,20.000000);
\draw[black!0.100000](2.000000,22.000000)--(6.000000,26.000000);
\draw[ultra thick,blue,dashed](-24.000000,24.000000) --(-20.000000,20.000000) -- (20.000000,20.000000) -- (24.000000,24.000000); 
\end{scope}
\end{tikzpicture} 
 

%% file: figures/bok10alpha4-1000probANISOTROPIC.tex
 
 \begin{tikzpicture}[xscale=0.060000,yscale=0.12*sqrt(50)]
\clip(-40*1.2,40*0.875) rectangle (40*1.2,40*1.125);
\draw[red,thick](-44.000000,44.000000) -- (-40.000000,40.000000) -- (-39.000000,39.026000) -- (-38.000000,38.240000) -- (-37.000000,37.940000) -- (-36.000000,38.230000) -- (-35.000000,38.966000) -- (-34.000000,39.696000) -- (-33.000000,39.988000) -- (-32.000000,39.716000) -- (-31.000000,39.082000) -- (-30.000000,38.540000) -- (-29.000000,38.486000) -- (-28.000000,38.922000) -- (-27.000000,39.614000) -- (-26.000000,40.160000) -- (-25.000000,40.216000) -- (-24.000000,39.792000) -- (-23.000000,39.134000) -- (-22.000000,38.676000) -- (-21.000000,38.768000) -- (-20.000000,39.300000) -- (-19.000000,40.028000) -- (-18.000000,40.514000) -- (-17.000000,40.426000) -- (-16.000000,39.894000) -- (-15.000000,39.210000) -- (-14.000000,38.814000) -- (-13.000000,38.928000) -- (-12.000000,39.520000) -- (-11.000000,40.220000) -- (-10.000000,40.616000) -- (-9.000000,40.504000) -- (-8.000000,39.914000) -- (-7.000000,39.210000) -- (-6.000000,38.896000) -- (-5.000000,39.114000) -- (-4.000000,39.782000) -- (-3.000000,40.522000) -- (-2.000000,40.856000) -- (-1.000000,40.644000) -- (0.000000,40.004000) -- (1.000000,39.332000) -- (2.000000,39.082000) -- (3.000000,39.422000) -- (4.000000,40.140000) -- (5.000000,40.824000) -- (6.000000,41.066000) -- (7.000000,40.722000) -- (8.000000,40.016000) -- (9.000000,39.360000) -- (10.000000,39.178000) -- (11.000000,39.582000) -- (12.000000,40.296000) -- (13.000000,40.954000) -- (14.000000,41.150000) -- (15.000000,40.770000) -- (16.000000,40.114000) -- (17.000000,39.616000) -- (18.000000,39.626000) -- (19.000000,40.112000) -- (20.000000,40.822000) -- (21.000000,41.330000) -- (22.000000,41.322000) -- (23.000000,40.846000) -- (24.000000,40.192000) -- (25.000000,39.790000) -- (26.000000,39.902000) -- (27.000000,40.516000) -- (28.000000,41.238000) -- (29.000000,41.658000) -- (30.000000,41.552000) -- (31.000000,40.952000) -- (32.000000,40.308000) -- (33.000000,40.040000) -- (34.000000,40.364000) -- (35.000000,41.098000) -- (36.000000,41.800000) -- (37.000000,42.084000) -- (38.000000,41.750000) -- (39.000000,40.966000) -- (40.000000,40.000000) -- (44.000000,44.000000);
\draw[black!85.300000](8.000000,40.000000)--(7.000000,41.000000);
\draw[black!55.200000](30.000000,42.000000)--(29.000000,43.000000);
\draw[black!0.100000](11.000000,45.000000)--(10.000000,46.000000);
\draw[black!10.700000](-34.000000,42.000000)--(-35.000000,43.000000);
\draw[black!2.900000](-3.000000,35.000000)--(-4.000000,36.000000);
\draw[black!0.200000](-2.000000,34.000000)--(-3.000000,35.000000);
\draw[black!79.500000](-8.000000,40.000000)--(-9.000000,41.000000);
\draw[black!98.300000](40.000000,40.000000)--(39.000000,41.000000);
\draw[black!0.300000](3.000000,45.000000)--(2.000000,46.000000);
\draw[black!5.500000](-19.000000,35.000000)--(-20.000000,36.000000);
\draw[black!71.200000](-24.000000,40.000000)--(-25.000000,41.000000);
\draw[black!17.500000](27.000000,37.000000)--(26.000000,38.000000);
\draw[black!50.200000](22.000000,42.000000)--(21.000000,43.000000);
\draw[black!0.100000](-21.000000,45.000000)--(-22.000000,46.000000);
\draw[black!0.200000](22.000000,34.000000)--(21.000000,35.000000);
\draw[black!49.400000](18.000000,38.000000)--(17.000000,39.000000);
\draw[black!100.000000](-40.000000,40.000000)--(-41.000000,41.000000);
\draw[black!10.700000](12.000000,36.000000)--(11.000000,37.000000);
\draw[black!69.000000](15.000000,41.000000)--(14.000000,42.000000);
\draw[black!15.000000](-4.000000,36.000000)--(-5.000000,37.000000);
\draw[black!100.000000](-46.000000,46.000000)--(-47.000000,47.000000);
\draw[black!1.200000](-10.000000,34.000000)--(-11.000000,35.000000);
\draw[black!13.200000](36.000000,44.000000)--(35.000000,45.000000);
\draw[black!65.000000](-37.000000,37.000000)--(-38.000000,38.000000);
\draw[black!100.000000](-42.000000,42.000000)--(-43.000000,43.000000);
\draw[black!1.800000](27.000000,45.000000)--(26.000000,46.000000);
\draw[black!0.400000](-25.000000,33.000000)--(-26.000000,34.000000);
\draw[black!82.900000](-23.000000,39.000000)--(-24.000000,40.000000);
\draw[black!0.900000](-28.000000,44.000000)--(-29.000000,45.000000);
\draw[black!70.100000](25.000000,39.000000)--(24.000000,40.000000);
\draw[black!0.100000](18.000000,46.000000)--(17.000000,47.000000);
\draw[black!6.500000](20.000000,44.000000)--(19.000000,45.000000);
\draw[black!98.700000](-39.000000,39.000000)--(-40.000000,40.000000);
\draw[black!100.000000](-44.000000,44.000000)--(-45.000000,45.000000);
\draw[black!0.200000](-5.000000,45.000000)--(-6.000000,46.000000);
\draw[black!1.200000](-18.000000,34.000000)--(-19.000000,35.000000);
\draw[black!9.300000](28.000000,44.000000)--(27.000000,45.000000);
\draw[black!4.700000](-11.000000,35.000000)--(-12.000000,36.000000);
\draw[black!3.600000](12.000000,44.000000)--(11.000000,45.000000);
\draw[black!9.000000](-27.000000,35.000000)--(-28.000000,36.000000);
\draw[black!67.200000](7.000000,41.000000)--(6.000000,42.000000);
\draw[black!0.600000](6.000000,34.000000)--(5.000000,35.000000);
\draw[black!62.500000](2.000000,38.000000)--(1.000000,39.000000);
\draw[black!1.600000](-4.000000,44.000000)--(-5.000000,45.000000);
\draw[black!18.900000](-12.000000,36.000000)--(-13.000000,37.000000);
\draw[black!60.600000](-1.000000,41.000000)--(-2.000000,42.000000);
\draw[black!89.200000](39.000000,41.000000)--(38.000000,42.000000);
\draw[black!1.500000](-20.000000,44.000000)--(-21.000000,45.000000);
\draw[black!0.400000](26.000000,46.000000)--(25.000000,47.000000);
\draw[black!24.400000](19.000000,37.000000)--(18.000000,38.000000);
\draw[black!46.800000](-25.000000,41.000000)--(-26.000000,42.000000);
\draw[black!2.600000](-26.000000,34.000000)--(-27.000000,35.000000);
\draw[black!82.800000](16.000000,40.000000)--(15.000000,41.000000);
\draw[black!59.100000](10.000000,38.000000)--(9.000000,39.000000);
\draw[black!29.700000](11.000000,37.000000)--(10.000000,38.000000);
\draw[black!1.700000](36.000000,36.000000)--(35.000000,37.000000);
\draw[black!37.300000](6.000000,42.000000)--(5.000000,43.000000);
\draw[black!14.900000](13.000000,43.000000)--(12.000000,44.000000);
\draw[black!82.000000](0.000000,40.000000)--(-1.000000,41.000000);
\draw[black!44.300000](-13.000000,37.000000)--(-14.000000,38.000000);
\draw[black!10.100000](-3.000000,43.000000)--(-4.000000,44.000000);
\draw[black!2.500000](35.000000,45.000000)--(34.000000,46.000000);
\draw[black!72.900000](-22.000000,38.000000)--(-23.000000,39.000000);
\draw[black!11.900000](-35.000000,35.000000)--(-36.000000,36.000000);
\draw[black!85.200000](-7.000000,39.000000)--(-8.000000,40.000000);
\draw[black!0.600000](34.000000,46.000000)--(33.000000,47.000000);
\draw[black!4.600000](28.000000,36.000000)--(27.000000,37.000000);
\draw[black!0.100000](30.000000,34.000000)--(29.000000,35.000000);
\draw[black!73.800000](23.000000,41.000000)--(22.000000,42.000000);
\draw[black!2.800000](-34.000000,34.000000)--(-35.000000,35.000000);
\draw[black!55.400000](-9.000000,41.000000)--(-10.000000,42.000000);
\draw[black!35.800000](37.000000,43.000000)--(36.000000,44.000000);
\draw[black!69.800000](-14.000000,38.000000)--(-15.000000,39.000000);
\draw[black!21.900000](-20.000000,36.000000)--(-21.000000,37.000000);
\draw[black!54.300000](-17.000000,41.000000)--(-18.000000,42.000000);
\draw[black!35.500000](-36.000000,36.000000)--(-37.000000,37.000000);
\draw[black!32.700000](3.000000,37.000000)--(2.000000,38.000000);
\draw[black!100.000000](-41.000000,41.000000)--(-42.000000,42.000000);
\draw[black!13.500000](5.000000,43.000000)--(4.000000,44.000000);
\draw[black!65.700000](-6.000000,38.000000)--(-7.000000,39.000000);
\draw[black!38.900000](-5.000000,37.000000)--(-6.000000,38.000000);
\draw[black!63.400000](33.000000,39.000000)--(32.000000,40.000000);
\draw[black!10.300000](-11.000000,43.000000)--(-12.000000,44.000000);
\draw[black!29.000000](-10.000000,42.000000)--(-11.000000,43.000000);
\draw[black!8.000000](20.000000,36.000000)--(19.000000,37.000000);
\draw[black!52.700000](-29.000000,37.000000)--(-30.000000,38.000000);
\draw[black!74.900000](17.000000,39.000000)--(16.000000,40.000000);
\draw[black!6.400000](-27.000000,43.000000)--(-28.000000,44.000000);
\draw[black!11.100000](4.000000,36.000000)--(3.000000,37.000000);
\draw[black!44.000000](26.000000,38.000000)--(25.000000,39.000000);
\draw[black!22.900000](21.000000,43.000000)--(20.000000,44.000000);
\draw[black!40.200000](14.000000,42.000000)--(13.000000,43.000000);
\draw[black!89.300000](-38.000000,38.000000)--(-39.000000,39.000000);
\draw[black!83.600000](1.000000,39.000000)--(0.000000,40.000000);
\draw[black!100.000000](-43.000000,43.000000)--(-44.000000,44.000000);
\draw[black!82.800000](9.000000,39.000000)--(8.000000,40.000000);
\draw[black!3.000000](4.000000,44.000000)--(3.000000,45.000000);
\draw[black!0.100000](-17.000000,33.000000)--(-18.000000,34.000000);
\draw[black!84.200000](-15.000000,39.000000)--(-16.000000,40.000000);
\draw[black!27.300000](-28.000000,36.000000)--(-29.000000,37.000000);
\draw[black!80.000000](31.000000,41.000000)--(30.000000,42.000000);
\draw[black!1.500000](-12.000000,44.000000)--(-13.000000,45.000000);
\draw[black!24.500000](-18.000000,42.000000)--(-19.000000,43.000000);
\draw[black!81.700000](-31.000000,39.000000)--(-32.000000,40.000000);
\draw[black!1.700000](21.000000,35.000000)--(20.000000,36.000000);
\draw[black!1.300000](19.000000,45.000000)--(18.000000,46.000000);
\draw[black!2.300000](5.000000,35.000000)--(4.000000,36.000000);
\draw[black!100.000000](-45.000000,45.000000)--(-46.000000,46.000000);
\draw[black!82.200000](32.000000,40.000000)--(31.000000,41.000000);
\draw[black!33.200000](34.000000,38.000000)--(33.000000,39.000000);
\draw[black!77.100000](-30.000000,38.000000)--(-31.000000,39.000000);
\draw[black!35.000000](-33.000000,41.000000)--(-34.000000,42.000000);
\draw[black!100.000000](-47.000000,47.000000)--(-48.000000,48.000000);
\draw[black!76.600000](-16.000000,40.000000)--(-17.000000,41.000000);
\draw[black!66.700000](38.000000,42.000000)--(37.000000,43.000000);
\draw[black!45.300000](-21.000000,37.000000)--(-22.000000,38.000000);
\draw[black!20.100000](-26.000000,42.000000)--(-27.000000,43.000000);
\draw[black!8.100000](-19.000000,43.000000)--(-20.000000,44.000000);
\draw[black!63.600000](-32.000000,40.000000)--(-33.000000,41.000000);
\draw[black!28.200000](29.000000,43.000000)--(28.000000,44.000000);
\draw[black!1.300000](-35.000000,43.000000)--(-36.000000,44.000000);
\draw[black!33.100000](-2.000000,42.000000)--(-3.000000,43.000000);
\draw[black!0.200000](-9.000000,33.000000)--(-10.000000,34.000000);
\draw[black!0.800000](29.000000,35.000000)--(28.000000,36.000000);
\draw[black!82.700000](24.000000,40.000000)--(23.000000,41.000000);
\draw[black!0.400000](-33.000000,33.000000)--(-34.000000,34.000000);
\draw[black!2.200000](13.000000,35.000000)--(12.000000,36.000000);
\draw[black!10.800000](35.000000,37.000000)--(34.000000,38.000000);
\draw[black!26.800000](-30.000000,38.000000)--(-26.000000,42.000000);
\draw[black!16.400000](16.000000,40.000000)--(20.000000,44.000000);
\draw[black!12.900000](27.000000,37.000000)--(31.000000,41.000000);
\draw[black!14.200000](-12.000000,36.000000)--(-8.000000,40.000000);
\draw[black!1.700000](36.000000,36.000000)--(40.000000,40.000000);
\draw[black!2.700000](-3.000000,35.000000)--(1.000000,39.000000);
\draw[black!0.200000](22.000000,34.000000)--(26.000000,38.000000);
\draw[black!30.000000](2.000000,38.000000)--(6.000000,42.000000);
\draw[black!4.300000](-19.000000,35.000000)--(-15.000000,39.000000);
\draw[black!0.100000](30.000000,34.000000)--(34.000000,38.000000);
\draw[black!3.500000](-11.000000,35.000000)--(-7.000000,39.000000);
\draw[black!30.900000](33.000000,39.000000)--(37.000000,43.000000);
\draw[black!100.000000](44.000000,44.000000)--(48.000000,48.000000);
\draw[black!6.600000](-24.000000,40.000000)--(-20.000000,44.000000);
\draw[black!2.700000](-1.000000,41.000000)--(3.000000,45.000000);
\draw[black!6.400000](-27.000000,35.000000)--(-23.000000,39.000000);
\draw[black!16.400000](-20.000000,36.000000)--(-16.000000,40.000000);
\draw[black!9.400000](-39.000000,39.000000)--(-35.000000,43.000000);
\draw[black!0.600000](6.000000,34.000000)--(10.000000,38.000000);
\draw[black!0.200000](-10.000000,42.000000)--(-6.000000,46.000000);
\draw[black!19.000000](11.000000,37.000000)--(15.000000,41.000000);
\draw[black!24.100000](-5.000000,37.000000)--(-1.000000,41.000000);
\draw[black!6.300000](20.000000,36.000000)--(24.000000,40.000000);
\draw[black!11.300000](8.000000,40.000000)--(12.000000,44.000000);
\draw[black!7.500000](23.000000,41.000000)--(27.000000,45.000000);
\draw[black!8.800000](-16.000000,40.000000)--(-12.000000,44.000000);
\draw[black!22.600000](32.000000,40.000000)--(36.000000,44.000000);
\draw[black!23.600000](-36.000000,36.000000)--(-32.000000,40.000000);
\draw[black!25.800000](-29.000000,37.000000)--(-25.000000,41.000000);
\draw[black!30.000000](10.000000,38.000000)--(14.000000,42.000000);
\draw[black!0.100000](13.000000,43.000000)--(17.000000,47.000000);
\draw[black!26.700000](26.000000,38.000000)--(30.000000,42.000000);
\draw[black!18.700000](-15.000000,39.000000)--(-11.000000,43.000000);
\draw[black!0.600000](29.000000,43.000000)--(33.000000,47.000000);
\draw[black!0.400000](-33.000000,33.000000)--(-29.000000,37.000000);
\draw[black!2.200000](13.000000,35.000000)--(17.000000,39.000000);
\draw[black!21.600000](3.000000,37.000000)--(7.000000,41.000000);
\draw[black!0.700000](29.000000,35.000000)--(33.000000,39.000000);
\draw[black!1.000000](-10.000000,34.000000)--(-6.000000,38.000000);
\draw[black!22.500000](34.000000,38.000000)--(38.000000,42.000000);
\draw[black!1.400000](-25.000000,41.000000)--(-21.000000,45.000000);
\draw[black!10.500000](0.000000,40.000000)--(4.000000,44.000000);
\draw[black!2.200000](-26.000000,34.000000)--(-22.000000,38.000000);
\draw[black!23.800000](-21.000000,37.000000)--(-17.000000,41.000000);
\draw[black!0.200000](-9.000000,33.000000)--(-5.000000,37.000000);
\draw[black!0.900000](-33.000000,41.000000)--(-29.000000,45.000000);
\draw[black!27.800000](-6.000000,38.000000)--(-2.000000,42.000000);
\draw[black!0.100000](-26.000000,42.000000)--(-22.000000,46.000000);
\draw[black!27.700000](17.000000,39.000000)--(21.000000,43.000000);
\draw[black!3.500000](7.000000,41.000000)--(11.000000,45.000000);
\draw[black!29.800000](-22.000000,38.000000)--(-18.000000,42.000000);
\draw[black!1.400000](22.000000,42.000000)--(26.000000,46.000000);
\draw[black!29.500000](-37.000000,37.000000)--(-33.000000,41.000000);
\draw[black!8.500000](-8.000000,40.000000)--(-4.000000,44.000000);
\draw[black!100.000000](40.000000,40.000000)--(44.000000,44.000000);
\draw[black!18.300000](-28.000000,36.000000)--(-24.000000,40.000000);
\draw[black!1.300000](-40.000000,40.000000)--(-36.000000,44.000000);
\draw[black!25.400000](9.000000,39.000000)--(13.000000,43.000000);
\draw[black!1.200000](14.000000,42.000000)--(18.000000,46.000000);
\draw[black!0.100000](6.000000,42.000000)--(10.000000,46.000000);
\draw[black!27.600000](25.000000,39.000000)--(29.000000,43.000000);
\draw[black!26.600000](-14.000000,38.000000)--(-10.000000,42.000000);
\draw[black!1.900000](30.000000,42.000000)--(34.000000,46.000000);
\draw[black!2.400000](-34.000000,34.000000)--(-30.000000,38.000000);
\draw[black!13.700000](-31.000000,39.000000)--(-27.000000,43.000000);
\draw[black!8.800000](4.000000,36.000000)--(8.000000,40.000000);
\draw[black!5.200000](15.000000,41.000000)--(19.000000,45.000000);
\draw[black!0.100000](-17.000000,33.000000)--(-13.000000,37.000000);
\draw[black!3.800000](28.000000,36.000000)--(32.000000,40.000000);
\draw[black!25.500000](-13.000000,37.000000)--(-9.000000,41.000000);
\draw[black!9.100000](35.000000,37.000000)--(39.000000,41.000000);
\draw[black!0.200000](-2.000000,34.000000)--(2.000000,38.000000);
\draw[black!1.500000](21.000000,35.000000)--(25.000000,39.000000);
\draw[black!23.800000](1.000000,39.000000)--(5.000000,43.000000);
\draw[black!0.400000](-25.000000,33.000000)--(-21.000000,37.000000);
\draw[black!1.100000](-18.000000,34.000000)--(-14.000000,38.000000);
\draw[black!5.500000](-32.000000,40.000000)--(-28.000000,44.000000);
\draw[black!23.000000](-7.000000,39.000000)--(-3.000000,43.000000);
\draw[black!25.000000](18.000000,38.000000)--(22.000000,42.000000);
\draw[black!0.300000](-2.000000,42.000000)--(2.000000,46.000000);
\draw[black!16.400000](-23.000000,39.000000)--(-19.000000,43.000000);
\draw[black!0.400000](21.000000,43.000000)--(25.000000,47.000000);
\draw[black!24.300000](-38.000000,38.000000)--(-34.000000,42.000000);
\draw[black!1.700000](5.000000,35.000000)--(9.000000,39.000000);
\draw[black!1.400000](-9.000000,41.000000)--(-5.000000,45.000000);
\draw[black!8.500000](12.000000,36.000000)--(16.000000,40.000000);
\draw[black!12.100000](-4.000000,36.000000)--(0.000000,40.000000);
\draw[black!16.400000](19.000000,37.000000)--(23.000000,41.000000);
\draw[black!18.900000](24.000000,40.000000)--(28.000000,44.000000);
\draw[black!1.500000](-17.000000,41.000000)--(-13.000000,45.000000);
\draw[black!10.700000](31.000000,41.000000)--(35.000000,45.000000);
\draw[black!9.100000](-35.000000,35.000000)--(-31.000000,39.000000);
\draw[ultra thick,blue,dashed](-44.000000,44.000000) --(-40.000000,40.000000) -- (40.000000,40.000000) -- (44.000000,44.000000); \end{tikzpicture} 
 

%% file: figures/bok20alpha4-1000probANISO.tex
 
 \begin{tikzpicture}[xscale=0.030,yscale=0.12*sqrt(50)]
\clip(-80*1.2,80*0.9375) rectangle (80*1.2,80*1.0625);
\draw[red,thick](-84.000000,84.000000) -- (-80.000000,80.000000) -- (-79.000000,79.032000) -- (-78.000000,78.234000) -- (-77.000000,77.870000) -- (-76.000000,78.084000) -- (-75.000000,78.726000) -- (-74.000000,79.434000) -- (-73.000000,79.776000) -- (-72.000000,79.576000) -- (-71.000000,79.018000) -- (-70.000000,78.474000) -- (-69.000000,78.268000) -- (-68.000000,78.602000) -- (-67.000000,79.186000) -- (-66.000000,79.750000) -- (-65.000000,79.978000) -- (-64.000000,79.676000) -- (-63.000000,79.158000) -- (-62.000000,78.700000) -- (-61.000000,78.580000) -- (-60.000000,78.880000) -- (-59.000000,79.422000) -- (-58.000000,79.884000) -- (-57.000000,79.994000) -- (-56.000000,79.708000) -- (-55.000000,79.180000) -- (-54.000000,78.714000) -- (-53.000000,78.636000) -- (-52.000000,78.950000) -- (-51.000000,79.524000) -- (-50.000000,80.004000) -- (-49.000000,80.094000) -- (-48.000000,79.792000) -- (-47.000000,79.244000) -- (-46.000000,78.798000) -- (-45.000000,78.754000) -- (-44.000000,79.166000) -- (-43.000000,79.730000) -- (-42.000000,80.172000) -- (-41.000000,80.216000) -- (-40.000000,79.790000) -- (-39.000000,79.222000) -- (-38.000000,78.832000) -- (-37.000000,78.846000) -- (-36.000000,79.294000) -- (-35.000000,79.892000) -- (-34.000000,80.280000) -- (-33.000000,80.274000) -- (-32.000000,79.846000) -- (-31.000000,79.284000) -- (-30.000000,78.926000) -- (-29.000000,78.970000) -- (-28.000000,79.426000) -- (-27.000000,80.026000) -- (-26.000000,80.400000) -- (-25.000000,80.358000) -- (-24.000000,79.876000) -- (-23.000000,79.268000) -- (-22.000000,78.936000) -- (-21.000000,79.042000) -- (-20.000000,79.596000) -- (-19.000000,80.222000) -- (-18.000000,80.572000) -- (-17.000000,80.466000) -- (-16.000000,79.936000) -- (-15.000000,79.356000) -- (-14.000000,79.062000) -- (-13.000000,79.284000) -- (-12.000000,79.864000) -- (-11.000000,80.464000) -- (-10.000000,80.768000) -- (-9.000000,80.550000) -- (-8.000000,79.980000) -- (-7.000000,79.382000) -- (-6.000000,79.146000) -- (-5.000000,79.382000) -- (-4.000000,79.972000) -- (-3.000000,80.600000) -- (-2.000000,80.836000) -- (-1.000000,80.604000) -- (0.000000,80.034000) -- (1.000000,79.440000) -- (2.000000,79.176000) -- (3.000000,79.412000) -- (4.000000,79.990000) -- (5.000000,80.556000) -- (6.000000,80.808000) -- (7.000000,80.566000) -- (8.000000,80.004000) -- (9.000000,79.474000) -- (10.000000,79.292000) -- (11.000000,79.570000) -- (12.000000,80.146000) -- (13.000000,80.680000) -- (14.000000,80.874000) -- (15.000000,80.600000) -- (16.000000,80.024000) -- (17.000000,79.506000) -- (18.000000,79.332000) -- (19.000000,79.686000) -- (20.000000,80.322000) -- (21.000000,80.870000) -- (22.000000,81.050000) -- (23.000000,80.704000) -- (24.000000,80.092000) -- (25.000000,79.582000) -- (26.000000,79.498000) -- (27.000000,79.868000) -- (28.000000,80.482000) -- (29.000000,81.000000) -- (30.000000,81.094000) -- (31.000000,80.736000) -- (32.000000,80.160000) -- (33.000000,79.728000) -- (34.000000,79.712000) -- (35.000000,80.124000) -- (36.000000,80.752000) -- (37.000000,81.202000) -- (38.000000,81.224000) -- (39.000000,80.822000) -- (40.000000,80.204000) -- (41.000000,79.772000) -- (42.000000,79.754000) -- (43.000000,80.162000) -- (44.000000,80.746000) -- (45.000000,81.164000) -- (46.000000,81.170000) -- (47.000000,80.782000) -- (48.000000,80.252000) -- (49.000000,79.870000) -- (50.000000,79.924000) -- (51.000000,80.384000) -- (52.000000,80.968000) -- (53.000000,81.378000) -- (54.000000,81.332000) -- (55.000000,80.890000) -- (56.000000,80.314000) -- (57.000000,79.980000) -- (58.000000,80.108000) -- (59.000000,80.590000) -- (60.000000,81.202000) -- (61.000000,81.546000) -- (62.000000,81.424000) -- (63.000000,80.956000) -- (64.000000,80.374000) -- (65.000000,80.078000) -- (66.000000,80.296000) -- (67.000000,80.896000) -- (68.000000,81.536000) -- (69.000000,81.846000) -- (70.000000,81.642000) -- (71.000000,81.052000) -- (72.000000,80.488000) -- (73.000000,80.276000) -- (74.000000,80.586000) -- (75.000000,81.294000) -- (76.000000,81.896000) -- (77.000000,82.116000) -- (78.000000,81.782000) -- (79.000000,80.964000) -- (80.000000,80.000000) -- (84.000000,84.000000);
\draw[black!52.100000](54.000000,82.000000)--(53.000000,83.000000);
\draw[black!9.200000](-35.000000,75.000000)--(-36.000000,76.000000);
\draw[black!37.100000](3.000000,77.000000)--(2.000000,78.000000);
\draw[black!43.200000](58.000000,78.000000)--(57.000000,79.000000);
\draw[black!98.200000](80.000000,80.000000)--(79.000000,81.000000);
\draw[black!0.100000](23.000000,73.000000)--(22.000000,74.000000);
\draw[black!0.700000](42.000000,86.000000)--(41.000000,87.000000);
\draw[black!72.100000](55.000000,81.000000)--(54.000000,82.000000);
\draw[black!0.700000](69.000000,75.000000)--(68.000000,76.000000);
\draw[black!77.200000](-70.000000,78.000000)--(-71.000000,79.000000);
\draw[black!14.600000](-11.000000,83.000000)--(-12.000000,84.000000);
\draw[black!60.900000](-9.000000,81.000000)--(-10.000000,82.000000);
\draw[black!0.200000](-46.000000,86.000000)--(-47.000000,87.000000);
\draw[black!40.400000](22.000000,82.000000)--(21.000000,83.000000);
\draw[black!31.800000](-73.000000,81.000000)--(-74.000000,82.000000);
\draw[black!100.000000](-85.000000,85.000000)--(-86.000000,86.000000);
\draw[black!78.200000](72.000000,80.000000)--(71.000000,81.000000);
\draw[black!37.400000](-2.000000,82.000000)--(-3.000000,83.000000);
\draw[black!15.000000](67.000000,77.000000)--(66.000000,78.000000);
\draw[black!0.800000](-57.000000,73.000000)--(-58.000000,74.000000);
\draw[black!2.600000](37.000000,75.000000)--(36.000000,76.000000);
\draw[black!55.900000](62.000000,82.000000)--(61.000000,83.000000);
\draw[black!78.800000](32.000000,80.000000)--(31.000000,81.000000);
\draw[black!78.800000](56.000000,80.000000)--(55.000000,81.000000);
\draw[black!64.200000](-56.000000,80.000000)--(-57.000000,81.000000);
\draw[black!1.900000](35.000000,85.000000)--(34.000000,86.000000);
\draw[black!16.400000](5.000000,83.000000)--(4.000000,84.000000);
\draw[black!34.100000](-10.000000,82.000000)--(-11.000000,83.000000);
\draw[black!78.500000](-8.000000,80.000000)--(-9.000000,81.000000);
\draw[black!0.600000](22.000000,74.000000)--(21.000000,75.000000);
\draw[black!18.100000](-66.000000,82.000000)--(-67.000000,83.000000);
\draw[black!9.100000](-59.000000,83.000000)--(-60.000000,84.000000);
\draw[black!59.900000](-72.000000,80.000000)--(-73.000000,81.000000);
\draw[black!0.500000](-41.000000,73.000000)--(-42.000000,74.000000);
\draw[black!15.800000](4.000000,76.000000)--(3.000000,77.000000);
\draw[black!6.100000](20.000000,84.000000)--(19.000000,85.000000);
\draw[black!5.300000](5.000000,75.000000)--(4.000000,76.000000);
\draw[black!0.500000](-45.000000,85.000000)--(-46.000000,86.000000);
\draw[black!78.500000](0.000000,80.000000)--(-1.000000,81.000000);
\draw[black!0.200000](15.000000,73.000000)--(14.000000,74.000000);
\draw[black!0.100000](18.000000,86.000000)--(17.000000,87.000000);
\draw[black!1.100000](3.000000,85.000000)--(2.000000,86.000000);
\draw[black!0.700000](46.000000,74.000000)--(45.000000,75.000000);
\draw[black!25.100000](-42.000000,82.000000)--(-43.000000,83.000000);
\draw[black!0.400000](-49.000000,73.000000)--(-50.000000,74.000000);
\draw[black!0.300000](73.000000,87.000000)--(72.000000,88.000000);
\draw[black!31.700000](61.000000,83.000000)--(60.000000,84.000000);
\draw[black!1.100000](-73.000000,73.000000)--(-74.000000,74.000000);
\draw[black!69.900000](39.000000,81.000000)--(38.000000,82.000000);
\draw[black!7.300000](-27.000000,75.000000)--(-28.000000,76.000000);
\draw[black!71.300000](-32.000000,80.000000)--(-33.000000,81.000000);
\draw[black!7.300000](28.000000,84.000000)--(27.000000,85.000000);
\draw[black!3.600000](21.000000,75.000000)--(20.000000,76.000000);
\draw[black!14.300000](68.000000,84.000000)--(67.000000,85.000000);
\draw[black!9.400000](36.000000,76.000000)--(35.000000,77.000000);
\draw[black!78.800000](16.000000,80.000000)--(15.000000,81.000000);
\draw[black!0.800000](14.000000,74.000000)--(13.000000,75.000000);
\draw[black!0.400000](58.000000,86.000000)--(57.000000,87.000000);
\draw[black!10.100000](-74.000000,82.000000)--(-75.000000,83.000000);
\draw[black!10.700000](-43.000000,75.000000)--(-44.000000,76.000000);
\draw[black!79.100000](64.000000,80.000000)--(63.000000,81.000000);
\draw[black!65.100000](-48.000000,80.000000)--(-49.000000,81.000000);
\draw[black!1.400000](27.000000,85.000000)--(26.000000,86.000000);
\draw[black!0.100000](-72.000000,72.000000)--(-73.000000,73.000000);
\draw[black!5.000000](67.000000,85.000000)--(66.000000,86.000000);
\draw[black!69.000000](49.000000,79.000000)--(48.000000,80.000000);
\draw[black!9.700000](44.000000,84.000000)--(43.000000,85.000000);
\draw[black!1.600000](53.000000,75.000000)--(52.000000,76.000000);
\draw[black!13.800000](-59.000000,75.000000)--(-60.000000,76.000000);
\draw[black!76.500000](48.000000,80.000000)--(47.000000,81.000000);
\draw[black!65.000000](-64.000000,80.000000)--(-65.000000,81.000000);
\draw[black!5.100000](13.000000,75.000000)--(12.000000,76.000000);
\draw[black!50.200000](42.000000,78.000000)--(41.000000,79.000000);
\draw[black!76.500000](-16.000000,80.000000)--(-17.000000,81.000000);
\draw[black!1.300000](-18.000000,74.000000)--(-19.000000,75.000000);
\draw[black!66.500000](-22.000000,78.000000)--(-23.000000,79.000000);
\draw[black!100.000000](-80.000000,80.000000)--(-81.000000,81.000000);
\draw[black!0.200000](7.000000,73.000000)--(6.000000,74.000000);
\draw[black!23.200000](-28.000000,76.000000)--(-29.000000,77.000000);
\draw[black!0.300000](-25.000000,73.000000)--(-26.000000,74.000000);
\draw[black!12.100000](20.000000,76.000000)--(19.000000,77.000000);
\draw[black!3.200000](29.000000,75.000000)--(28.000000,76.000000);
\draw[black!0.100000](-1.000000,73.000000)--(-2.000000,74.000000);
\draw[black!25.800000](-44.000000,76.000000)--(-45.000000,77.000000);
\draw[black!0.100000](2.000000,86.000000)--(1.000000,87.000000);
\draw[black!2.800000](-50.000000,74.000000)--(-51.000000,75.000000);
\draw[black!23.200000](-50.000000,82.000000)--(-51.000000,83.000000);
\draw[black!2.200000](-26.000000,74.000000)--(-27.000000,75.000000);
\draw[black!32.500000](-68.000000,76.000000)--(-69.000000,77.000000);
\draw[black!23.600000](51.000000,77.000000)--(50.000000,78.000000);
\draw[black!100.000000](-86.000000,86.000000)--(-87.000000,87.000000);
\draw[black!49.000000](46.000000,82.000000)--(45.000000,83.000000);
\draw[black!79.000000](-15.000000,79.000000)--(-16.000000,80.000000);
\draw[black!38.300000](-13.000000,77.000000)--(-14.000000,78.000000);
\draw[black!2.600000](-20.000000,84.000000)--(-21.000000,85.000000);
\draw[black!31.200000](-18.000000,82.000000)--(-19.000000,83.000000);
\draw[black!68.200000](-77.000000,77.000000)--(-78.000000,78.000000);
\draw[black!2.800000](-42.000000,74.000000)--(-43.000000,75.000000);
\draw[black!100.000000](-82.000000,82.000000)--(-83.000000,83.000000);
\draw[black!21.600000](59.000000,77.000000)--(58.000000,78.000000);
\draw[black!24.700000](-36.000000,76.000000)--(-37.000000,77.000000);
\draw[black!0.900000](11.000000,85.000000)--(10.000000,86.000000);
\draw[black!0.100000](49.000000,87.000000)--(48.000000,88.000000);
\draw[black!0.700000](66.000000,86.000000)--(65.000000,87.000000);
\draw[black!0.800000](-2.000000,74.000000)--(-3.000000,75.000000);
\draw[black!64.800000](65.000000,79.000000)--(64.000000,80.000000);
\draw[black!12.800000](60.000000,84.000000)--(59.000000,85.000000);
\draw[black!9.100000](75.000000,77.000000)--(74.000000,78.000000);
\draw[black!67.900000](31.000000,81.000000)--(30.000000,82.000000);
\draw[black!9.200000](36.000000,84.000000)--(35.000000,85.000000);
\draw[black!79.700000](1.000000,79.000000)--(0.000000,80.000000);
\draw[black!4.300000](-4.000000,84.000000)--(-5.000000,85.000000);
\draw[black!54.000000](26.000000,78.000000)--(25.000000,79.000000);
\draw[black!0.200000](-62.000000,86.000000)--(-63.000000,87.000000);
\draw[black!75.900000](-63.000000,79.000000)--(-64.000000,80.000000);
\draw[black!0.800000](-68.000000,84.000000)--(-69.000000,85.000000);
\draw[black!1.100000](19.000000,85.000000)--(18.000000,86.000000);
\draw[black!5.300000](4.000000,84.000000)--(3.000000,85.000000);
\draw[black!43.800000](-21.000000,77.000000)--(-22.000000,78.000000);
\draw[black!98.400000](-79.000000,79.000000)--(-80.000000,80.000000);
\draw[black!15.000000](12.000000,76.000000)--(11.000000,77.000000);
\draw[black!100.000000](-84.000000,84.000000)--(-85.000000,85.000000);
\draw[black!61.500000](-1.000000,81.000000)--(-2.000000,82.000000);
\draw[black!3.800000](-12.000000,84.000000)--(-13.000000,85.000000);
\draw[black!5.500000](75.000000,85.000000)--(74.000000,86.000000);
\draw[black!3.300000](-58.000000,74.000000)--(-59.000000,75.000000);
\draw[black!11.500000](-51.000000,75.000000)--(-52.000000,76.000000);
\draw[black!80.900000](40.000000,80.000000)--(39.000000,81.000000);
\draw[black!0.100000](-64.000000,72.000000)--(-65.000000,73.000000);
\draw[black!3.400000](51.000000,85.000000)--(50.000000,86.000000);
\draw[black!89.900000](-78.000000,78.000000)--(-79.000000,79.000000);
\draw[black!27.500000](35.000000,77.000000)--(34.000000,78.000000);
\draw[black!0.600000](-37.000000,85.000000)--(-38.000000,86.000000);
\draw[black!0.300000](-53.000000,85.000000)--(-54.000000,86.000000);
\draw[black!45.000000](30.000000,82.000000)--(29.000000,83.000000);
\draw[black!0.600000](-13.000000,85.000000)--(-14.000000,86.000000);
\draw[black!4.000000](-28.000000,84.000000)--(-29.000000,85.000000);
\draw[black!15.000000](-67.000000,75.000000)--(-68.000000,76.000000);
\draw[black!12.000000](28.000000,76.000000)--(27.000000,77.000000);
\draw[black!49.900000](-33.000000,81.000000)--(-34.000000,82.000000);
\draw[black!2.900000](-34.000000,74.000000)--(-35.000000,75.000000);
\draw[black!6.600000](60.000000,76.000000)--(59.000000,77.000000);
\draw[black!69.500000](-38.000000,78.000000)--(-39.000000,79.000000);
\draw[black!19.000000](21.000000,83.000000)--(20.000000,84.000000);
\draw[black!3.600000](-44.000000,84.000000)--(-45.000000,85.000000);
\draw[black!67.200000](23.000000,81.000000)--(22.000000,82.000000);
\draw[black!0.500000](-65.000000,73.000000)--(-66.000000,74.000000);
\draw[black!73.400000](63.000000,81.000000)--(62.000000,82.000000);
\draw[black!0.200000](62.000000,74.000000)--(61.000000,75.000000);
\draw[black!78.100000](-31.000000,79.000000)--(-32.000000,80.000000);
\draw[black!90.900000](79.000000,81.000000)--(78.000000,82.000000);
\draw[black!27.000000](43.000000,77.000000)--(42.000000,78.000000);
\draw[black!12.400000](52.000000,84.000000)--(51.000000,85.000000);
\draw[black!2.400000](-60.000000,84.000000)--(-61.000000,85.000000);
\draw[black!33.300000](74.000000,78.000000)--(73.000000,79.000000);
\draw[black!33.800000](69.000000,83.000000)--(68.000000,84.000000);
\draw[black!38.100000](-65.000000,81.000000)--(-66.000000,82.000000);
\draw[black!3.700000](-66.000000,74.000000)--(-67.000000,75.000000);
\draw[black!50.400000](34.000000,78.000000)--(33.000000,79.000000);
\draw[black!5.400000](-11.000000,75.000000)--(-12.000000,76.000000);
\draw[black!74.100000](-24.000000,80.000000)--(-25.000000,81.000000);
\draw[black!24.900000](37.000000,83.000000)--(36.000000,84.000000);
\draw[black!67.900000](-30.000000,78.000000)--(-31.000000,79.000000);
\draw[black!80.600000](24.000000,80.000000)--(23.000000,81.000000);
\draw[black!100.000000](-87.000000,87.000000)--(-88.000000,88.000000);
\draw[black!27.700000](-34.000000,82.000000)--(-35.000000,83.000000);
\draw[black!12.700000](-27.000000,83.000000)--(-28.000000,84.000000);
\draw[black!0.300000](38.000000,74.000000)--(37.000000,75.000000);
\draw[black!71.300000](-40.000000,80.000000)--(-41.000000,81.000000);
\draw[black!0.700000](-10.000000,74.000000)--(-11.000000,75.000000);
\draw[black!53.600000](-53.000000,77.000000)--(-54.000000,78.000000);
\draw[black!0.400000](-33.000000,73.000000)--(-34.000000,74.000000);
\draw[black!63.100000](2.000000,78.000000)--(1.000000,79.000000);
\draw[black!66.700000](78.000000,82.000000)--(77.000000,83.000000);
\draw[black!0.500000](50.000000,86.000000)--(49.000000,87.000000);
\draw[black!46.900000](-29.000000,77.000000)--(-30.000000,78.000000);
\draw[black!72.700000](-62.000000,78.000000)--(-63.000000,79.000000);
\draw[black!0.100000](-14.000000,86.000000)--(-15.000000,87.000000);
\draw[black!0.900000](-29.000000,85.000000)--(-30.000000,86.000000);
\draw[black!16.300000](-75.000000,75.000000)--(-76.000000,76.000000);
\draw[black!35.200000](11.000000,77.000000)--(10.000000,78.000000);
\draw[black!46.800000](50.000000,78.000000)--(49.000000,79.000000);
\draw[black!61.900000](7.000000,81.000000)--(6.000000,82.000000);
\draw[black!26.000000](45.000000,83.000000)--(44.000000,84.000000);
\draw[black!55.500000](-61.000000,77.000000)--(-62.000000,78.000000);
\draw[black!1.800000](76.000000,76.000000)--(75.000000,77.000000);
\draw[black!0.200000](39.000000,73.000000)--(38.000000,74.000000);
\draw[black!0.100000](-56.000000,72.000000)--(-57.000000,73.000000);
\draw[black!30.100000](27.000000,77.000000)--(26.000000,78.000000);
\draw[black!4.300000](59.000000,85.000000)--(58.000000,86.000000);
\draw[black!0.100000](-22.000000,86.000000)--(-23.000000,87.000000);
\draw[black!0.200000](26.000000,86.000000)--(25.000000,87.000000);
\draw[black!4.500000](-74.000000,74.000000)--(-75.000000,75.000000);
\draw[black!0.800000](-5.000000,85.000000)--(-6.000000,86.000000);
\draw[black!0.200000](54.000000,74.000000)--(53.000000,75.000000);
\draw[black!58.600000](18.000000,78.000000)--(17.000000,79.000000);
\draw[black!3.100000](-52.000000,84.000000)--(-53.000000,85.000000);
\draw[black!18.200000](13.000000,83.000000)--(12.000000,84.000000);
\draw[black!63.500000](15.000000,81.000000)--(14.000000,82.000000);
\draw[black!3.700000](68.000000,76.000000)--(67.000000,77.000000);
\draw[black!59.100000](10.000000,78.000000)--(9.000000,79.000000);
\draw[black!37.400000](-5.000000,77.000000)--(-6.000000,78.000000);
\draw[black!45.100000](-49.000000,81.000000)--(-50.000000,82.000000);
\draw[black!79.500000](71.000000,81.000000)--(70.000000,82.000000);
\draw[black!71.300000](41.000000,79.000000)--(40.000000,80.000000);
\draw[black!38.400000](66.000000,78.000000)--(65.000000,79.000000);
\draw[black!73.300000](-54.000000,78.000000)--(-55.000000,79.000000);
\draw[black!16.200000](-4.000000,76.000000)--(-5.000000,77.000000);
\draw[black!2.600000](43.000000,85.000000)--(42.000000,86.000000);
\draw[black!8.400000](52.000000,76.000000)--(51.000000,77.000000);
\draw[black!1.600000](-75.000000,83.000000)--(-76.000000,84.000000);
\draw[black!32.600000](-60.000000,76.000000)--(-61.000000,77.000000);
\draw[black!69.400000](47.000000,81.000000)--(46.000000,82.000000);
\draw[black!1.300000](6.000000,74.000000)--(5.000000,75.000000);
\draw[black!64.600000](-14.000000,78.000000)--(-15.000000,79.000000);
\draw[black!17.200000](-12.000000,76.000000)--(-13.000000,77.000000);
\draw[black!12.300000](-19.000000,83.000000)--(-20.000000,84.000000);
\draw[black!55.300000](-17.000000,81.000000)--(-18.000000,82.000000);
\draw[black!51.800000](-25.000000,81.000000)--(-26.000000,82.000000);
\draw[black!20.900000](29.000000,83.000000)--(28.000000,84.000000);
\draw[black!39.300000](-76.000000,76.000000)--(-77.000000,77.000000);
\draw[black!48.700000](-37.000000,77.000000)--(-38.000000,78.000000);
\draw[black!100.000000](-81.000000,81.000000)--(-82.000000,82.000000);
\draw[black!1.200000](74.000000,86.000000)--(73.000000,87.000000);
\draw[black!39.000000](77.000000,83.000000)--(76.000000,84.000000);
\draw[black!76.500000](9.000000,79.000000)--(8.000000,80.000000);
\draw[black!10.900000](-35.000000,83.000000)--(-36.000000,84.000000);
\draw[black!75.900000](17.000000,79.000000)--(16.000000,80.000000);
\draw[black!31.200000](19.000000,77.000000)--(18.000000,78.000000);
\draw[black!6.200000](12.000000,84.000000)--(11.000000,85.000000);
\draw[black!0.500000](-61.000000,85.000000)--(-62.000000,86.000000);
\draw[black!39.500000](14.000000,82.000000)--(13.000000,83.000000);
\draw[black!11.100000](44.000000,76.000000)--(43.000000,77.000000);
\draw[black!72.100000](-46.000000,78.000000)--(-47.000000,79.000000);
\draw[black!51.700000](-45.000000,77.000000)--(-46.000000,78.000000);
\draw[black!43.700000](-57.000000,81.000000)--(-58.000000,82.000000);
\draw[black!9.800000](-51.000000,83.000000)--(-52.000000,84.000000);
\draw[black!3.100000](45.000000,75.000000)--(44.000000,76.000000);
\draw[black!19.700000](-20.000000,76.000000)--(-21.000000,77.000000);
\draw[black!6.400000](-19.000000,75.000000)--(-20.000000,76.000000);
\draw[black!0.400000](34.000000,86.000000)--(33.000000,87.000000);
\draw[black!14.700000](-3.000000,83.000000)--(-4.000000,84.000000);
\draw[black!11.100000](-43.000000,83.000000)--(-44.000000,84.000000);
\draw[black!0.300000](41.000000,87.000000)--(40.000000,88.000000);
\draw[black!60.300000](-69.000000,77.000000)--(-70.000000,78.000000);
\draw[black!80.400000](-23.000000,79.000000)--(-24.000000,80.000000);
\draw[black!5.800000](-67.000000,83.000000)--(-68.000000,84.000000);
\draw[black!1.100000](61.000000,75.000000)--(60.000000,76.000000);
\draw[black!48.600000](38.000000,82.000000)--(37.000000,83.000000);
\draw[black!0.900000](-21.000000,85.000000)--(-22.000000,86.000000);
\draw[black!75.500000](25.000000,79.000000)--(24.000000,80.000000);
\draw[black!71.600000](33.000000,79.000000)--(32.000000,80.000000);
\draw[black!29.100000](-26.000000,82.000000)--(-27.000000,83.000000);
\draw[black!60.300000](73.000000,79.000000)--(72.000000,80.000000);
\draw[black!78.400000](-39.000000,79.000000)--(-40.000000,80.000000);
\draw[black!100.000000](-83.000000,83.000000)--(-84.000000,84.000000);
\draw[black!31.200000](-52.000000,76.000000)--(-53.000000,77.000000);
\draw[black!18.100000](76.000000,84.000000)--(75.000000,85.000000);
\draw[black!2.900000](-36.000000,84.000000)--(-37.000000,85.000000);
\draw[black!61.800000](-6.000000,78.000000)--(-7.000000,79.000000);
\draw[black!27.900000](53.000000,83.000000)--(52.000000,84.000000);
\draw[black!60.200000](70.000000,82.000000)--(69.000000,83.000000);
\draw[black!0.300000](30.000000,74.000000)--(29.000000,75.000000);
\draw[black!66.700000](57.000000,79.000000)--(56.000000,80.000000);
\draw[black!76.400000](-55.000000,79.000000)--(-56.000000,80.000000);
\draw[black!77.400000](-47.000000,79.000000)--(-48.000000,80.000000);
\draw[black!79.900000](-7.000000,79.000000)--(-8.000000,80.000000);
\draw[black!23.600000](-58.000000,82.000000)--(-59.000000,83.000000);
\draw[black!77.900000](-71.000000,79.000000)--(-72.000000,80.000000);
\draw[black!47.300000](-41.000000,81.000000)--(-42.000000,82.000000);
\draw[black!36.100000](6.000000,82.000000)--(5.000000,83.000000);
\draw[black!0.100000](-32.000000,72.000000)--(-33.000000,73.000000);
\draw[black!3.900000](-3.000000,75.000000)--(-4.000000,76.000000);
\draw[black!0.100000](8.000000,72.000000)--(7.000000,73.000000);
\draw[black!78.000000](8.000000,80.000000)--(7.000000,81.000000);
\draw[black!0.300000](30.000000,74.000000)--(34.000000,78.000000);
\draw[black!0.200000](39.000000,73.000000)--(43.000000,77.000000);
\draw[black!11.800000](-12.000000,76.000000)--(-8.000000,80.000000);
\draw[black!16.400000](40.000000,80.000000)--(44.000000,84.000000);
\draw[black!20.900000](72.000000,80.000000)--(76.000000,84.000000);
\draw[black!21.200000](-5.000000,77.000000)--(-1.000000,81.000000);
\draw[black!23.000000](-76.000000,76.000000)--(-72.000000,80.000000);
\draw[black!24.500000](-37.000000,77.000000)--(-33.000000,81.000000);
\draw[black!21.600000](17.000000,79.000000)--(21.000000,83.000000);
\draw[black!10.500000](-59.000000,75.000000)--(-55.000000,79.000000);
\draw[black!0.500000](-18.000000,82.000000)--(-14.000000,86.000000);
\draw[black!23.500000](-30.000000,78.000000)--(-26.000000,82.000000);
\draw[black!1.200000](22.000000,82.000000)--(26.000000,86.000000);
\draw[black!0.300000](-50.000000,82.000000)--(-46.000000,86.000000);
\draw[black!0.400000](45.000000,83.000000)--(49.000000,87.000000);
\draw[black!2.800000](-57.000000,81.000000)--(-53.000000,85.000000);
\draw[black!0.900000](61.000000,75.000000)--(65.000000,79.000000);
\draw[black!2.400000](-50.000000,74.000000)--(-46.000000,78.000000);
\draw[black!0.500000](-41.000000,73.000000)--(-37.000000,77.000000);
\draw[black!27.700000](73.000000,79.000000)--(77.000000,83.000000);
\draw[black!9.300000](63.000000,81.000000)--(67.000000,85.000000);
\draw[black!1.000000](-73.000000,73.000000)--(-69.000000,77.000000);
\draw[black!8.000000](-40.000000,80.000000)--(-36.000000,84.000000);
\draw[black!8.500000](20.000000,76.000000)--(24.000000,80.000000);
\draw[black!4.000000](5.000000,75.000000)--(9.000000,79.000000);
\draw[black!18.900000](-60.000000,76.000000)--(-56.000000,80.000000);
\draw[black!0.300000](-33.000000,73.000000)--(-29.000000,77.000000);
\draw[black!5.100000](-27.000000,75.000000)--(-23.000000,79.000000);
\draw[black!1.000000](14.000000,82.000000)--(18.000000,86.000000);
\draw[black!24.600000](49.000000,79.000000)--(53.000000,83.000000);
\draw[black!19.800000](-52.000000,76.000000)--(-48.000000,80.000000);
\draw[black!16.000000](32.000000,80.000000)--(36.000000,84.000000);
\draw[black!3.000000](68.000000,76.000000)--(72.000000,80.000000);
\draw[black!1.600000](-80.000000,80.000000)--(-76.000000,84.000000);
\draw[black!0.400000](-49.000000,73.000000)--(-45.000000,77.000000);
\draw[black!8.500000](55.000000,81.000000)--(59.000000,85.000000);
\draw[black!0.700000](46.000000,74.000000)--(50.000000,78.000000);
\draw[black!22.800000](-46.000000,78.000000)--(-42.000000,82.000000);
\draw[black!4.200000](-1.000000,81.000000)--(3.000000,85.000000);
\draw[black!3.100000](-3.000000,75.000000)--(1.000000,79.000000);
\draw[black!2.900000](29.000000,75.000000)--(33.000000,79.000000);
\draw[black!12.000000](8.000000,80.000000)--(12.000000,84.000000);
\draw[black!10.500000](4.000000,76.000000)--(8.000000,80.000000);
\draw[black!1.900000](-65.000000,81.000000)--(-61.000000,85.000000);
\draw[black!2.400000](45.000000,75.000000)--(49.000000,79.000000);
\draw[black!0.100000](38.000000,74.000000)--(42.000000,78.000000);
\draw[black!18.100000](35.000000,77.000000)--(39.000000,81.000000);
\draw[black!8.000000](44.000000,76.000000)--(48.000000,80.000000);
\draw[black!0.300000](68.000000,84.000000)--(72.000000,88.000000);
\draw[black!3.500000](-9.000000,81.000000)--(-5.000000,85.000000);
\draw[black!5.500000](60.000000,76.000000)--(64.000000,80.000000);
\draw[black!15.500000](48.000000,80.000000)--(52.000000,84.000000);
\draw[black!3.000000](21.000000,75.000000)--(25.000000,79.000000);
\draw[black!1.000000](-2.000000,82.000000)--(2.000000,86.000000);
\draw[black!24.400000](26.000000,78.000000)--(30.000000,82.000000);
\draw[black!21.400000](9.000000,79.000000)--(13.000000,83.000000);
\draw[black!0.100000](-64.000000,72.000000)--(-60.000000,76.000000);
\draw[black!0.300000](-66.000000,82.000000)--(-62.000000,86.000000);
\draw[black!19.000000](-23.000000,79.000000)--(-19.000000,83.000000);
\draw[black!3.100000](-33.000000,81.000000)--(-29.000000,85.000000);
\draw[black!28.800000](-69.000000,77.000000)--(-65.000000,81.000000);
\draw[black!23.400000](41.000000,79.000000)--(45.000000,83.000000);
\draw[black!0.200000](15.000000,73.000000)--(19.000000,77.000000);
\draw[black!23.200000](34.000000,78.000000)--(38.000000,82.000000);
\draw[black!25.100000](-6.000000,78.000000)--(-2.000000,82.000000);
\draw[black!21.100000](-13.000000,77.000000)--(-9.000000,81.000000);
\draw[black!8.800000](28.000000,76.000000)--(32.000000,80.000000);
\draw[black!0.200000](54.000000,74.000000)--(58.000000,78.000000);
\draw[black!4.300000](13.000000,75.000000)--(17.000000,79.000000);
\draw[black!11.300000](67.000000,77.000000)--(71.000000,81.000000);
\draw[black!24.900000](57.000000,79.000000)--(61.000000,83.000000);
\draw[black!28.900000](-77.000000,77.000000)--(-73.000000,81.000000);
\draw[black!15.500000](-36.000000,76.000000)--(-32.000000,80.000000);
\draw[black!100.000000](80.000000,80.000000)--(84.000000,84.000000);
\draw[black!0.100000](-56.000000,72.000000)--(-52.000000,76.000000);
\draw[black!0.100000](-19.000000,83.000000)--(-15.000000,87.000000);
\draw[black!16.500000](-31.000000,79.000000)--(-27.000000,83.000000);
\draw[black!5.900000](23.000000,81.000000)--(27.000000,85.000000);
\draw[black!2.500000](-34.000000,74.000000)--(-30.000000,78.000000);
\draw[black!11.300000](-67.000000,75.000000)--(-63.000000,79.000000);
\draw[black!0.800000](-26.000000,82.000000)--(-22.000000,86.000000);
\draw[black!0.900000](-34.000000,82.000000)--(-30.000000,86.000000);
\draw[black!6.700000](-56.000000,80.000000)--(-52.000000,84.000000);
\draw[black!0.200000](-51.000000,83.000000)--(-47.000000,87.000000);
\draw[black!2.300000](37.000000,75.000000)--(41.000000,79.000000);
\draw[black!0.700000](-10.000000,74.000000)--(-6.000000,78.000000);
\draw[black!2.300000](-42.000000,74.000000)--(-38.000000,78.000000);
\draw[black!24.200000](74.000000,78.000000)--(78.000000,82.000000);
\draw[black!22.800000](-7.000000,79.000000)--(-3.000000,83.000000);
\draw[black!3.400000](-74.000000,74.000000)--(-70.000000,78.000000);
\draw[black!16.800000](-39.000000,79.000000)--(-35.000000,83.000000);
\draw[black!19.300000](19.000000,77.000000)--(23.000000,81.000000);
\draw[black!0.100000](8.000000,72.000000)--(12.000000,76.000000);
\draw[black!23.300000](-61.000000,77.000000)--(-57.000000,81.000000);
\draw[black!10.800000](-16.000000,80.000000)--(-12.000000,84.000000);
\draw[black!0.100000](13.000000,83.000000)--(17.000000,87.000000);
\draw[black!7.500000](-48.000000,80.000000)--(-44.000000,84.000000);
\draw[black!1.700000](-25.000000,81.000000)--(-21.000000,85.000000);
\draw[black!0.700000](61.000000,83.000000)--(65.000000,87.000000);
\draw[black!13.600000](-55.000000,79.000000)--(-51.000000,83.000000);
\draw[black!0.100000](-27.000000,83.000000)--(-23.000000,87.000000);
\draw[black!15.200000](51.000000,77.000000)--(55.000000,81.000000);
\draw[black!7.300000](31.000000,81.000000)--(35.000000,85.000000);
\draw[black!23.600000](66.000000,78.000000)--(70.000000,82.000000);
\draw[black!18.900000](56.000000,80.000000)--(60.000000,84.000000);
\draw[black!0.100000](44.000000,84.000000)--(48.000000,88.000000);
\draw[black!14.000000](-47.000000,79.000000)--(-43.000000,83.000000);
\draw[black!7.300000](75.000000,77.000000)--(79.000000,81.000000);
\draw[black!0.700000](-2.000000,74.000000)--(2.000000,78.000000);
\draw[black!0.600000](-42.000000,82.000000)--(-38.000000,86.000000);
\draw[black!21.400000](3.000000,77.000000)--(7.000000,81.000000);
\draw[black!20.400000](-62.000000,78.000000)--(-58.000000,82.000000);
\draw[black!5.100000](-19.000000,75.000000)--(-15.000000,79.000000);
\draw[black!12.900000](16.000000,80.000000)--(20.000000,84.000000);
\draw[black!0.800000](-73.000000,81.000000)--(-69.000000,85.000000);
\draw[black!1.500000](30.000000,82.000000)--(34.000000,86.000000);
\draw[black!23.900000](50.000000,78.000000)--(54.000000,82.000000);
\draw[black!10.400000](-8.000000,80.000000)--(-4.000000,84.000000);
\draw[black!0.400000](53.000000,83.000000)--(57.000000,87.000000);
\draw[black!9.000000](47.000000,81.000000)--(51.000000,85.000000);
\draw[black!15.100000](-44.000000,76.000000)--(-40.000000,80.000000);
\draw[black!0.100000](-3.000000,83.000000)--(1.000000,87.000000);
\draw[black!18.200000](27.000000,77.000000)--(31.000000,81.000000);
\draw[black!25.000000](10.000000,78.000000)--(14.000000,82.000000);
\draw[black!26.700000](2.000000,78.000000)--(6.000000,82.000000);
\draw[black!0.200000](-67.000000,83.000000)--(-63.000000,87.000000);
\draw[black!24.600000](-22.000000,78.000000)--(-18.000000,82.000000);
\draw[black!17.600000](-68.000000,76.000000)--(-64.000000,80.000000);
\draw[black!16.100000](43.000000,77.000000)--(47.000000,81.000000);
\draw[black!24.100000](33.000000,79.000000)--(37.000000,83.000000);
\draw[black!27.600000](-14.000000,78.000000)--(-10.000000,82.000000);
\draw[black!1.900000](38.000000,82.000000)--(42.000000,86.000000);
\draw[black!4.300000](70.000000,82.000000)--(74.000000,86.000000);
\draw[black!21.800000](58.000000,78.000000)--(62.000000,82.000000);
\draw[black!21.700000](-78.000000,78.000000)--(-74.000000,82.000000);
\draw[black!0.100000](23.000000,73.000000)--(27.000000,77.000000);
\draw[black!0.700000](-57.000000,73.000000)--(-53.000000,77.000000);
\draw[black!16.000000](-28.000000,76.000000)--(-24.000000,80.000000);
\draw[black!13.600000](24.000000,80.000000)--(28.000000,84.000000);
\draw[black!20.300000](11.000000,77.000000)--(15.000000,81.000000);
\draw[black!3.200000](-66.000000,74.000000)--(-62.000000,78.000000);
\draw[black!24.400000](-21.000000,77.000000)--(-17.000000,81.000000);
\draw[black!0.700000](69.000000,75.000000)--(73.000000,79.000000);
\draw[black!12.500000](-71.000000,79.000000)--(-67.000000,83.000000);
\draw[black!0.300000](36.000000,84.000000)--(40.000000,88.000000);
\draw[black!4.700000](-11.000000,75.000000)--(-7.000000,79.000000);
\draw[black!7.100000](39.000000,81.000000)--(43.000000,85.000000);
\draw[black!7.900000](-43.000000,75.000000)--(-39.000000,79.000000);
\draw[black!12.600000](71.000000,81.000000)--(75.000000,85.000000);
\draw[black!27.300000](65.000000,79.000000)--(69.000000,83.000000);
\draw[black!11.800000](-75.000000,75.000000)--(-71.000000,79.000000);
\draw[black!23.100000](-38.000000,78.000000)--(-34.000000,82.000000);
\draw[black!28.000000](18.000000,78.000000)--(22.000000,82.000000);
\draw[black!0.100000](7.000000,73.000000)--(11.000000,77.000000);
\draw[black!2.500000](-58.000000,74.000000)--(-54.000000,78.000000);
\draw[black!3.200000](-17.000000,81.000000)--(-13.000000,85.000000);
\draw[black!0.300000](-25.000000,73.000000)--(-21.000000,77.000000);
\draw[black!6.300000](-35.000000,75.000000)--(-31.000000,79.000000);
\draw[black!3.100000](-49.000000,81.000000)--(-45.000000,85.000000);
\draw[black!9.700000](-24.000000,80.000000)--(-20.000000,84.000000);
\draw[black!22.200000](-54.000000,78.000000)--(-50.000000,82.000000);
\draw[black!1.400000](53.000000,75.000000)--(57.000000,79.000000);
\draw[black!8.700000](-51.000000,75.000000)--(-47.000000,79.000000);
\draw[black!0.100000](-1.000000,73.000000)--(3.000000,77.000000);
\draw[black!1.800000](76.000000,76.000000)--(80.000000,80.000000);
\draw[black!19.800000](64.000000,80.000000)--(68.000000,84.000000);
\draw[black!0.100000](-72.000000,72.000000)--(-68.000000,76.000000);
\draw[black!2.300000](-41.000000,81.000000)--(-37.000000,85.000000);
\draw[black!0.900000](6.000000,82.000000)--(10.000000,86.000000);
\draw[black!1.100000](6.000000,74.000000)--(10.000000,78.000000);
\draw[black!14.500000](-63.000000,79.000000)--(-59.000000,83.000000);
\draw[black!1.300000](-18.000000,74.000000)--(-14.000000,78.000000);
\draw[black!1.900000](-26.000000,74.000000)--(-22.000000,78.000000);
\draw[black!5.000000](15.000000,81.000000)--(19.000000,85.000000);
\draw[black!0.600000](14.000000,74.000000)--(18.000000,78.000000);
\draw[black!5.000000](-72.000000,80.000000)--(-68.000000,84.000000);
\draw[black!23.100000](-53.000000,77.000000)--(-49.000000,81.000000);
\draw[black!0.400000](29.000000,83.000000)--(33.000000,87.000000);
\draw[black!23.300000](42.000000,78.000000)--(46.000000,82.000000);
\draw[black!3.900000](54.000000,82.000000)--(58.000000,86.000000);
\draw[black!2.900000](46.000000,82.000000)--(50.000000,86.000000);
\draw[black!26.300000](-45.000000,77.000000)--(-41.000000,81.000000);
\draw[black!11.100000](0.000000,80.000000)--(4.000000,84.000000);
\draw[black!0.200000](62.000000,74.000000)--(66.000000,78.000000);
\draw[black!100.000000](84.000000,84.000000)--(88.000000,88.000000);
\draw[black!5.300000](7.000000,81.000000)--(11.000000,85.000000);
\draw[black!19.700000](1.000000,79.000000)--(5.000000,83.000000);
\draw[black!6.700000](-64.000000,80.000000)--(-60.000000,84.000000);
\draw[black!0.400000](37.000000,83.000000)--(41.000000,87.000000);
\draw[black!6.800000](36.000000,76.000000)--(40.000000,80.000000);
\draw[black!0.200000](21.000000,83.000000)--(25.000000,87.000000);
\draw[black!4.300000](62.000000,82.000000)--(66.000000,86.000000);
\draw[black!19.500000](-15.000000,79.000000)--(-11.000000,83.000000);
\draw[black!6.800000](52.000000,76.000000)--(56.000000,80.000000);
\draw[black!0.800000](-10.000000,82.000000)--(-6.000000,86.000000);
\draw[black!15.000000](59.000000,77.000000)--(63.000000,81.000000);
\draw[black!8.500000](-79.000000,79.000000)--(-75.000000,83.000000);
\draw[black!0.500000](22.000000,74.000000)--(26.000000,78.000000);
\draw[black!0.100000](-32.000000,72.000000)--(-28.000000,76.000000);
\draw[black!24.000000](-29.000000,77.000000)--(-25.000000,81.000000);
\draw[black!24.500000](25.000000,79.000000)--(29.000000,83.000000);
\draw[black!10.000000](12.000000,76.000000)--(16.000000,80.000000);
\draw[black!0.400000](-65.000000,73.000000)--(-61.000000,77.000000);
\draw[black!13.300000](-20.000000,76.000000)--(-16.000000,80.000000);
\draw[black!8.700000](-32.000000,80.000000)--(-28.000000,84.000000);
\draw[black!0.300000](-58.000000,82.000000)--(-54.000000,86.000000);
\draw[black!0.900000](69.000000,83.000000)--(73.000000,87.000000);
\draw[black!20.300000](-70.000000,78.000000)--(-66.000000,82.000000);
\draw[black!12.300000](-4.000000,76.000000)--(0.000000,80.000000);
\draw[ultra thick,blue,dashed](-84.000000,84.000000) --(-80.000000,80.000000) -- (80.000000,80.000000) -- (84.000000,84.000000); \end{tikzpicture} 
 

%% file: figures/single-walk-bok80.tex
 
 \begin{tikzpicture}[xscale=0.0075*10,yscale=0.12*sqrt(50)*10]
\draw[](-32.400000,32.400000) -- (-32.000000,32.000000) -- (-31.900000,31.900000) -- (-31.800000,31.800000) -- (-31.700000,31.700000) -- (-31.600000,31.600000) -- (-31.500000,31.700000) -- (-31.400000,31.800000) -- (-31.300000,31.900000) -- (-31.200000,32.000000) -- (-31.100000,31.900000) -- (-31.000000,31.800000) -- (-30.900000,31.900000) -- (-30.800000,32.000000) -- (-30.700000,32.100000) -- (-30.600000,32.200000) -- (-30.500000,32.100000) -- (-30.400000,32.000000) -- (-30.300000,31.900000) -- (-30.200000,31.800000) -- (-30.100000,31.700000) -- (-30.000000,31.800000) -- (-29.900000,31.900000) -- (-29.800000,32.000000) -- (-29.700000,32.100000) -- (-29.600000,32.000000) -- (-29.500000,31.900000) -- (-29.400000,31.800000) -- (-29.300000,31.700000) -- (-29.200000,31.600000) -- (-29.100000,31.500000) -- (-29.000000,31.600000) -- (-28.900000,31.700000) -- (-28.800000,31.800000) -- (-28.700000,31.900000) -- (-28.600000,31.800000) -- (-28.500000,31.700000) -- (-28.400000,31.600000) -- (-28.300000,31.500000) -- (-28.200000,31.600000) -- (-28.100000,31.700000) -- (-28.000000,31.800000) -- (-27.900000,31.900000) -- (-27.800000,31.800000) -- (-27.700000,31.700000) -- (-27.600000,31.600000) -- (-27.500000,31.700000) -- (-27.400000,31.800000) -- (-27.300000,31.900000) -- (-27.200000,32.000000) -- (-27.100000,31.900000) -- (-27.000000,31.800000) -- (-26.900000,31.700000) -- (-26.800000,31.600000) -- (-26.700000,31.700000) -- (-26.600000,31.800000) -- (-26.500000,31.900000) -- (-26.400000,32.000000) -- (-26.300000,31.900000) -- (-26.200000,32.000000) -- (-26.100000,32.100000) -- (-26.000000,32.200000) -- (-25.900000,32.300000) -- (-25.800000,32.200000) -- (-25.700000,32.100000) -- (-25.600000,32.000000) -- (-25.500000,31.900000) -- (-25.400000,32.000000) -- (-25.300000,32.100000) -- (-25.200000,32.200000) -- (-25.100000,32.300000) -- (-25.000000,32.200000) -- (-24.900000,32.100000) -- (-24.800000,32.000000) -- (-24.700000,31.900000) -- (-24.600000,31.800000) -- (-24.500000,31.900000) -- (-24.400000,32.000000) -- (-24.300000,32.100000) -- (-24.200000,32.200000) -- (-24.100000,32.100000) -- (-24.000000,32.000000) -- (-23.900000,31.900000) -- (-23.800000,31.800000) -- (-23.700000,31.700000) -- (-23.600000,31.600000) -- (-23.500000,31.500000) -- (-23.400000,31.600000) -- (-23.300000,31.700000) -- (-23.200000,31.800000) -- (-23.100000,31.900000) -- (-23.000000,31.800000) -- (-22.900000,31.700000) -- (-22.800000,31.600000) -- (-22.700000,31.500000) -- (-22.600000,31.400000) -- (-22.500000,31.500000) -- (-22.400000,31.600000) -- (-22.300000,31.700000) -- (-22.200000,31.800000) -- (-22.100000,31.700000) -- (-22.000000,31.600000) -- (-21.900000,31.500000) -- (-21.800000,31.600000) -- (-21.700000,31.700000) -- (-21.600000,31.800000) -- (-21.500000,31.900000) -- (-21.400000,31.800000) -- (-21.300000,31.700000) -- (-21.200000,31.800000) -- (-21.100000,31.900000) -- (-21.000000,32.000000) -- (-20.900000,32.100000) -- (-20.800000,32.000000) -- (-20.700000,31.900000) -- (-20.600000,31.800000) -- (-20.500000,31.700000) -- (-20.400000,31.800000) -- (-20.300000,31.900000) -- (-20.200000,32.000000) -- (-20.100000,32.100000) -- (-20.000000,32.000000) -- (-19.900000,31.900000) -- (-19.800000,31.800000) -- (-19.700000,31.700000) -- (-19.600000,31.600000) -- (-19.500000,31.700000) -- (-19.400000,31.800000) -- (-19.300000,31.900000) -- (-19.200000,32.000000) -- (-19.100000,31.900000) -- (-19.000000,31.800000) -- (-18.900000,31.900000) -- (-18.800000,32.000000) -- (-18.700000,32.100000) -- (-18.600000,32.200000) -- (-18.500000,32.100000) -- (-18.400000,32.000000) -- (-18.300000,31.900000) -- (-18.200000,32.000000) -- (-18.100000,32.100000) -- (-18.000000,32.200000) -- (-17.900000,32.300000) -- (-17.800000,32.200000) -- (-17.700000,32.100000) -- (-17.600000,32.200000) -- (-17.500000,32.300000) -- (-17.400000,32.400000) -- (-17.300000,32.500000) -- (-17.200000,32.400000) -- (-17.100000,32.300000) -- (-17.000000,32.200000) -- (-16.900000,32.100000) -- (-16.800000,32.000000) -- (-16.700000,31.900000) -- (-16.600000,31.800000) -- (-16.500000,31.900000) -- (-16.400000,32.000000) -- (-16.300000,32.100000) -- (-16.200000,32.200000) -- (-16.100000,32.100000) -- (-16.000000,32.200000) -- (-15.900000,32.300000) -- (-15.800000,32.400000) -- (-15.700000,32.500000) -- (-15.600000,32.400000) -- (-15.500000,32.300000) -- (-15.400000,32.200000) -- (-15.300000,32.100000) -- (-15.200000,32.000000) -- (-15.100000,31.900000) -- (-15.000000,31.800000) -- (-14.900000,31.700000) -- (-14.800000,31.800000) -- (-14.700000,31.900000) -- (-14.600000,32.000000) -- (-14.500000,32.100000) -- (-14.400000,32.000000) -- (-14.300000,31.900000) -- (-14.200000,31.800000) -- (-14.100000,31.700000) -- (-14.000000,31.600000) -- (-13.900000,31.500000) -- (-13.800000,31.600000) -- (-13.700000,31.700000) -- (-13.600000,31.800000) -- (-13.500000,31.900000) -- (-13.400000,31.800000) -- (-13.300000,31.700000) -- (-13.200000,31.600000) -- (-13.100000,31.500000) -- (-13.000000,31.400000) -- (-12.900000,31.500000) -- (-12.800000,31.600000) -- (-12.700000,31.700000) -- (-12.600000,31.800000) -- (-12.500000,31.700000) -- (-12.400000,31.800000) -- (-12.300000,31.900000) -- (-12.200000,32.000000) -- (-12.100000,32.100000) -- (-12.000000,32.000000) -- (-11.900000,31.900000) -- (-11.800000,31.800000) -- (-11.700000,31.900000) -- (-11.600000,32.000000) -- (-11.500000,32.100000) -- (-11.400000,32.200000) -- (-11.300000,32.100000) -- (-11.200000,32.000000) -- (-11.100000,31.900000) -- (-11.000000,31.800000) -- (-10.900000,31.900000) -- (-10.800000,32.000000) -- (-10.700000,32.100000) -- (-10.600000,32.200000) -- (-10.500000,32.100000) -- (-10.400000,32.000000) -- (-10.300000,31.900000) -- (-10.200000,31.800000) -- (-10.100000,31.700000) -- (-10.000000,31.800000) -- (-9.900000,31.900000) -- (-9.800000,32.000000) -- (-9.700000,32.100000) -- (-9.600000,32.000000) -- (-9.500000,31.900000) -- (-9.400000,31.800000) -- (-9.300000,31.700000) -- (-9.200000,31.800000) -- (-9.100000,31.900000) -- (-9.000000,32.000000) -- (-8.900000,32.100000) -- (-8.800000,32.000000) -- (-8.700000,31.900000) -- (-8.600000,32.000000) -- (-8.500000,32.100000) -- (-8.400000,32.200000) -- (-8.300000,32.300000) -- (-8.200000,32.200000) -- (-8.100000,32.100000) -- (-8.000000,32.000000) -- (-7.900000,31.900000) -- (-7.800000,32.000000) -- (-7.700000,32.100000) -- (-7.600000,32.200000) -- (-7.500000,32.300000) -- (-7.400000,32.200000) -- (-7.300000,32.100000) -- (-7.200000,32.000000) -- (-7.100000,31.900000) -- (-7.000000,31.800000) -- (-6.900000,31.700000) -- (-6.800000,31.800000) -- (-6.700000,31.900000) -- (-6.600000,32.000000) -- (-6.500000,32.100000) -- (-6.400000,32.000000) -- (-6.300000,31.900000) -- (-6.200000,31.800000) -- (-6.100000,31.900000) -- (-6.000000,32.000000) -- (-5.900000,32.100000) -- (-5.800000,32.200000) -- (-5.700000,32.100000) -- (-5.600000,32.000000) -- (-5.500000,31.900000) -- (-5.400000,31.800000) -- (-5.300000,31.700000) -- (-5.200000,31.800000) -- (-5.100000,31.900000) -- (-5.000000,32.000000) -- (-4.900000,32.100000) -- (-4.800000,32.000000) -- (-4.700000,31.900000) -- (-4.600000,31.800000) -- (-4.500000,31.900000) -- (-4.400000,32.000000) -- (-4.300000,32.100000) -- (-4.200000,32.200000) -- (-4.100000,32.100000) -- (-4.000000,32.000000) -- (-3.900000,31.900000) -- (-3.800000,31.800000) -- (-3.700000,31.700000) -- (-3.600000,31.600000) -- (-3.500000,31.700000) -- (-3.400000,31.800000) -- (-3.300000,31.900000) -- (-3.200000,32.000000) -- (-3.100000,31.900000) -- (-3.000000,31.800000) -- (-2.900000,31.700000) -- (-2.800000,31.600000) -- (-2.700000,31.500000) -- (-2.600000,31.400000) -- (-2.500000,31.500000) -- (-2.400000,31.600000) -- (-2.300000,31.700000) -- (-2.200000,31.800000) -- (-2.100000,31.900000) -- (-2.000000,32.000000) -- (-1.900000,32.100000) -- (-1.800000,32.200000) -- (-1.700000,32.100000) -- (-1.600000,32.000000) -- (-1.500000,31.900000) -- (-1.400000,32.000000) -- (-1.300000,32.100000) -- (-1.200000,32.200000) -- (-1.100000,32.300000) -- (-1.000000,32.200000) -- (-0.900000,32.100000) -- (-0.800000,32.000000) -- (-0.700000,31.900000) -- (-0.600000,31.800000) -- (-0.500000,31.900000) -- (-0.400000,32.000000) -- (-0.300000,32.100000) -- (-0.200000,32.200000) -- (-0.100000,32.100000) -- (0.000000,32.000000) -- (0.100000,31.900000) -- (0.200000,32.000000) -- (0.300000,32.100000) -- (0.400000,32.200000) -- (0.500000,32.300000) -- (0.600000,32.200000) -- (0.700000,32.100000) -- (0.800000,32.000000) -- (0.900000,31.900000) -- (1.000000,31.800000) -- (1.100000,31.700000) -- (1.200000,31.800000) -- (1.300000,31.900000) -- (1.400000,32.000000) -- (1.500000,32.100000) -- (1.600000,32.000000) -- (1.700000,31.900000) -- (1.800000,32.000000) -- (1.900000,32.100000) -- (2.000000,32.200000) -- (2.100000,32.300000) -- (2.200000,32.200000) -- (2.300000,32.100000) -- (2.400000,32.000000) -- (2.500000,31.900000) -- (2.600000,31.800000) -- (2.700000,31.900000) -- (2.800000,32.000000) -- (2.900000,32.100000) -- (3.000000,32.200000) -- (3.100000,32.100000) -- (3.200000,32.000000) -- (3.300000,31.900000) -- (3.400000,32.000000) -- (3.500000,32.100000) -- (3.600000,32.200000) -- (3.700000,32.300000) -- (3.800000,32.200000) -- (3.900000,32.100000) -- (4.000000,32.000000) -- (4.100000,31.900000) -- (4.200000,32.000000) -- (4.300000,32.100000) -- (4.400000,32.200000) -- (4.500000,32.300000) -- (4.600000,32.200000) -- (4.700000,32.100000) -- (4.800000,32.000000) -- (4.900000,31.900000) -- (5.000000,32.000000) -- (5.100000,32.100000) -- (5.200000,32.200000) -- (5.300000,32.300000) -- (5.400000,32.200000) -- (5.500000,32.100000) -- (5.600000,32.000000) -- (5.700000,31.900000) -- (5.800000,31.800000) -- (5.900000,31.900000) -- (6.000000,32.000000) -- (6.100000,32.100000) -- (6.200000,32.200000) -- (6.300000,32.100000) -- (6.400000,32.000000) -- (6.500000,31.900000) -- (6.600000,32.000000) -- (6.700000,32.100000) -- (6.800000,32.200000) -- (6.900000,32.300000) -- (7.000000,32.200000) -- (7.100000,32.100000) -- (7.200000,32.000000) -- (7.300000,31.900000) -- (7.400000,32.000000) -- (7.500000,32.100000) -- (7.600000,32.200000) -- (7.700000,32.300000) -- (7.800000,32.200000) -- (7.900000,32.100000) -- (8.000000,32.000000) -- (8.100000,31.900000) -- (8.200000,32.000000) -- (8.300000,32.100000) -- (8.400000,32.200000) -- (8.500000,32.300000) -- (8.600000,32.200000) -- (8.700000,32.100000) -- (8.800000,32.000000) -- (8.900000,31.900000) -- (9.000000,32.000000) -- (9.100000,32.100000) -- (9.200000,32.200000) -- (9.300000,32.300000) -- (9.400000,32.200000) -- (9.500000,32.100000) -- (9.600000,32.000000) -- (9.700000,32.100000) -- (9.800000,32.200000) -- (9.900000,32.300000) -- (10.000000,32.400000) -- (10.100000,32.300000) -- (10.200000,32.200000) -- (10.300000,32.100000) -- (10.400000,32.000000) -- (10.500000,31.900000) -- (10.600000,31.800000) -- (10.700000,31.700000) -- (10.800000,31.600000) -- (10.900000,31.700000) -- (11.000000,31.800000) -- (11.100000,31.900000) -- (11.200000,32.000000) -- (11.300000,31.900000) -- (11.400000,32.000000) -- (11.500000,32.100000) -- (11.600000,32.200000) -- (11.700000,32.300000) -- (11.800000,32.200000) -- (11.900000,32.100000) -- (12.000000,32.000000) -- (12.100000,31.900000) -- (12.200000,32.000000) -- (12.300000,32.100000) -- (12.400000,32.200000) -- (12.500000,32.300000) -- (12.600000,32.200000) -- (12.700000,32.100000) -- (12.800000,32.000000) -- (12.900000,32.100000) -- (13.000000,32.200000) -- (13.100000,32.300000) -- (13.200000,32.400000) -- (13.300000,32.300000) -- (13.400000,32.200000) -- (13.500000,32.100000) -- (13.600000,32.000000) -- (13.700000,31.900000) -- (13.800000,31.800000) -- (13.900000,31.900000) -- (14.000000,32.000000) -- (14.100000,32.100000) -- (14.200000,32.200000) -- (14.300000,32.100000) -- (14.400000,32.000000) -- (14.500000,31.900000) -- (14.600000,32.000000) -- (14.700000,32.100000) -- (14.800000,32.200000) -- (14.900000,32.300000) -- (15.000000,32.200000) -- (15.100000,32.100000) -- (15.200000,32.000000) -- (15.300000,31.900000) -- (15.400000,32.000000) -- (15.500000,32.100000) -- (15.600000,32.200000) -- (15.700000,32.300000) -- (15.800000,32.200000) -- (15.900000,32.100000) -- (16.000000,32.000000) -- (16.100000,32.100000) -- (16.200000,32.200000) -- (16.300000,32.300000) -- (16.400000,32.400000) -- (16.500000,32.300000) -- (16.600000,32.200000) -- (16.700000,32.100000) -- (16.800000,32.200000) -- (16.900000,32.300000) -- (17.000000,32.400000) -- (17.100000,32.500000) -- (17.200000,32.400000) -- (17.300000,32.300000) -- (17.400000,32.200000) -- (17.500000,32.100000) -- (17.600000,32.000000) -- (17.700000,31.900000) -- (17.800000,32.000000) -- (17.900000,32.100000) -- (18.000000,32.200000) -- (18.100000,32.300000) -- (18.200000,32.200000) -- (18.300000,32.100000) -- (18.400000,32.000000) -- (18.500000,32.100000) -- (18.600000,32.200000) -- (18.700000,32.300000) -- (18.800000,32.400000) -- (18.900000,32.300000) -- (19.000000,32.200000) -- (19.100000,32.100000) -- (19.200000,32.000000) -- (19.300000,31.900000) -- (19.400000,32.000000) -- (19.500000,32.100000) -- (19.600000,32.200000) -- (19.700000,32.300000) -- (19.800000,32.200000) -- (19.900000,32.100000) -- (20.000000,32.200000) -- (20.100000,32.300000) -- (20.200000,32.400000) -- (20.300000,32.500000) -- (20.400000,32.400000) -- (20.500000,32.300000) -- (20.600000,32.200000) -- (20.700000,32.100000) -- (20.800000,32.000000) -- (20.900000,31.900000) -- (21.000000,31.800000) -- (21.100000,31.700000) -- (21.200000,31.800000) -- (21.300000,31.900000) -- (21.400000,32.000000) -- (21.500000,32.100000) -- (21.600000,32.000000) -- (21.700000,31.900000) -- (21.800000,32.000000) -- (21.900000,32.100000) -- (22.000000,32.200000) -- (22.100000,32.300000) -- (22.200000,32.200000) -- (22.300000,32.100000) -- (22.400000,32.000000) -- (22.500000,32.100000) -- (22.600000,32.200000) -- (22.700000,32.300000) -- (22.800000,32.400000) -- (22.900000,32.300000) -- (23.000000,32.200000) -- (23.100000,32.100000) -- (23.200000,32.000000) -- (23.300000,31.900000) -- (23.400000,32.000000) -- (23.500000,32.100000) -- (23.600000,32.200000) -- (23.700000,32.300000) -- (23.800000,32.200000) -- (23.900000,32.100000) -- (24.000000,32.000000) -- (24.100000,31.900000) -- (24.200000,31.800000) -- (24.300000,31.900000) -- (24.400000,32.000000) -- (24.500000,32.100000) -- (24.600000,32.200000) -- (24.700000,32.100000) -- (24.800000,32.000000) -- (24.900000,31.900000) -- (25.000000,32.000000) -- (25.100000,32.100000) -- (25.200000,32.200000) -- (25.300000,32.300000) -- (25.400000,32.200000) -- (25.500000,32.100000) -- (25.600000,32.000000) -- (25.700000,31.900000) -- (25.800000,31.800000) -- (25.900000,31.700000) -- (26.000000,31.800000) -- (26.100000,31.900000) -- (26.200000,32.000000) -- (26.300000,32.100000) -- (26.400000,32.000000) -- (26.500000,31.900000) -- (26.600000,31.800000) -- (26.700000,31.900000) -- (26.800000,32.000000) -- (26.900000,32.100000) -- (27.000000,32.200000) -- (27.100000,32.100000) -- (27.200000,32.000000) -- (27.300000,31.900000) -- (27.400000,31.800000) -- (27.500000,31.900000) -- (27.600000,32.000000) -- (27.700000,32.100000) -- (27.800000,32.200000) -- (27.900000,32.100000) -- (28.000000,32.000000) -- (28.100000,31.900000) -- (28.200000,31.800000) -- (28.300000,31.900000) -- (28.400000,32.000000) -- (28.500000,32.100000) -- (28.600000,32.200000) -- (28.700000,32.100000) -- (28.800000,32.000000) -- (28.900000,31.900000) -- (29.000000,32.000000) -- (29.100000,32.100000) -- (29.200000,32.200000) -- (29.300000,32.300000) -- (29.400000,32.200000) -- (29.500000,32.100000) -- (29.600000,32.000000) -- (29.700000,31.900000) -- (29.800000,31.800000) -- (29.900000,31.900000) -- (30.000000,32.000000) -- (30.100000,32.100000) -- (30.200000,32.200000) -- (30.300000,32.100000) -- (30.400000,32.000000) -- (30.500000,31.900000) -- (30.600000,32.000000) -- (30.700000,32.100000) -- (30.800000,32.200000) -- (30.900000,32.300000) -- (31.000000,32.200000) -- (31.100000,32.100000) -- (31.200000,32.200000) -- (31.300000,32.300000) -- (31.400000,32.400000) -- (31.500000,32.500000) -- (31.600000,32.400000) -- (31.700000,32.300000) -- (31.800000,32.200000) -- (31.900000,32.100000) -- (32.000000,32.000000) -- (32.400000,32.400000);
\draw[ultra thick,blue,dashed](-32.400000,32.400000) --(-32.000000,32.000000) -- (32.000000,32.000000) -- (32.400000,32.400000); \end{tikzpicture} 
 